\documentclass[12pt]{article}
\usepackage{amsthm,amsfonts,amsmath,amssymb,latexsym,epsfig}
\usepackage[all]{xy}
\usepackage{makeidx}
\newif\ifprintable
\printablefalse
\ifprintable
\newcommand{\hyperlink}[2]{{#2}}
\newcommand{\hypertarget}[2]{{#2}}
\newcommand{\phantomsection}{{}}
\newcommand{\url}[1]{{#1}}

\else
\usepackage[colorlinks=true]{hyperref}
\fi
\makeindex
\newif\ifmarking
\markingtrue

\title{Combinatorial Floer Homology}

\author{
Vin de Silva \\
Pomona College
\and
Joel~W.~Robbin\\
University of Wisconsin\thanks{Part of this paper was written
while JWR visited the FIM at ETH Z\"urich}
\and
Dietmar~A.~Salamon\thanks{Partially supported by the 
Swiss National Science Foundation Grant 200021-127136}\\
ETH-Z\"urich
}

\date{6 February 2013}

\newcommand{\MAT}[1]{\left[\begin{array}{#1}}
\newcommand{\RIX}{\end{array}\right]}
\newcommand{\p}{\partial}
\newcommand{\alg}[2]{{\mathrm{alg}\left(#1,#2\right)}}
\newcommand{\geo}[2]{{\mathrm{geo}\left(#1,#2\right)}}
\newcommand{\num}[2]{{\mathrm{num}\left(#1,#2\right)}}
\newcommand{\one}{{{\mathchoice {\rm 1\mskip-4mu l} {\rm 1\mskip-4mu l}
{\rm 1\mskip-4.5mu l} {\rm 1\mskip-5mu l}}}}

\newcommand{\C}{{\mathbb C}}
\newcommand{\D}{{\mathbb D}}

\newcommand{\HH}{{\mathbb H}}
\newcommand{\K}{{\mathbb K}}
\newcommand{\N}{{\mathbb N}}

\newcommand{\R}{{\mathbb R}}
\renewcommand{\S}{{\mathbb S}}
\newcommand{\T}{{\mathbb T}}
\newcommand{\Z}{{\mathbb Z}}
\newcommand{\cA}{{\mathcal A}}   

\newcommand{\cD}{{\mathcal D}}

\newcommand{\cH}{{\mathcal H}}

\newcommand{\cM}{{\mathcal M}}   

\newcommand{\cP}{{\mathcal P}}

\newcommand{\cW}{{\mathcal W}}

\newcommand{\Om}{{\Omega}}
\newcommand{\om}{{\omega}}
\newcommand{\eps}{{\varepsilon}}
\renewcommand{\phi}{{\varphi}}

\newcommand{\tSi}{{\widetilde{\Sigma}}}
\newcommand{\talpha}{{\widetilde{\alpha}}}
\newcommand{\tbeta}{{\widetilde{\beta}}}
\newcommand{\tga}{{\widetilde{\gamma}}}

\newcommand{\tLa}{{\widetilde{\Lambda}}}
\newcommand{\tA}{{\widetilde{A}}}
\newcommand{\tB}{{\widetilde{B}}}

\newcommand{\tU}{{\widetilde{U}}}
\newcommand{\tf}{{\widetilde{f}}}
\renewcommand{\th}{{\widetilde{h}}}
\newcommand{\tu}{{\widetilde{u}}}
\newcommand{\tv}{{\widetilde{v}}}

\newcommand{\tx}{{\widetilde{x}}}
\newcommand{\ty}{{\widetilde{y}}}
\newcommand{\tz}{{\widetilde{z}}}

\newcommand{\im}{{\rm im }}        
\newcommand{\id}{{\rm id}}         

\newcommand{\INT}{{\rm int}}       

\newcommand{\IM}{{\rm Im }}        
\newcommand{\RE}{{\rm Re}}         
\renewcommand{\Re}{{\rm Re}}       
\newcommand{\Diff}{{\rm Diff}}        
\newcommand{\End}{{\rm End}}          
\newcommand{\w}{{\rm w}}

\renewcommand{\i}{{\mathbf{i}}}
\newcommand{\comb}{{\mathrm{comb}}}
\newcommand{\Floer}{{\mathrm{Floer}}}

\newcommand{\SO}{{\rm SO}}

\newcommand{\PSL}{{\rm PSL}}

\newcommand{\Cinf}{C^{\infty}}

\newcommand{\CF}{{\mathrm{CF}}}
\newcommand{\HF}{{\mathrm{HF}}}
\newcommand{\BC}{{\mathrm{BC}}}

\newcommand{\RP}{{\mathbb{R}\mathrm{P}}}

\def\NABLA#1{{\mathop{\nabla\kern-.5ex\lower1ex\hbox{$#1$}}}}
\def\Nabla#1{\nabla\kern-.5ex{}_{#1}}
\def\abs#1{\mathopen|#1\mathclose|}
\def\Abs#1{\left|#1\right|}

\def\Norm#1{\left\|#1\right\|}
\renewcommand{\p}{{\partial}}

\newtheorem{theorem}{Theorem}[section]
\newtheorem{corollary}[theorem]{Corollary}

\newtheorem{lemma}[theorem]{Lemma}
\newtheorem{proposition}[theorem]{Proposition}
\newtheorem{definition}[theorem]{Definition}
\newtheorem{remark}[theorem]{Remark}
\newtheorem{example}[theorem]{Example}

\begin{document}
\maketitle
\newpage

\tableofcontents


\newpage

\section{Introduction}

The Floer homology of a transverse pair of Lagrangian submanifolds
in a symplectic manifold   is, under favorable hypotheses, 
the homology of a chain complex generated by the intersection points.  
The boundary operator counts index one holomorphic strips with 
boundary on the Lagrangian submanifolds.
This theory was introduced by Floer 
in~\cite{FLOER1,FLOER2};\phantomsection\label{FLOERa}
see also the three papers~\cite{OH} 
of Oh.\phantomsection\label{OHa}
In this memoir we consider the following special case:

\begin{quote}
\hypertarget{property_H}{{\bf (H)}}
{\it $\Sigma$ is a connected oriented 2-manifold without boundary
and $\alpha,\beta\subset\Sigma$ are connected smooth 
one dimensional oriented submanifolds without boundary 
which are closed as subsets of $\Sigma$
and intersect transversally.}
We do not assume that $\Sigma$ is compact, 
but when it is, $\alpha$ and $\beta$ are embedded circles. 
\end{quote}
In this special case there is a purely combinatorial approach
to Lagrangian Floer homology which was first developed by 
de Silva~\cite{DESILVA}.\phantomsection\label{DESILVA1}
We give a full and detailed definition 
of this combinatorial Floer homology (see Theorem~\ref{thm:floer1}) 
under the hypothesis that $\alpha$ and $\beta$ are noncontractible 
embedded circles and are not isotopic to each other.  
Under this hypothesis,  {\em combinatorial Floer homology 
is invariant under isotopy, not just Hamiltonian isotopy, 
as in  Floer's original work} (see Theorem~\ref{thm:floer2}).
Combinatorial Floer homology {\em is} isomorphic 
to analytic Floer homology as defined by Floer  
(see Theorem~\ref{thm:floer3}).

Floer homology is the homology of a chain complex 
$\CF(\alpha,\beta)$ with basis consisting of the points 
of the intersection $\alpha\cap\beta$
(and coefficients in $\Z/2\Z$). The boundary operator 
$\p:\CF(\alpha,\beta)\to\CF(\alpha,\beta)$ has the form
$$
\p x =\sum_y n(x,y) y.
$$
In the case of analytic Floer homology as defined by 
Floer $n(x,y)$ denotes the number (mod two) of 
equivalence classes of holomorphic strips $v:\S\to\Sigma$ 
satisfying the boundary conditions 
$$
v(\R)\subset\alpha,\qquad 
v(\R+\i)\subset\beta,\qquad
v(-\infty)=x,\qquad
v(+\infty)=y
$$
and having Maslov index one.
The boundary operator in combinatorial Floer homology 
has the same form but now $n(x,y)$ denotes 
the number (mod two) of equivalence classes of 
smooth immersions $u:\D\to\Sigma$ satisfying 
$$
u(\D\cap\R)\subset\alpha,\qquad
u(\D\cap S^1)\subset\beta,\qquad
u(-1)=x,\qquad  u(+1)=y. 
$$
We call such an immersion a {\em smooth lune}. Here 
$$
\S:=\R+\i[0,1], \qquad \D:=\{z\in\C\,|\,\IM\,z\ge 0,\,|z|\le 1\}
$$ 
denote the standard strip and the standard half disc respectively.
We develop the combinatorial theory without appeal to the 
difficult analysis required for the analytic theory.
The invariance under isotopy rather than just Hamiltonian 
isotopy (Theorem~\ref{thm:floer3}) is a benefit of this approach.
A corollary is the formula  
$$
\dim\HF(\alpha,\beta) = \geo{\alpha}{\beta}
$$ 
for the dimension of the 
\hyperlink{floer_homology}{Floer Homology} $\HF(\alpha,\beta)$
(see Corollary~\ref{cor:HF}). Here $\geo{\alpha}{\beta}$ denotes the 
\hyperlink{geometric_intersection_number}{\em geometric intersection number} 
of the curves $\alpha$ and $\beta$.
In Remark~\ref{rmk:Zcoeff} we indicate how to 
define combinatorial Floer homology with integer coefficients,
but we do not discuss integer coefficients in analytic 
Floer homology.

Let $\cD$ denote the space of all smooth maps $u:\D\to\Sigma$ 
satisfying the boundary conditions $u(\D\cap\R)\subset\alpha$ 
and $u(\D\cap S^1)\subset\beta$.
For $x,y\in\alpha\cap\beta$ let $\cD(x,y)$
denote the subset of all $u\in\cD$ satisfying 
the endpoint conditions $u(-1)=x$ and $u(1)=y$.
Each $u\in\cD$ determines a locally constant 
function 
$$
{\w:\Sigma\setminus(\alpha\cup\beta)\to\Z}
$$ 
defined as the degree
$$
\w(z) := \deg(u,z), \qquad z\in\Sigma\setminus(\alpha\cup\beta).
$$
When $z$ is a regular value of $u$ this is the algebraic 
number of points in the preimage $u^{-1}(z)$. 
The function $\w$ depends only on the homotopy class of~$u$.
In Theorem~\ref{thm:trace} we prove that the homotopy class 
of $u\in\cD$ is uniquely determined by its endpoints $x,y$ 
and its degree function $\w$.  
Theorem~\ref{thm:maslov} says that the Viterbo--Maslov index 
of every smooth map~$u\in\cD(x,y)$ is determined by the values 
of $\w$ near the endpoints $x$ and $y$ of $u$, namely,
it is given by the following {\em trace formula}
\index{trace!formula}\index{Viterbo--Maslov index!trace formula}
$$
\mu(u) = \frac{m_x(\Lambda_u)+m_y(\Lambda_u)}{2},\qquad
\Lambda_u:=(x,y,\w).
$$
Here $m_x$ denotes the sum of the four values 
of $\w$ encountered when walking along a 
small circle surrounding $x$, and similarly for $y$.
\hyperlink{Part_I}{Part~I} of this memoir is devoted to 
proving this formula.

\hyperlink{Part_II}{Part~II} 
gives a combinatorial characterization 
of smooth lunes.  Specifically, the equivalent conditions~(ii) and~(iii) 
of Theorem~\ref{thm:lune1} are necessary for the 
existence of a smooth lune. This implies the fact 
(not obvious to us) that a lune cannot contain 
either of its endpoints in the interior of its image.
In the simply connected case we prove in 
the same theorem that the necessary conditions
are also sufficient. We conjecture that they characterize 
smooth lunes in general. Theorem~\ref{thm:lune2}
shows that any two smooth lunes with the 
same counting function $\w$ are isotopic and thus
the equivalence class of a smooth lune is uniquely 
determined by its combinatorial data.  The proofs 
of these theorems are carried out in 
Sections~\ref{sec:ARC} and~\ref{sec:CL}.
Together they provide a solution to the 
Picard--Loewner problem in a special case;
see for example~\cite{FRANCIS2} and the references
cited therein, e.g.~\cite{TITUS,BLANK,POENARU}. 
\phantomsection\label{FRANCIS2,TITUS,BLANK,POENARU}
Our result is a special case
because no critical points are allowed (lunes are immersions),
the source is a disc and not a Riemann surface with positive genus,
and the prescribed boundary circle decomposes into two embedded arcs.

\hyperlink{Part_III}{Part~III} introduces combinatorial Floer homology.
Here  we restrict our discussion to the case 
where $\alpha$ and $\beta$ are noncontractible 
embedded circles which are not isotopic to each other 
(with either orientation). The basic definitions are given 
in Section~\ref{sec:FLOER}.  That the square of the boundary 
operator is indeed zero in the combinatorial setting 
will be proved in Section~\ref{sec:HEART} by 
analyzing \hyperlink{broken_hearts}{\it broken hearts}. 
Propositions~\ref{prop:heart} and~\ref{prop:heartbreaker} 
say that {\em there are two ways to break a heart}
and this is why the square of the boundary operator is zero.
In  Section~\ref{sec:ISOTOPY} we prove the  isotopy invariance 
of combinatorial Floer homology by examining generic deformations 
of loops  that change the number of intersection points.
This is very much in the spirit of Floer's original 
proof of deformation invariance 
(under Hamiltonian isotopy of the Lagrangian manifolds) 
of analytic Floer homology.   
The main theorem in Section~\ref{sec:LS} asserts, in the 
general setting, that smooth lunes (up to isotopy) are in 
one-to-one correspondence with index one holomorphic strips 
(up to translation). 
The proof is self-contained and does not use 
any of the other results in this memoir. 
It is based on an equation 
(the {\em index formula}~\eqref{eq:MASLOV} 
in Theorem~\ref{thm:MASLOV}) which expresses 
the Viterbo--Maslov index of a holomorphic 
strip in terms of its critical points and its angles at infinity.
A linear version of this equation 
(the {\em linear index formula}~\eqref{eq:index} 
in Lemma~\ref{le:index}) also shows that every holomorphic strip 
is regular in the sense that the linearized operator is surjective. 
It follows from these observations 
that the combinatorial and analytic definitions 
of Floer homology agree as asserted in Theorem~\ref{thm:floer3}.
In fact, our results show that the two chain complexes agree.

There are many directions in which the theory developed in the 
present memoir can be extended.  Some of these are discussed 
in Section~\ref{sec:FD}.  For example, it has been understood for 
some time that the Donaldson triangle product and the Fukaya 
category have combinatorial analogues in dimension two, 
and that these analogues are isomorphic  to the original analytic theories.
The combinatorial approach to the Donaldson triangle product
has been outlined in the PhD thesis of the first 
author~\cite{DESILVA}\phantomsection\label{DESILVA0},
and the combinatorial approach to the derived Fukaya category
has been used by Abouzaid~\cite{A}\phantomsection\label{A0} 
to compute it. Our formula for the Viterbo--Maslov index
in Theorem~\ref{thm:maslov} and our combinatorial characterization
of smooth lunes in Theorem~\ref{thm:lune1} are not needed for
their applications.   In our memoir these two results are limited
to the elements of $\cD$.  
(To our knowledge, they have not been extended to triangles 
or more general polygons in the existing literature.)

When $\Sigma=\T^2$, the Heegaard--Floer theory 
of Ozsvath--Szabo~\cite{OS1,OS2} can be interpreted as a refinement 
of the combinatorial Floer theory, in that the winding number of 
a lune at a prescribed point in $\T^2\setminus(\alpha\cup\beta)$
is taken into account in the definition of their boundary operator.
However, for higher genus surfaces Heegaard--Floer theory
does not include the combinatorial Floer theory discussed 
in the present memoir as a special case.

Appendix~\ref{app:path} contains a proof that, under suitable hypotheses,  
the space of paths connecting $\alpha$ to $\beta$ is simply connected.
Appendix~\ref{app:diffeos_of_D}  contains a proof that the group 
of orientation preserving diffeomorphisms of the half disc 
fixing the corners is connected.  Appendix~\ref{app:homological} 
contains an account of Floer's algebraic deformation argument.
Appendix~\ref{app:asymptotic} summarizes the relevant results
in~\cite{ASYMPTOTIC} about the asymptotic behavior of holomorphic strips.

{\bf Acknowledgement.}
We would like to thank the referee for his/her careful work.

\newpage
\part*{I. The Viterbo--Maslov Index}
\addcontentsline{toc}{part}{Part I. The Viterbo--Maslov Index}


Throughout this memoir we assume~\hyperlink{property_H}{(H)}.  
We often write ``assume~\hyperlink{property_H}{(H)}'' to remind the reader 
of our standing hypothesis.

\section{Chains and Traces}\label{sec:CT}

Define a cell complex structure on $\Sigma$ by taking  
the set of zero-cells to be the set $\alpha\cap\beta$,\hypertarget{Part_I}{}
the set of one-cells to be  the set of connected components 
of $(\alpha\setminus\beta)\cup(\beta\setminus\alpha)$
with compact closure, and the set of two-cells to be the 
set of connected components of $\Sigma\setminus(\alpha\cup\beta)$
with compact closure.  (There is an abuse of language here 
as the ``two-cells'' need not be  homeomorphs of the open unit
disc if the genus of $\Sigma$ is positive and the ``one-cells" 
need not be arcs if $\alpha\cap\beta=\emptyset$.)
Define a boundary operator $\p$ as follows. 
For each two-cell $F$ let 
$$
\p F=\sum \pm E, 
$$
where the sum is over the one-cells $E$ which abut $F$ and 
the plus sign is chosen iff the orientation of $E$ (determined 
from the given orientations of $\alpha$ and $\beta$)
agrees with the boundary orientation of $F$
as a connected open subset of the oriented manifold  $\Sigma$.
For each one-cell $E$ let 
$$
\p E=y-x
$$
where $x$ and $y$ are the endpoints of the arc $E$ and the 
orientation of $E$ goes from $x$ to $y$. (The one-cell $E$ 
is either a subarc of $\alpha$ or a subarc of $\beta$ and both 
$\alpha$ and $\beta$ are oriented one-manifolds.) 
For $k=0,1,2$ a {\bf $k$-chain} \index{chain}
is defined to be a formal linear
combination (with integer coefficients) of $k$-cells, i.e.\ 
a two-chain is a locally constant map 
$\Sigma\setminus(\alpha\cup\beta)\to\Z$ 
(whose support has compact closure in $\Sigma$)
and a one-chain is a locally constant map 
$(\alpha\setminus\beta)\cup(\beta\setminus\alpha)\to\Z$
(whose support has compact closure in $\alpha\cup\beta$).
It follows directly from the definitions that
$\p^2F=0$  for each two-cell $F$.

Each $u\in\cD$ determines a two-chain $\w$ via
\begin{equation}\label{eq:w-V}
\w(z) := \deg(u,z), \qquad 
z\in\Sigma\setminus(\alpha\cup\beta).
\end{equation}
and a one-chain $\nu$ via
\begin{equation}\label{eq:nu-V}
\nu(z) := 
\left\{\begin{array}{rl}
\deg(u\big|_{\p\D\cap\R\;\,}:\p\D\cap\R\;\to\alpha,z),&
\mbox{for } z\in\alpha\setminus\beta, \\
-\deg(u\big|_{\p\D\cap S^1}:\p\D\cap S^1\to\beta,z),&
\mbox{for } z\in\beta\setminus\alpha.
\end{array}\right.
\end{equation}
Here we orient the one-manifolds $\D\cap\R$ and 
$\D\cap S^1$ from $-1$ to $+1$.  For any one-chain 
$\nu:(\alpha\setminus\beta)\cup(\beta\setminus\alpha)\to\Z$
denote 
$$
\nu_\alpha:=\nu|_{\alpha\setminus\beta}:
\alpha\setminus\beta\to\Z,\qquad
\nu_\beta:=\nu|_{\beta\setminus\alpha}:
\beta\setminus\alpha\to\Z.
$$
Conversely, given locally constant functions 
$\nu_\alpha:\alpha\setminus\beta\to\Z$ 
(whose support has compact closure in $\alpha$)
and $\nu_\beta:\beta\setminus\alpha\to\Z$
(whose support has compact closure in $\beta$),
denote by $\nu=\nu_\alpha-\nu_\beta$ the one-chain that 
agrees with $\nu_\alpha$ on $\alpha\setminus\beta$ and 
agrees with $-\nu_\beta$ on $\beta\setminus\alpha$.

\begin{definition}[{\bf Traces}]\label{def:trace-V}\rm 
Fix two (not necessarily distinct)
intersection points $x,y\in\alpha\cap\beta$.

\smallskip\noindent{\bf (i)}
Let $\w:\Sigma\setminus(\alpha\cup\beta)\to\Z$
be a two-chain.  The triple 
$
\Lambda=(x,y,\w)
$
is called an {\bf $(\alpha,\beta)$-trace} if 
\index{trace!a@$(\alpha,\beta)$}
there exists an element $u\in\cD(x,y)$ such that
$\w$ is given by~\eqref{eq:w-V}. 
In this case $\Lambda=:\Lambda_u$ 
is also called the {\bf $(\alpha,\beta)$-trace of $u$}
and we sometimes write $\w_u:=\w$.
\index{trace!of $u$, denoted $\Lambda_u$}

\smallskip\noindent{\bf (ii)}
Let $\Lambda=(x,y,\w)$ be an $(\alpha,\beta)$-trace.
The triple $\p\Lambda:=(x,y,\p\w)$
is called the {\bf boundary of $\Lambda$.} \index{boundary}

\smallskip\noindent{\bf (iii)}
A one-chain 
$\nu:(\alpha\setminus\beta)\cup(\beta\setminus\alpha)\to\Z$
is called an {\bf $(x,y)$-trace} if there 
exist \index{trace!x@$(x,y)$}
smooth curves $\gamma_\alpha:[0,1]\to\alpha$
and $\gamma_\beta:[0,1]\to\beta$ such that
$\gamma_\alpha(0)=\gamma_\beta(0)=x$,
$\gamma_\alpha(1)=\gamma_\beta(1)=y$,
$\gamma_\alpha$ and $\gamma_\beta$ 
are homotopic in $\Sigma$ with fixed endpoints, and 
\begin{equation}\label{eq:admissible}
\nu(z)=\left\{\begin{array}{rl}
\deg(\gamma_\alpha,z),&\mbox{for }z\in\alpha\setminus\beta,\\
-\deg(\gamma_\beta,z),&\mbox{for }z\in\beta\setminus\alpha.
\end{array}\right.
\end{equation}
\end{definition}

\begin{remark}\label{rmk:nu}\rm
Assume $\Sigma$ is simply connected.
Then the condition on $\gamma_\alpha$ and $\gamma_\beta$ 
to be homotopic with fixed endpoints is redundant.
Moreover, if $x=y$ then a one-chain $\nu$ 
is an $(x,y)$-trace if and only if the restrictions
$$
\nu_\alpha:=\nu|_{\alpha\setminus\beta},\qquad 
\nu_\beta:=-\nu|_{\beta\setminus\alpha}
$$ 
are constant.
If $x\ne y$ and $\alpha,\beta$ are embedded circles 
and $A,B$ denote the positively oriented arcs from $x$ to $y$ 
in $\alpha,\beta$, then a one-chain $\nu$ 
is an $(x,y)$-trace if and only if 
$$
\nu_\alpha|_{\alpha\setminus(A\cup\beta)}=\nu_\alpha|_{A\setminus\beta}-1
$$
and
$$
\nu_\beta|_{\beta\setminus(B\cup\alpha)}=\nu_\beta|_{B\setminus\alpha}-1.
$$
In particular, when walking along $\alpha$ or $\beta$, 
the function $\nu$ only changes its value at $x$ and $y$.
\end{remark}

\begin{lemma}\label{le:boundary}
Let $x,y\in\alpha\cap\beta$ and $u\in\cD(x,y)$.
Then the boundary of the $(\alpha,\beta)$-trace
$\Lambda_u$ of $u$ is the triple 
$
\p\Lambda_u=(x,y,\nu),
$
where $\nu$ is given by~\eqref{eq:nu-V}.
In other words, if $\w$ is given by~\eqref{eq:w-V}
and $\nu$ is given by~\eqref{eq:nu-V}
then $\nu=\p\w$.
\end{lemma}

\begin{proof}
Choose an embedding $\gamma:[-1,1]\to\Sigma$ 
such that $u$ is transverse to $\gamma$,  
$\gamma(t)\in\Sigma\setminus(\alpha\cup\beta)$ for $t\ne 0$, 
$\gamma(-1)$, $\gamma(1)$ are regular values of $u$, 
${\gamma(0)\in\alpha\setminus\beta}$ 
is a regular value of $u|_{\D\cap\R}$,
and $\gamma$ intersects $\alpha$ transversally 
at $t=0$ such that orientations match in
$$
T_{\gamma(0)}\Sigma = T_{\gamma(0)}\alpha \oplus\R\dot\gamma(0).
$$
Denote $\Gamma:=\gamma([-1,1])$. Then $u^{-1}(\Gamma)\subset\D$ 
is a $1$-dimensional submanifold with boundary
$$
\p u^{-1}(\Gamma) = u^{-1}(\gamma(-1))\cup u^{-1}(\gamma(1))
\cup \bigl(u^{-1}(\gamma(0))\cap\R)\bigr).
$$
If $z\in u^{-1}(\Gamma)$ then 
$$
\im\,du(z)+T_{u(z)}\Gamma=T_{u(z)}\Sigma,\qquad
T_zu^{-1}(\Gamma)=du(z)^{-1}T_{u(z)}\Gamma.
$$ 
We orient $u^{-1}(\Gamma)$ such that the orientations 
match in
$$
T_{u(z)}\Sigma = T_{u(z)}\Gamma\oplus du(z)\i T_zu^{-1}(\Gamma).
$$
In other words, if $z\in u^{-1}(\Gamma)$ and $u(z)=\gamma(t)$, 
then a nonzero tangent vector $\zeta\in T_zu^{-1}(\Gamma)$
is positive if and only if the pair $(\dot\gamma(t),du(z)\i\zeta)$ 
is a positive basis of $T_{\gamma(t)}\Sigma$.
Then the boundary orientation of $u^{-1}(\Gamma)$ 
at the elements of $u^{-1}(\gamma(1))$ agrees with the 
algebraic count  in the definition of $\w(\gamma(1))$,  
at the elements of $u^{-1}(\gamma(-1))$ is opposite 
to the algebraic count in the definition of $\w(\gamma(-1))$,
and at the elements of $u^{-1}(\gamma(0))\cap\R$ is opposite 
to the algebraic count in the definition of $\nu(\gamma(0))$.
Hence
$$
\w(\gamma(1)) = \w(\gamma(-1)) + \nu(\gamma(0)).
$$
In other words the value of $\nu$ at a point in
$\alpha\setminus\beta$ is equal to the value 
of~$\w$ slightly to the left of $\alpha$ 
minus the value of~$\w$ slightly to the right of~$\alpha$.
Likewise, the value of $\nu$ at a point in
$\beta\setminus\alpha$ is equal to the value 
of~$\w$ slightly to the right of $\beta$ 
minus the value of~$\w$ slightly to the left of~$\beta$.
This proves Lemma~\ref{le:boundary}.
\end{proof}

\begin{theorem}\label{thm:trace}
\smallskip\noindent{\bf (i)}
Two elements of $\cD$ belong to the same connected
component of~$\cD$ if and only if they have 
the same $(\alpha,\beta)$-trace.

\smallskip\noindent{\bf (ii)}
Assume $\Sigma$ is diffeomorphic to the two-sphere.
Let $x,y\in\alpha\cap\beta$ and let 
$\w:\Sigma\setminus(\alpha\cup\beta)\to\Z$
be a locally constant function.
Then $\Lambda=(x,y,\w)$ is an $(\alpha,\beta)$-trace 
if and only if $\p\w$ is an $(x,y)$-trace.

\smallskip\noindent{\bf (iii)}
Assume $\Sigma$ is not diffeomorphic to the two-sphere
and let $x,y\in\alpha\cap\beta$. If $\nu$ is an $(x,y)$-trace,
then there is a unique two-chain $\w$ such that 
$\Lambda:=(x,y,\w)$ is an $(\alpha,\beta)$-trace
and $\p\w=\nu$.
\end{theorem}

\begin{proof} 
We prove~(i). ``Only if'' follows from the standard 
arguments in degree theory as in Milnor~\cite{M}.  
\phantomsection\label{M}
To prove ``if'', fix two intersection points 
$$
x,y\in\alpha\cap\beta
$$ 
and, for $X=\Sigma,\alpha,\beta$, 
denote by $\cP(x,y;X)$ the space of all smooth curves 
${\gamma:[0,1]\to X}$ satisfying $\gamma(0)=x$ and $\gamma(1)=y$.
Every $u\in\cD(x,y)$ determines smooth paths 
$
\gamma_{u,\alpha}\in\cP(x,y;\alpha)
$
and
$
\gamma_{u,\beta}\in\cP(x,y;\beta)
$ 
via
\begin{equation}\label{eq:gaualbe}
\gamma_{u,\alpha}(s):=u(-\cos(\pi s),0),\qquad
\gamma_{u,\beta}(s)=u(-\cos(\pi s),\sin(\pi s)).
\end{equation}
These paths are homotopic in $\Sigma$ with fixed
endpoints. An explicit homotopy is the map 
$$
F_u:=u\circ\phi:[0,1]^2\to\Sigma
$$
where $\phi:[0,1]^2\to\D$ is the map
$$
\phi(s,t):=(-\cos(\pi s),t\sin(\pi s)).
$$
By Lemma~\ref{le:boundary}, the homotopy class of 
$\gamma_{u,\alpha}$ in $\cP(x,y;\alpha)$ is uniquely determined by 
$$
\nu_\alpha:=\p\w_u|_{\alpha\setminus\beta}:\alpha\setminus\beta\to\Z
$$
and that of $\gamma_{u,\beta}$ in $\cP(x,y;\beta)$ is uniquely determined 
by 
$$
\nu_\beta:=-\p\w_u|_{\beta\setminus\alpha}:\beta\setminus\alpha\to\Z.
$$  
Hence they are both uniquely determined by the $(\alpha,\beta)$-trace
of $u$. If $\Sigma$ is not diffeomorphic to the $2$-sphere 
the assertion follows from the fact that
each component of $\cP(x,y;\Sigma)$ is contractible
(because the universal cover of $\Sigma$ is diffeomorphic
to the complex plane).  Now assume $\Sigma$ is diffeomorphic 
to the $2$-sphere. Then $\pi_1(\cP(x,y;\Sigma))=\Z$
acts on $\pi_0(\cD)$ because the correspondence $u\mapsto F_u$
identifies $\pi_0(\cD)$ with a space of homotopy classes of 
paths in $\cP(x,y;\Sigma)$ connecting 
$\cP(x,y;\alpha)$ to $\cP(x,y;\beta)$.   
The induced action on the space of two-chains 
$\w:\Sigma\setminus(\alpha\cup\beta)$ is given by adding 
a global constant.  Hence the map $u\mapsto\w$ induces an injective 
map 
$$
\pi_0(\cD(x,y))\to\{\mbox{$2$-chains}\}.
$$
This proves~(i).

We prove~(ii) and~(iii). 
Let $\w$ be a two-chain, suppose that 
$
\nu:=\p\w
$
is an $(x,y)$-trace, and denote
$
\Lambda := (x,y,\w).
$
Let ${\gamma_\alpha:[0,1]\to\alpha}$ and 
${\gamma_\beta:[0,1]\to\beta}$ 
be as in Definition~\ref{def:trace-V}.
Then there is a $u'\in\cD(x,y)$ such that the map
$s\mapsto u'(-\cos(\pi s),0)$ is homotopic to $\gamma_\alpha$
and $s\mapsto u'(-\cos(\pi s),\sin(\pi s))$ is homotopic to  $\gamma_\beta$.  
By definition the $(\alpha,\beta)$-trace of $u'$
is $\Lambda'=(x,y,\w')$ for some two-chain $\w'$.  
By Lemma~\ref{le:boundary}, we have
$$
\p\w'=\nu=\p\w
$$
and hence $\w-\w'=:d$ is constant.
If $\Sigma$ is not diffeomorphic to the two-sphere
and $\Lambda$ is the $(\alpha,\beta)$-trace of some element 
$u\in\cD$, then $u$ is homotopic to $u'$ (as $\cP(x,y;\Sigma)$ 
is simply connected) and hence $d=0$ and $\Lambda=\Lambda'$.
If $\Sigma$ is diffeomorphic to the $2$-sphere choose
a smooth map $v:S^2\to \Sigma$ of degree $d$ and replace $u'$ 
by the connected sum $u:=u'\# v$.  
Then $\Lambda$ is the $(\alpha,\beta)$-trace of $u$.
This proves Theorem~\ref{thm:trace}.
\end{proof}

\begin{remark}\label{rmk:trace}\rm
Let $\Lambda=(x,y,\w)$ be an $(\alpha,\beta)$-trace
and define 
$$
\nu_\alpha:=\p\w|_{\alpha\setminus\beta},\qquad
\nu_\beta:=-\p\w|_{\beta\setminus\alpha}.
$$

\smallskip\noindent{\bf (i)}
The two-chain $\w$ is uniquely 
determined by the condition $\p\w=\nu_\alpha-\nu_\beta$ 
and its value at one point.  To see this, think of the 
embedded circles $\alpha$ and $\beta$ as traintracks. \index{traintracks}
Crossing $\alpha$ at a point $z\in\alpha\setminus\beta$
increases $\w$ by $\nu_\alpha(z)$
if the train comes from the left, and decreases it 
by~$\nu_\alpha(z)$ if the train comes from the right.
Crossing $\beta$ at a point $z\in\beta\setminus\alpha$
decreases $\w$ by $\nu_\beta(z)$
if the train comes from the left 
and increases it by $\nu_\beta(z)$
if the train comes from the right.
Moreover, $\nu_\alpha$ extends continuously to $\alpha\setminus\{x,y\}$
and $\nu_\beta$ extends continuously to $\beta\setminus\{x,y\}$.
At each intersection point $z\in(\alpha\cap\beta)\setminus\{x,y\}$ 
with intersection index $+1$ (respectively $-1$)
the function $\w$ takes the values
$$
k,\quad
k+\nu_\alpha(z),\quad
k+\nu_\alpha(z)-\nu_\beta(z),\quad
k-\nu_\beta(z)
$$ 
as we march counterclockwise (respectively clockwise)
along a small circle surrounding the intersection point.

\bigbreak

\smallskip\noindent{\bf (ii)}
If $\Sigma$ is not diffeomorphic to the $2$-sphere
then, by Theorem~\ref{thm:trace}~(iii), 
the $(\alpha,\beta)$-trace $\Lambda$ is uniquely determined 
by its boundary $\p\Lambda=(x,y,\nu_\alpha-\nu_\beta)$.

\smallskip\noindent{\bf (iii)}
Assume $\Sigma$ is not diffeomorphic to the $2$-sphere
and choose a universal covering $\pi:\C\to\Sigma$.
Choose a point $\tx\in\pi^{-1}(x)$ and lifts 
$\talpha$ and $\tbeta$ of~$\alpha$ and~$\beta$ 
such that
$
\tx\in\talpha\cap\tbeta.
$
Then $\Lambda$ lifts to an $(\talpha,\tbeta)$-trace  
$$
\tLa = (\tx,\ty,\widetilde\w).
$$
More precisely, the one chain $\nu:=\nu_\alpha-\nu_\beta=\p\w$ 
is an $(x,y)$-trace, by Lemma~\ref{le:boundary}.  
The paths  $\gamma_\alpha:[0,1]\to\alpha$ and $\gamma_\beta:[0,1]\to\beta$ 
in Definition~\ref{def:trace-V} 
lift to unique paths $\gamma_\talpha:[0,1]\to\talpha$ and 
$\gamma_\tbeta:[0,1]\to\tbeta$ connecting $\tx$ to $\ty$.
For $\tz\in\C\setminus(\tA\cup\tB)$ the number
$\widetilde\w(\tz)$ is the winding number of the loop 
$
\gamma_\talpha-\gamma_\tbeta
$ 
about $\tz$ (by Rouch\'e's theorem). 
The two-chain $\w$ is then given by 
$$
\w(z)
= \sum_{\tz\in\pi^{-1}(z)} \widetilde\w(\tz),\qquad
z\in\Sigma\setminus(\alpha\cup\beta).
$$
To see this, lift an element $u\in\cD(x,y)$ with 
$(\alpha,\beta)$-trace $\Lambda$ to the 
universal cover to obtain an element 
$\tu\in\cD(\tx,\ty)$ with $\Lambda_\tu=\tLa$
and consider the degree.
\end{remark}

\begin{definition}[{\bf Catenation}]\label{def:catenation}\rm
Let $x,y,z\in\alpha\cap\beta$. 
The {\bf catenation of two $(\alpha,\beta)$-traces 
\index{catenation}
$\Lambda=(x,y,\w)$ and $\Lambda'=(y,z,\w')$} is defined by 
$$
\Lambda\#\Lambda' := (x,z,\w+\w').
$$
Let $u\in\cD(x,y)$ and $u'\in\cD(y,z)$
and suppose that $u$ and $u'$ are constant 
near the ends $\pm1\in\D$. 
For $0<\lambda<1$ sufficiently close to one 
the {\bf $\lambda$-catenation of $u$ and $u'$} 
is the map $u\#_\lambda u'\in\cD(x,z)$ defined by
$$
(u\#_\lambda u')(\zeta) := \left\{\begin{array}{ll}
u\left(\frac{\zeta+\lambda}{1+\lambda\zeta}\right),&
\mbox{for }\Re\,\zeta\le 0,\\
u'\left(\frac{\zeta-\lambda}{1-\lambda\zeta}\right),&
\mbox{for }\Re\,\zeta\ge 0.
\end{array}\right.
$$
\end{definition}

\begin{lemma}\label{le:catenation}
If $u\in\cD(x,y)$ and $u'\in\cD(y,z)$ are as in 
Definition~\ref{def:catenation} then
$$
\Lambda_{u\#_\lambda u'} = \Lambda_u \#\Lambda_{u'}.
$$
Thus the catenation of two $(\alpha,\beta)$-traces 
is again an $(\alpha,\beta)$-trace.
\end{lemma}

\begin{proof}
This follows directly from the definitions.
\end{proof}


\section{The Maslov Index}\label{sec:M}

\begin{definition}\label{def:maslov}\rm
Let $x,y\in\alpha\cap\beta$ and $u\in\cD(x,y)$.
Choose an orientation preserving trivialization 
$$
\D\times\R^2\to u^*T\Sigma:(z,\zeta)\mapsto\Phi(z)\zeta
$$
and consider the Lagrangian paths 
$$
\lambda_0,\lambda_1:[0,1]\to\RP^1
$$
given by 
\begin{equation*}
\begin{split}
\lambda_0(s)&:=\Phi(-\cos(\pi s),0)^{-1}
T_{u(-\cos(\pi s),0)}\alpha,\\
\lambda_1(s)&:=\Phi(-\cos(\pi s),\sin(\pi s))^{-1}
T_{u(-\cos(\pi s),\sin(\pi s))}\beta.
\end{split}
\end{equation*}
The {\bf Viterbo--Maslov index of $u$} is defined as the 
\index{Viterbo--Maslov index}\index{Maslov index}
\index{Viterbo--Maslov index!of a pair of Lagrangian paths}
\index{Viterbo--Maslov index!of an $(\alpha,\beta)$-trace}
relative Maslov index of the pair of Lagrangian paths
$(\lambda_0,\lambda_1)$ and will be denoted by
$$
\mu(u) := \mu(\Lambda_u) := \mu(\lambda_0,\lambda_1).
$$
By the naturality and homotopy axioms for the relative 
Maslov index (see for example~\cite{RS2}), 
\phantomsection\label{RS2a}
the number $\mu(u)$ is independent of the choice of the 
trivialization and depends only on the homotopy class of $u$;
hence it depends only on the $(\alpha,\beta)$-trace of $u$,
by Theorem~\ref{thm:trace}.   The relative Maslov index 
$\mu(\lambda_0,\lambda_1)$ is the degree of the loop in 
$\RP^1$ obtained by traversing $\lambda_0$,  followed by 
a counterclockwise turn from $\lambda_0(1)$ to $\lambda_1(1)$,
followed by traversing $\lambda_1$ in reverse time, followed
by a clockwise turn from $\lambda_1(0)$ to $\lambda_0(0)$.
This index was first defined by Viterbo~\cite{VITERBO} 
\phantomsection\label{V1} (in all dimensions).  
Another exposition is contained 
in~\cite{RS2}.\phantomsection\label{RS2b}
\end{definition}

\begin{remark}\label{rmk:mascat}\rm
The Viterbo--Maslov index is additive under catenation,
i.e.\ if 
$$
\Lambda=(x,y,\w),\qquad
\Lambda'=(y,z,\w')
$$ 
are $(\alpha,\beta)$-traces then
$$
\mu(\Lambda\#\Lambda') = \mu(\Lambda)+\mu(\Lambda').
$$
For a proof of this formula see~\cite{VITERBO,RS2}.
\phantomsection\label{V2}\label{RS2c}
\end{remark}

\begin{definition}\label{def:arc-V}\rm
Let $\Lambda=(x,y,\w)$ be an $(\alpha,\beta)$-trace
and denote $\nu_\alpha:=\p\w|_{\alpha\setminus\beta}$ and
$\nu_\beta:=-\p\w|_{\beta\setminus\alpha}$.
$\Lambda$ is said to satisfy the 
{\bf arc condition} if \index{arc condition}
\begin{equation}\label{eq:arc-V}
x\ne y,\qquad \min\Abs{\nu_\alpha} = \min\Abs{\nu_\beta}=0.
\end{equation}
When $\Lambda$ satisfies the arc condition
there are arcs $A\subset\alpha$ and $B\subset\beta$ 
from $x$ to $y$ such that
\begin{equation}\label{eq:nuAB-V}
\nu_\alpha(z) = \left\{\begin{array}{rl}
\pm1,&\mbox{if }z\in A,\\
0,&\mbox{if }z\in\alpha\setminus\overline A,
\end{array}\right.\;\;
\nu_\beta(z) = \left\{\begin{array}{rl}
\pm1,&\mbox{if }z\in B,\\
0,&\mbox{if }z\in\beta\setminus\overline B.
\end{array}\right.
\end{equation}
Here the plus sign is chosen iff the orientation of $A$ from 
$x$ to $y$ agrees with that of $\alpha$, respectively the 
orientation of $B$ from $x$ to $y$ agrees with that of~$\beta$.
In this situation the quadruple $(x,y,A,B)$ and the triple
$(x,y,\p\w)$ determine one another and we also write
$$
\p\Lambda = (x,y,A,B)
$$ 
for the boundary of $\Lambda$.
When $u\in\cD$ and $\Lambda_u=(x,y,\w)$ 
satisfies the arc condition and $\p\Lambda_u=(x,y,A,B)$
then 
$$
s\mapsto u(-\cos(\pi s),0)
$$
is homotopic in $\alpha$ 
to a path traversing $A$ and the path 
$$
s\mapsto u(-\cos(\pi s),\sin(\pi s))
$$
is homotopic in~$\beta$ to a path traversing $B$.
\end{definition}

\begin{theorem}\label{thm:maslov}
Let $\Lambda=(x,y,\w)$ be an $(\alpha,\beta)$-trace.  
For $z\in\alpha\cap\beta$ denote by $m_z(\Lambda)$ 
the sum of the four values of $\w$ encountered when 
walking along a small circle surrounding~$z$.
Then the Viterbo--Maslov index of~$\Lambda$ 
is given by
\begin{equation}\label{eq:maslov}
\mu(\Lambda) = \frac{m_x(\Lambda)+m_y(\Lambda)}{2}.
\end{equation}
We call this the {\bf trace formula}.  
\index{Viterbo--Maslov index!trace formula}
\index{trace!formula}
\end{theorem}

We first prove the trace formula for the $2$-plane $\C$ 
and the $2$-sphere $S^2$ (Section~\ref{sec:PLANE} on 
page~\pageref{proof:maslov1}).  
When $\Sigma$ is not simply connected we reduce 
the result to the case of the $2$-plane 
(Section~\ref{sec:LIFT} page~\pageref{proof:maslov2}).  
The key is the identity
\begin{equation}\label{eq:mtilde}
m_{g\tx}(\tLa) + m_{g^{-1}\ty}(\tLa)=0
\end{equation}
for every lift $\tLa$ to the universal cover 
and every deck transformation $g\ne\id$. 
We call this the {\bf cancellation formula}.
\index{Viterbo--Maslov index!cancellation formula}
\index{cancellation formula}


\section{The Simply Connected Case}\label{sec:PLANE}

A connected oriented $2$-manifold $\Sigma$
is called {\bf planar} if it admits an 
(orientation preserving) \index{planar $2$-manifold}
embedding into the complex plane. 

\begin{proposition}\label{prop:maslovC}
The trace formula~\eqref{eq:maslov} holds when $\Sigma$ is planar.
\end{proposition}

\begin{proof}
Assume first that $\Sigma=\C$ and 
$
\Lambda=(x,y,\w)
$ 
satisfies the arc condition.  Thus
the boundary of $\Lambda$ has the form
$$
\p\Lambda = (x,y,A,B),
$$
where $A\subset\alpha$ and $B\subset\beta$ are arcs from 
$x$ to $y$ and $\w(z)$ is the winding number of the loop $A-B$ 
about the point $z\in\Sigma\setminus(A\cup B)$
(see Remark~\ref{rmk:trace}). Hence the 
trace formula~\eqref{eq:maslov} can be written in the form
\begin{equation}\label{eq:masLov}
\mu(\Lambda) = 2k_x + 2k_y + \frac{\eps_x-\eps_y}{2}.
\end{equation}
Here 
$
\eps_z=\eps_z(\Lambda)\in\{+1,-1\}
$ 
denotes the intersection index of $A$ and $B$ 
at a point $z\in A\cap B$, $k_x=k_x(\Lambda)$ 
denotes the value of the winding number $\w$ 
at a point in $\alpha\setminus A$ close to $x$,
and $k_y=k_y(\Lambda)$ denotes the value of $\w$ 
at a point in $\alpha\setminus A$ close to $y$.
We now prove~\eqref{eq:masLov} under the hypothesis
that $\Lambda$ satisfies the arc condition.
The proof is by induction on the number of intersection 
points of $B$ and~$\alpha$ and has seven steps.

\medskip\noindent{\bf Step~1.}
{\it We may assume without loss of generality that
\begin{equation}\label{eq:AB}
\Sigma=\C,\qquad \alpha=\R,\qquad A=[x,y],\qquad x<y,
\end{equation}
and $B\subset\C$ is an embedded arc from $x$ to $y$ 
that is transverse to $\R$.}

\medskip\noindent
Choose a diffeomorphism from $\Sigma$ 
to $\C$ that maps $A$ to a bounded closed interval 
and maps $x$ to the left endpoint of $A$.
If $\alpha$ is not compact the diffeomorphism can be chosen
such that it also maps $\alpha$ to $\R$.
If $\alpha$ is an embedded circle the diffeomorphism 
can be chosen such that its restriction to $B$ is 
transverse to $\R$; now replace the image of~$\alpha$ by~$\R$. 
This proves Step~1.

\medskip\noindent{\bf Step~2.}
{\it Assume~\eqref{eq:AB} and let 
$\bar\Lambda:=(x,y,z\mapsto -\w(\bar z))$ 
be the $(\alpha,\bar\beta)$-trace 
obtained from $\Lambda$ by complex conjugation.
Then $\Lambda$ satisfies~\eqref{eq:masLov} if and only if 
$\bar\Lambda$ satisfies~\eqref{eq:masLov}.}

\medskip\noindent
Step~2 follows from the fact that the numbers
$\mu,k_x,k_y,\eps_x,\eps_y$ change sign 
under complex conjugation.

\medskip\noindent{\bf Step~3.}
{\it  Assume~\eqref{eq:AB}.  If $B\cap\R=\{x,y\}$ then $\Lambda$ 
satisfies~\eqref{eq:masLov}.}

\medskip\noindent
In this case $B$ is contained in the upper 
or lower closed half plane and the loop $A\cup B$ 
bounds a disc contained in the same half plane.
By Step~1 we may assume that $B$ is contained in the 
upper half space. Then $\eps_x=1$, $\eps_y=-1$,
and $\mu(\Lambda)=1$.  Moreover, the winding number 
$\w$ is one in the disc encircled by $A$ 
and $B$ and is zero in the complement of its closure.  
Since the intervals $(-\infty,0)$ and $(0,\infty)$ 
are contained in this complement, we have $k_x=k_y=0$.
This proves Step~3.

\medskip\noindent{\bf Step~4.}
{\it Assume~\eqref{eq:AB} and $\#(B\cap\R)>2$, 
follow the arc of $B$, starting at $x$, 
and let~$x'$ be the next intersection point with $\R$. 
Assume $x'<x$, denote by $B'$ the arc in $B$ from $x'$ to $y$, 
and let $A':=[x',y]$ (see Figure~\ref{fig:maslov4}).
If the $(\alpha,\beta)$-trace 
$\Lambda'$ with boundary $\p\Lambda'=(x',y,A',B')$ 
satisfies~\eqref{eq:masLov} so does $\Lambda$.}

\begin{figure}[htp] 
\centering 
\includegraphics[scale=0.4]{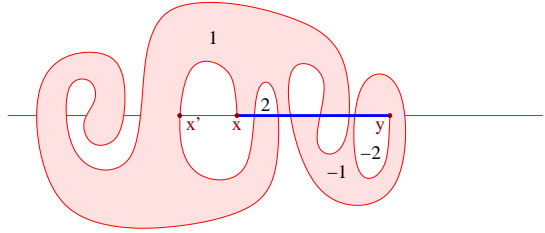}      
\caption{{Maslov index and catenation: $x'<x<y$.}}\label{fig:maslov4}      
\end{figure}

\medskip\noindent
By Step~2 we may assume $\eps_x(\Lambda)=1$.
Orient $B$ from $x$ to $y$.
The Viterbo--Maslov index of $\Lambda$ is 
minus the Maslov index of the path 
$
B\to\RP^1:z\mapsto T_zB, 
$
relative to the Lagrangian subspace $\R\subset\C$.
Since the Maslov index of the arc in $B$ from 
$x$ to $x'$ is $+1$ we have
\begin{equation}\label{eq:case11}
\mu(\Lambda) = \mu(\Lambda')-1.
\end{equation}
Since the orientations of $A'$ and $B'$ 
agree with those of $A$ and $B$ we have
\begin{equation}\label{eq:case12}
\eps_{x'}(\Lambda') = \eps_{x'}(\Lambda)=-1,\qquad
\eps_y(\Lambda')=\eps_y(\Lambda).
\end{equation}
Now let $x_1<x_2<\cdots<x_m<x$ be the intersection
points of $\R$ and $B$ in the interval $(-\infty,x)$
and let $\eps_i\in\{-1,+1\}$ be the intersection
index of $\R$ and $B$ at $x_i$. 
Then there is an integer $\ell\in\{1,\dots,m\}$ 
such that $x_\ell=x'$ and $\eps_\ell=-1$. 
Moreover, the winding number $\w$
slightly to the left of $x$ is 
$$
k_x(\Lambda) = \sum_{i=1}^m\eps_i.
$$
It agrees with the value of $\w$ 
slightly to the right of $x'=x_\ell$.  Hence 
\begin{equation}\label{eq:case13}
k_x(\Lambda) 
= \sum_{i=1}^\ell \eps_i 
= \sum_{i=1}^{\ell-1}\eps_i - 1 
= k_{x'}(\Lambda') - 1,\qquad
k_y(\Lambda')=k_y(\Lambda).
\end{equation}
It follows from equation~\eqref{eq:masLov} for $\Lambda'$
and equations~\eqref{eq:case11}, \eqref{eq:case12}, 
and~\eqref{eq:case13} that
\begin{eqnarray*}
\mu(\Lambda) 
&=& 
\mu(\Lambda')-1 \\
&=& 
2k_{x'}(\Lambda')+2k_y(\Lambda')
+ \frac{\eps_{x'}(\Lambda')-\eps_y(\Lambda')}{2}
- 1 \\
&=&
2k_{x'}(\Lambda')+2k_y(\Lambda')
+ \frac{-1-\eps_y(\Lambda)}{2}
- 1 \\
&=&
2k_{x'}(\Lambda')+2k_y(\Lambda')
+ \frac{1-\eps_y(\Lambda)}{2}
- 2  \\
&=&
2k_x(\Lambda)+2k_y(\Lambda)
+ \frac{\eps_x(\Lambda)-\eps_y(\Lambda)}{2}.
\end{eqnarray*}
This proves Step~4.

\medskip\noindent{\bf Step~5.}
{\it Assume~\eqref{eq:AB} and $\#(B\cap\R)>2$, 
follow the arc of $B$, starting at $x$, 
and let~$x'$ be the next intersection 
point with $\R$. Assume $x<x'<y$, 
denote by $B'$ the arc in $B$ from $x'$ to $y$, 
and let $A':=[x',y]$ (see Figure~\ref{fig:maslov5}).  
If the $(\alpha,\beta)$-trace
$\Lambda'$ with boundary $\p\Lambda'=(x',y,A',B')$ 
satisfies~\eqref{eq:masLov} so does $\Lambda$.}
\begin{figure}[htp] 
\centering 
\includegraphics[scale=0.4]{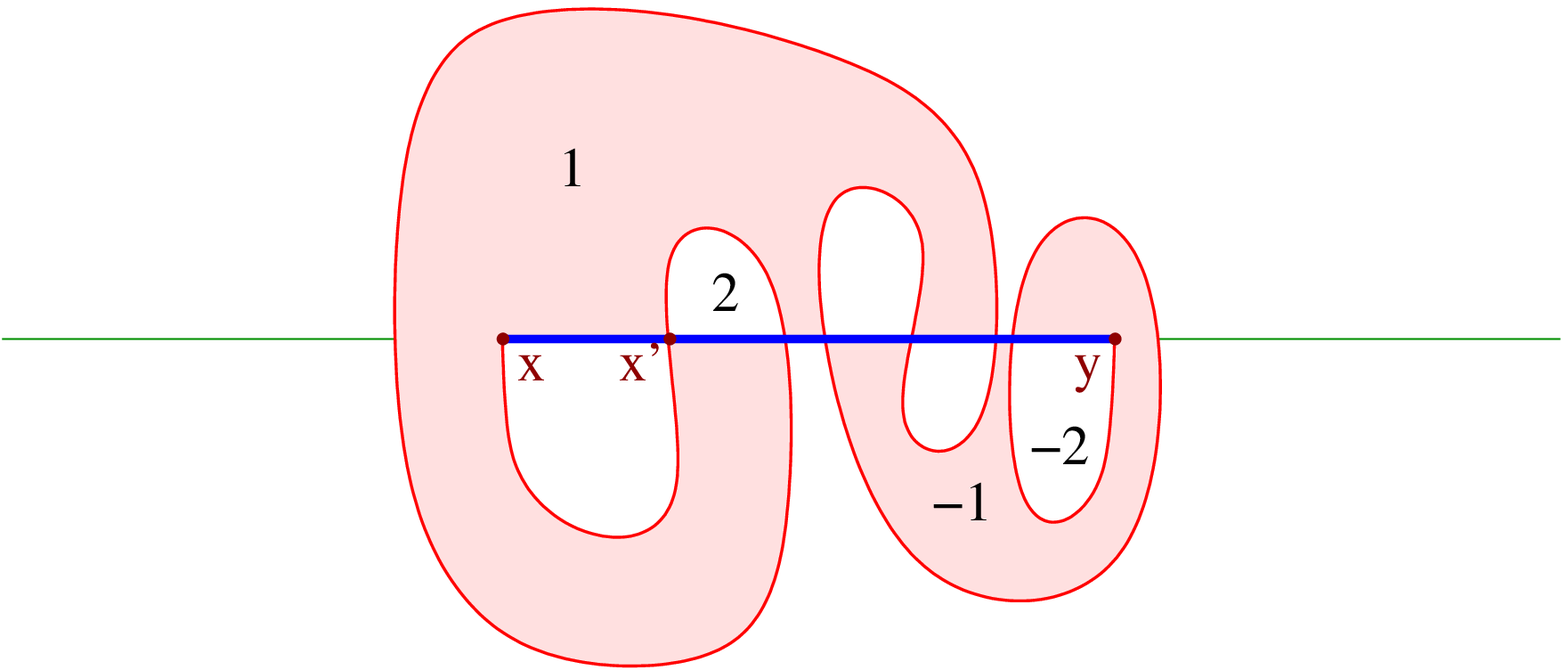} 
\caption{{Maslov index and catenation: $x<x'<y$.}}
\label{fig:maslov5}      
\end{figure}

\medskip\noindent
By Step~2 we may assume $\eps_x(\Lambda)=1$.
Since the Maslov index of the arc in $B$ 
from $x$ to $x'$ is $-1$, we have 
\begin{equation}\label{eq:case21}
\mu(\Lambda) = \mu(\Lambda')+1.
\end{equation}
Since the orientations of $A'$ and $B'$ 
agree with those of $A$ and $B$ we have
\begin{equation}\label{eq:case22}
\eps_{x'}(\Lambda') = \eps_{x'}(\Lambda)=-1,\qquad
\eps_y(\Lambda')=\eps_y(\Lambda).
\end{equation}
Now let $x<x_1<x_2<\cdots<x_m<x'$ be the intersection
points of $\R$ and $B$ in the interval $(x,x')$
and let $\eps_i\in\{-1,+1\}$ be the intersection
index of $\R$ and $B$ at $x_i$.  
Since the interval $[x,x']$ in $A$ and the arc in $B$
from $x$ to~$x'$ bound an open half disc, every subarc of $B$ 
in this half disc must enter and exit through the open 
interval $(x,x')$.  Hence the intersections 
indices of $\R$ and $B$ at the points $x_1,\dots,x_m$ 
cancel in pairs and thus
$$
\sum_{i=1}^m\eps_i=0. 
$$
Since $k_{x'}(\Lambda')$ is the sum of the intersection
indices of $\R$ and $B'$ at all points to the left of $x'$
we obtain
\begin{equation}\label{eq:case23}
k_{x'}(\Lambda') 
= k_x(\Lambda) + \sum_{i=1}^m\eps_i
= k_x(\Lambda),\qquad
k_y(\Lambda')=k_y(\Lambda).
\end{equation}
It follows from equation~\eqref{eq:masLov} for $\Lambda'$
and equations~\eqref{eq:case21}, \eqref{eq:case22}, 
and~\eqref{eq:case23} that
\begin{eqnarray*}
\mu(\Lambda) 
&=& 
\mu(\Lambda')+1 \\
&=& 
2k_{x'}(\Lambda')+2k_y(\Lambda')
+ \frac{\eps_{x'}(\Lambda')-\eps_y(\Lambda')}{2}
+ 1 \\
&=&
2k_x(\Lambda)+2k_y(\Lambda)
+ \frac{-1-\eps_y(\Lambda)}{2}
+ 1 \\
&=&
2k_x(\Lambda)+2k_y(\Lambda)
+ \frac{\eps_x(\Lambda)-\eps_y(\Lambda)}{2}.
\end{eqnarray*}
This proves Step~5.

\medskip\noindent{\bf Step~6.}
{\it Assume~\eqref{eq:AB} and $\#(B\cap\R)>2$, 
follow the arc of $B$, starting at $x$, 
and let~$y'$ be the next intersection point with $\R$. 
Assume $y'>y$. Denote by $B'$ the arc in $B$ from $y$ to $y'$, 
and let $A':=[y,y']$ (see Figure~\ref{fig:maslov6}).
If the $(\alpha,\beta)$-trace 
$\Lambda'$ with boundary $\p\Lambda'=(y,y',A',B')$
satisfies~\eqref{eq:masLov} so does $\Lambda$.}
\begin{figure}[htp] 
\centering 
\includegraphics[scale=0.4]{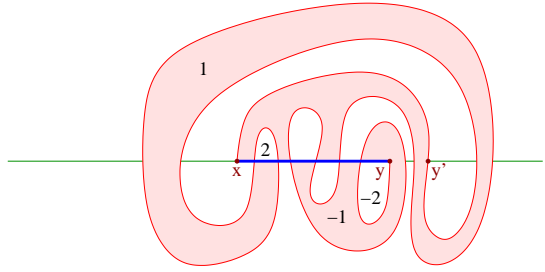} 
\caption{{Maslov index and catenation: $x<y<y'$.}}
\label{fig:maslov6}      
\end{figure}

\medskip\noindent
By Step~2 we may assume $\eps_x(\Lambda)=1$.
Since the orientation of $B'$ from $y$ to $y'$ is opposite
to the orientation of $B$ and the Maslov index 
of the arc in $B$ from $x$ to $y'$ is $-1$, we have 
\begin{equation}\label{eq:case31}
\mu(\Lambda) = 1-\mu(\Lambda').
\end{equation}
Using again the fact that the orientation of $B'$ is opposite
to the orientation of $B$ we have
\begin{equation}\label{eq:case32}
\eps_y(\Lambda') = -\eps_y(\Lambda),\qquad
\eps_{y'}(\Lambda')=-\eps_{y'}(\Lambda) = 1.
\end{equation}
Now let $x_1<x_2<\cdots<x_m$ be all 
intersection points of $\R$ and $B$ 
and let $\eps_i\in\{-1,+1\}$ be 
the intersection index of $\R$ and $B$ at $x_i$.  
Choose 
$$
j<k<\ell
$$ 
such that
$$
x_j=x,\qquad x_k=y,\qquad x_\ell=y'.
$$
Then 
$$
\eps_j=\eps_x(\Lambda)=1,\qquad
\eps_k=\eps_y(\Lambda),\qquad
\eps_\ell=\eps_{y'}(\Lambda)=-1, 
$$
and 
$$
k_x(\Lambda) = \sum_{i<j}\eps_i,\qquad
k_y(\Lambda) = - \sum_{i>k}\eps_i.
$$
For $i\ne j$ the intersection index
of $\R$ and $B'$ at $x_i$ is $-\eps_i$.
Moreover, $k_y(\Lambda')$ is the sum of the 
intersection indices of $\R$ and $B'$ 
at all points to the left of $y$ 
and $k_{y'}(\Lambda')$ is minus the sum of the 
intersection indices of $\R$ and $B'$ 
at all points to the right of $y'$. Hence
$$
k_y(\Lambda') 
= - \sum_{i<j}\eps_i
  - \sum_{j<i<k}\eps_i,\qquad
k_{y'}(\Lambda') 
= \sum_{i>\ell}\eps_i.
$$
We claim that 
\begin{equation}\label{eq:case33}
k_{y'}(\Lambda') + k_x(\Lambda) = 0,\qquad
k_y(\Lambda') + k_y(\Lambda) = \frac{1+\eps_y(\Lambda)}{2}.
\end{equation}
To see this, note that the value of the winding number 
$\w$ slightly to the left of $x$ agrees with the value
of $\w$ slightly to the right of $y'$, and hence
$$
0 = \sum_{i<j}\eps_i + \sum_{i>\ell}\eps_i 
= k_x(\Lambda) + k_{y'}(\Lambda').
$$
This proves the first equation in~\eqref{eq:case33}. 
To prove the second equation in~\eqref{eq:case33} 
we observe that
$$
\sum_{i=1}^m\eps_i = \frac{\eps_x(\Lambda)+\eps_y(\Lambda)}{2}
$$
and hence
\begin{eqnarray*}
k_y(\Lambda') + k_y(\Lambda)
&=&
- \sum_{i<j}\eps_i 
- \sum_{j<i<k}\eps_i
- \sum_{i>k}\eps_i  \\
&=&
\eps_j + \eps_k - \sum_{i=1}^m\eps_i \\
&=&
\eps_x(\Lambda) + \eps_y(\Lambda) - \sum_{i=1}^m\eps_i \\
&=& 
\frac{\eps_x(\Lambda)+\eps_y(\Lambda)}{2} \\
&=& 
\frac{1+\eps_y(\Lambda)}{2}.
\end{eqnarray*}
This proves the second equation in~\eqref{eq:case33}.

It follows from equation~\eqref{eq:masLov} for $\Lambda'$
and equations~\eqref{eq:case31}, \eqref{eq:case32}, 
and~\eqref{eq:case33} that
\begin{eqnarray*}
\mu(\Lambda) 
&=& 
1-\mu(\Lambda') \\
&=& 
1 - 2k_y(\Lambda')-2k_{y'}(\Lambda')
- \frac{\eps_y(\Lambda')-\eps_{y'}(\Lambda')}{2} \\
&=& 
1 - 2k_y(\Lambda')-2k_{y'}(\Lambda')
- \frac{-\eps_y(\Lambda)-1}{2} \\
&=& 
2k_y(\Lambda) - \eps_y(\Lambda) + 2k_x(\Lambda)
+ \frac{1+\eps_y(\Lambda)}{2} \\
&=&
2k_x(\Lambda)+2k_y(\Lambda) 
+ \frac{1-\eps_y(\Lambda)}{2}.
\end{eqnarray*}
Here the first equality follows from~\eqref{eq:case31}, 
the second equality follows from~\eqref{eq:masLov} for $\Lambda'$,
the third equality follows from~\eqref{eq:case32}, 
and the fourth equality follows from~\eqref{eq:case33}.
This proves Step~6.

\medskip\noindent{\bf Step~7.}
{\it The trace formula~\eqref{eq:maslov} holds when $\Sigma=\C$
and $\Lambda$ satisfies the arc condition.}

\medskip\noindent 
It follows from Steps~3-6 by induction that equation~\eqref{eq:masLov}
holds for every $(\alpha,\beta)$-trace $\Lambda=(x,y,\w)$
whose boundary $\p\Lambda=(x,y,A,B)$
satisfies~\eqref{eq:AB}.  Hence Step~7 follows from Step~1.

\medskip
Next we drop the hypothesis that $\Lambda$ satisfies the 
arc condition and extend the result to planar surfaces.  
This requires a further three steps.

\medskip\noindent{\bf Step~8.}
{\it The trace formula~\eqref{eq:maslov} holds when $\Sigma=\C$ and $x=y$.}

\medskip\noindent
Under these hypotheses $\nu_\alpha:=\p\w|_{\alpha\setminus\beta}$ 
and $\nu_\beta:=-\p\w|_{\beta\setminus\alpha}$ are constant.
There are four cases.

\smallskip\noindent{\bf Case~1.} 
{\it $\alpha$ is an embedded circle
and $\beta$ is not an embedded circle.} 
In this case we have $\nu_\beta\equiv0$ and $B=\{x\}$.
Moroeover, $\alpha$ is the boundary of a unique disc $\Delta_\alpha$
and we assume that $\alpha$ is oriented as the boundary
of $\Delta_\alpha$. Then the path $\gamma_\alpha:[0,1]\to\Sigma$
in Definition~\ref{def:trace-V} satisfies 
$\gamma_\alpha(0)=\gamma_\alpha(1)=x$ and 
is homotopic to $\nu_\alpha\alpha$. Hence
$$
m_x(\Lambda) = m_y(\Lambda) = 2\nu_\alpha
= \mu(\Lambda).
$$
Here the last equation follows from the fact 
that $\Lambda$ can be obtained as the catenation 
of $\nu_\alpha$ copies of the disc $\Delta_\alpha$.

\smallskip\noindent{\bf Case~2.} 
{\it $\alpha$ is not an embedded circle
and $\beta$ is an embedded circle.} 
This follows from Case~1 by interchanging $\alpha$ and $\beta$.

\smallskip\noindent{\bf Case~3.} 
{\it $\alpha$ and $\beta$ are embedded circles.} 
In this case there is a unique pair of embedded 
discs $\Delta_\alpha$ and $\Delta_\beta$ with boundaries 
$\alpha$ and $\beta$, respectively. Orient $\alpha$
and $\beta$ as the boundaries of these discs. 
Then, for every $z\in\Sigma\setminus\alpha\cup\beta$, 
we have
$$
\w(z)=\left\{\begin{array}{ll}
\nu_\alpha-\nu_\beta,&\mbox{for } 
z\in\Delta_\alpha\cap\Delta_\beta,\\
\nu_\alpha,&\mbox{for } 
z\in\Delta_\alpha\setminus\overline{\Delta}_\beta,\\
-\nu_\beta,&
\mbox{for } z\in\Delta_\beta\setminus\overline{\Delta}_\alpha,\\
0,& \mbox{for } 
z\in\Sigma\setminus\overline{\Delta}_\alpha\cup\overline{\Delta}_\beta.
\end{array}\right.
$$
Hence
$$
m_x(\Lambda) = m_y(\Lambda) 
= 2\nu_\alpha-2\nu_\beta = \mu(\Lambda).
$$
Here the last equation follows from the fact 
$\Lambda$ can be obtained as the catenation 
of $\nu_\alpha$ copies of the disc $\Delta_\alpha$ 
(with the orientation inherited from $\Sigma$)
and $\nu_\beta$ copies of $-\Delta_\beta$
(with the opposite orientation).

\smallskip\noindent{\bf Case~4.} 
{\it Neither $\alpha$ nor $\beta$ is an embedded circle.} 
Under this hypothesis we have $\nu_\alpha=\nu_\beta=0$.
Hence it follows from Theorem~\ref{thm:trace}
that $\w=0$ and $\Lambda=\Lambda_u$ for the constant 
map $u\equiv x\in\cD(x,x)$.  Thus
$$
m_x(\Lambda)=m_y(\Lambda)=\mu(\Lambda)=0.
$$
This proves Step~8.

\medskip\noindent{\bf Step~9.}
{\it The trace formula~\eqref{eq:maslov} holds when $\Sigma=\C$.}

\medskip\noindent
By Step~8, it suffices to assume $x\ne y$.
It follows from Theorem~\ref{thm:trace} that every 
$u\in\cD(x,y)$ is homotopic to a catentation $u = u_0\#v$,
where $u_0\in\cD(x,y)$ satisfies the arc condition 
and $v\in\cD(y,y)$.  Hence it follows from Steps~7 and~8 
that
\begin{eqnarray*}
\mu(\Lambda_u)
&=&
\mu(\Lambda_{u_0}) + \mu(\Lambda_v) \\
&=&
\frac{m_x(\Lambda_{u_0})+m_y(\Lambda_{u_0})}{2} + m_y(\Lambda_v) \\
&=&
\frac{m_x(\Lambda_u)+m_y(\Lambda_u)}{2}.
\end{eqnarray*}
Here the last equation follows from the fact that 
$
\w_u=\w_{u_0}+\w_v
$ 
and hence 
$
m_z(\Lambda_u)=m_z(\Lambda_{u_0})+m_z(\Lambda_v)
$
for every $z\in\alpha\cap\beta$. This proves Step~9.  

\medskip\noindent{\bf Step~10.}
{\it The trace formula~\eqref{eq:maslov} holds when $\Sigma$ is planar.}

\medskip\noindent
Choose an element $u\in\cD(x,y)$ such that $\Lambda_u=\Lambda$.
Modifying $\alpha$ and $\beta$ on the complement of $u(\D)$,
if necessary, we may assume without loss of generality that
$\alpha$ and $\beta$ are embedded circles.
Let $\iota:\Sigma\to\C$ be an orientation preserving embedding.
Then $\iota_*\Lambda := \Lambda_{\iota\circ u}$ is an 
$(\iota(\alpha),\iota(\beta))$-trace in $\C$ and hence 
satisfies the trace formula~\eqref{eq:maslov} by Step~9.  
Since $m_{\iota(x)}(\iota_*\Lambda)=m_x(\Lambda)$, 
$m_{\iota(y)}(\iota_*\Lambda)=m_y(\Lambda)$,
and $\mu(\iota_*\Lambda)=\mu(\Lambda)$ it follows 
that $\Lambda$ also satisfies the trace formula. 
This proves Step~10 and Proposition~\ref{prop:maslovC}
\end{proof}

\begin{remark}\label{rmk:maslov}\rm
Let $\Lambda=(x,y,A,B)$ be an $(\alpha,\beta)$-trace in $\C$ 
as in Step~1 in the proof of Theorem~\ref{thm:maslov}.
Thus $x<y$ are real numbers, $A$ is the interval $[x,y]$, 
and $B$ is an embedded arc with endpoints $x,y$
which is oriented from $x$ to $y$ and is 
transverse to $\R$. Thus
$
Z:=B\cap\R
$
is a finite set.  Define a map
$$
f:Z\setminus\{y\}\to Z\setminus\{x\}
$$
as follows. Given $z\in Z\setminus\{y\}$ walk 
along $B$ towards $y$ and let $f(z)$ be the next 
intersection point with $\R$. This map is bijective.  
Now let $I$ be any of the three open intervals 
$(-\infty,x)$, $(x,y)$, $(y,\infty)$.  
Any arc in $B$ from $z$ to $f(z)$ with
both endpoints in the same interval $I$ 
can be removed by an isotopy of $B$ which
does not pass through $x,y$. 
Call $\Lambda$ a {\bf reduced $(\alpha,\beta)$-trace} 
\index{reduced $(\alpha,\beta)$-trace}\index{trace!reduced}
if $z\in I$ implies $f(z)\notin I$ for each of the three intervals. 
Then every $(\alpha,\beta)$-trace is isotopic to a reduced 
$(\alpha,\beta')$-trace and the isotopy does not affect the 
numbers $\mu,k_x,k_y,\eps_x,\eps_y$.
\begin{figure}[htp]
\centering 
\includegraphics[scale=0.6]{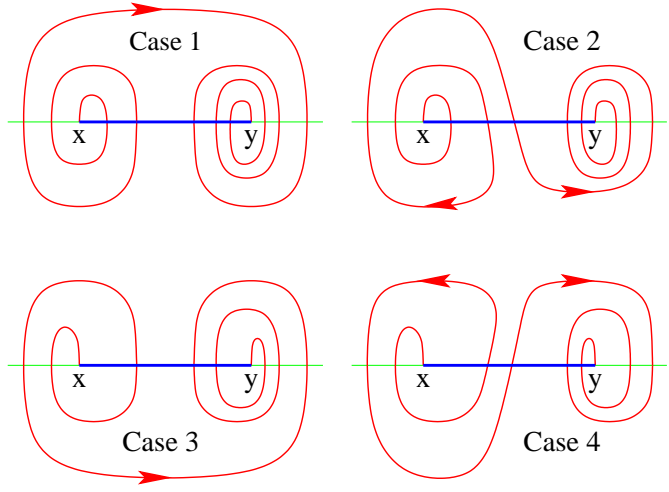} 
\caption{{Reduced $(\alpha,\beta)$-traces in $\C$.}}
\label{fig:maslov}
\end{figure}

Let $Z^+$ (respectively $Z^-$) denote the set 
of all points $z\in Z=B\cap\R$ where the positive 
tangent vectors in $T_zB$ point up (respectively down).
One can prove that every reduced $(\alpha,\beta)$-trace
satisfies one of the following conditions. 

\smallskip
\centerline{{\bf Case~1:} 
If $z\in Z^+\setminus\{y\}$ then $f(z)>z$.\qquad
{\bf Case~2:} $Z^-\subset[x,y]$.}

\centerline{{\bf Case~3:} 
If $z\in Z^-\setminus\{y\}$ then $f(z)>z$.\qquad
{\bf Case~4:} $Z^+\subset[x,y]$.}

\smallskip\noindent
(Examples with $\eps_x=1$ and $\eps_y=-1$ 
are depicted in Figure~\ref{fig:maslov}.)
One can then show directly that the reduced $(\alpha,\beta)$-traces
satisfy equation~\eqref{eq:masLov}.  This gives rise 
to an alternative proof of Proposition~\ref{prop:maslovC} 
via case distinction. 
\end{remark}

\begin{proof}[Proof of Theorem~\ref{thm:maslov} 
in the Simply Connected Case]
\phantomsection\label{proof:maslov1} 
If $\Sigma$ is diffeomorphic to the $2$-plane
the result has been established 
in Proposition~\ref{prop:maslovC}.
Hence assume 
$$
\Sigma=S^2. 
$$
Let $u\in\cD(x,y)$.  
If $u$ is not surjective the assertion follows from the case 
of the complex plane (Proposition~\ref{prop:maslovC})
via stereographic projection.  Hence assume $u$ is 
surjective and choose a regular value 
$z\in S^2\setminus(\alpha\cup\beta)$ of $u$.  
Denote 
$$
u^{-1}(z) = \{z_1,\dots,z_k\}.
$$
For $i=1,\dots,k$ let $\eps_i=\pm1$ according to whether or not
the differential $du(z_i):\C\to T_z\Sigma$ is orientation preserving.
Choose an open disc $\Delta\subset S^2$ centered at $z$
such that 
$$
\bar\Delta\cap(\alpha\cup\beta)=\emptyset
$$
and $u^{-1}(\Delta)$ is a union of open neighborhoods
$U_i\subset\D$ of $z_i$ with disjoint closures such that 
$$
u|_{U_i}:U_i\to\Delta
$$ 
is a diffeomorphism for each $i$
which extends to a neighborhood of $\bar U_i$. 
Now choose a continuous map $u':\D\to S^2$ which agrees
with $u$ on $\D\setminus\bigcup_iU_i$ and restricts to a 
diffeomorphism from $\bar U_i$ to $S^2\setminus\Delta$ 
for each $i$. Then $z$ does not belong to the image of $u'$
and hence the trace formula~\eqref{eq:maslov} holds for $u'$
(after smoothing along the boundaries $\p U_i$).  
Moreover, the diffeomorphism 
$$
u'|_{\bar U_i}:\bar U_i\to S^2\setminus\Delta
$$
is orientation preserving if and only if $\eps_i=-1$.
Hence
\begin{equation*}
\begin{split}
\mu(\Lambda_u) &= \mu(\Lambda_{u'}) + 4\sum_{i=1}^k\eps_i,\\
m_x(\Lambda_u) &= m_x(\Lambda_{u'}) + 4\sum_{i=1}^k\eps_i,\\
m_y(\Lambda_u) &= m_y(\Lambda_{u'}) + 4\sum_{i=1}^k\eps_i.
\end{split}
\end{equation*}
By Proposition~\ref{prop:maslovC} the trace formula~\eqref{eq:maslov} 
holds for $\Lambda_{u'}$ and hence it also holds for $\Lambda_u$.  
This proves Theorem~\ref{thm:maslov} when $\Sigma$ is simply
connected.
\end{proof}


\section{The Non Simply Connected Case}\label{sec:LIFT}

The key step for extending Proposition~\ref{prop:maslovC} 
to non-simply connected two-manifolds is the next
result about lifts to the universal cover. 

\begin{proposition}\label{prop:gm}
Suppose $\Sigma$ is not diffeomorphic to the $2$-sphere.
Let ${\Lambda=(x,y,\w)}$ be an 
$(\alpha,\beta)$-trace and ${\pi:\C\to\Sigma}$ be a 
universal covering. Denote by $\Gamma\subset\Diff(\C)$
the group of deck transformations. 
Choose an element ${\tx\in\pi^{-1}(x)}$ and let $\talpha$ 
and $\tbeta$ be the lifts of $\alpha$ and $\beta$ 
through~$\tx$. Let $\tLa=(\tx,\ty,\widetilde\w)$ be the lift
of $\Lambda$ with left endpoint~$\tx$. Then
$\tLa$ satisfies the {\bf cancellation formula}
\index{cancellation formula}
\index{Viterbo--Maslov index!cancellation formula} 
\begin{equation}\label{eq:gm}
m_{g\tx}(\tLa) + m_{g^{-1}\ty}(\tLa) 
= 0
\end{equation}
for every $g\in\Gamma\setminus\{\id\}$.
\rm(Proof on page~\pageref{proof:gm}.)
\end{proposition}

\begin{lemma}[{\bf Annulus Reduction}]
\label{le:gm} \index{annulus reduction}
Suppose $\Sigma$ is not diffeomorphic to the $2$-sphere. 
Let $\Lambda$, $\pi$, $\Gamma$, $\tLa$
be as in Proposition~\ref{prop:gm}.  If 
\begin{equation}\label{eq:GM}
m_{g\tx}(\tLa) + m_{g^{-1}\ty}(\tLa) 
= m_{g^{-1}\tx}(\tLa) + m_{g\ty}(\tLa)
\end{equation}
for all $g\in\Gamma\setminus\{\id\}$ 
then the cancellation formula~\eqref{eq:gm} holds 
for all $g\in\Gamma\setminus\{\id\}$.
\end{lemma}

\begin{proof}
If~\eqref{eq:gm} does not hold then there is 
a deck transformation $h\in\Gamma\setminus\{\id\}$ 
such that
$
m_{h\tx}(\tLa) + m_{h^{-1}\ty}(\tLa) \ne 0.
$
Since there can only be finitely many such 
$h\in\Gamma\setminus\{\id\}$, 
there is an integer $k\ge 1$ such that 
$m_{h^k\tx}(\tLa)+m_{h^{-k}\ty}(\tLa)\ne0$
and $m_{h^\ell\tx}(\tLa)+m_{h^{-\ell}\ty}(\tLa)=0$ 
for every integer $\ell>k$. Define $g:=h^k$.  Then
\begin{equation}\label{eq:GM2}
m_{g\tx}(\tLa) + m_{g^{-1}\ty}(\tLa) \ne 0
\end{equation}
and $m_{g^k\tx}(\tLa) + m_{g^{-k}\ty}(\tLa)=0$
for every integer $k\in\Z\setminus\{-1,0,1\}$.
Define
$$
\Sigma_0:=\C/\Gamma_0,\qquad 
\Gamma_0 := \left\{g^k\,|\,k\in\Z\right\}.
$$
Then $\Sigma_0$ is diffeomorphic to the annulus.
Let $\pi_0:\C\to\Sigma_0$ be the obvious 
projection, define $\alpha_0:=\pi_0(\talpha)$, $\beta_0:=\pi_0(\tbeta)$,
and let $\Lambda_0:=(x_0,y_0,\w_0)$
be the $(\alpha_0,\beta_0)$-trace in $\Sigma_0$ with 
$x_0:=\pi_0(\tx)$, $y_0:=\pi_0(\ty)$, and
$$
\w_0(z_0) := \sum_{\tz\in\pi_0^{-1}(z_0)}\widetilde\w(\tz),\qquad 
z_0\in\Sigma_0\setminus(\alpha_0\cup\beta_0).
$$
Then 
\begin{equation*}
\begin{split}
m_{x_0}(\Lambda_0) 
&= m_\tx(\tLa) + \sum_{k\in\Z\setminus\{0\}}m_{g^k\tx}(\tLa),\\
m_{y_0}(\Lambda_0) 
&= m_\ty(\tLa) + \sum_{k\in\Z\setminus\{0\}}m_{g^{-k}\ty}(\tLa).
\end{split}
\end{equation*}
By Proposition~\ref{prop:maslovC} both $\tLa$ 
and $\Lambda_0$ satisfy the trace formula~\eqref{eq:maslov} 
and they have the same Viterbo--Maslov index. Hence
\begin{eqnarray*}
0 
&=&
\mu(\Lambda_0)-\mu(\tLa) \\
&=&
\frac{m_{x_0}(\Lambda_0) + m_{y_0}(\Lambda_0)}{2}
- \frac{m_\tx(\tLa)+m_\ty(\tLa)}{2} \\
&=&
\frac12
\sum_{k\ne 0}\left(m_{g^k\tx}(\tLa)+m_{g^{-k}\ty}(\tLa)\right) \\
&=&
m_{g\tx}(\tLa)+m_{g^{-1}\ty}(\tLa).
\end{eqnarray*}
Here the last equation follows from~\eqref{eq:GM}.
This contradicts~\eqref{eq:GM2} and proves Lemma~\ref{le:gm}.
\end{proof}

\bigbreak

\begin{lemma}\label{le:AB}
Suppose $\Sigma$ is not diffeomorphic to the $2$-sphere.
Let $\Lambda$, $\pi$, $\Gamma$, $\tLa$ be as in 
Proposition~\ref{prop:gm} and denote 
$\nu_\talpha:=\p\widetilde\w|_{\talpha\setminus\tbeta}$
and $\nu_\tbeta:=-\p\widetilde\w|_{\tbeta\setminus\talpha}$.
Choose smooth paths
$$
\gamma_\talpha:[0,1]\to\talpha,\qquad \gamma_\tbeta:[0,1]\to\tbeta
$$
from $\gamma_\talpha(0)=\gamma_\tbeta(0)=\tx$ to
$\gamma_\talpha(1)=\gamma_\tbeta(1)=\ty$ such that
$\gamma_\talpha$ is an immersion when $\nu_\talpha\not\equiv0$
and constant when $\nu_\talpha\equiv0$, the same holds
for $\gamma_\tbeta$, and
\begin{equation*}
\begin{split}
\nu_\talpha(\tz) =\deg(\gamma_\talpha,\tz)&\quad
\mbox{for}\quad \tz\in\talpha\setminus\{\tx,\ty\},\\
\nu_\tbeta(\tz) =\deg(\gamma_\tbeta,\tz)&\quad
\mbox{for}\quad\tz\in\tbeta\setminus\{\tx,\ty\}.
\end{split}
\end{equation*}
Define 
$$
\tA := \gamma_\talpha([0,1]),\qquad \tB := \gamma_\tbeta([0,1]).
$$
Then, for every $g\in\Gamma$, we have
\begin{equation}\label{eq:gxyA1}
g\tx\in\tA
\qquad\iff\qquad 
g^{-1}\ty\in\tA,
\end{equation}
\begin{equation}\label{eq:gxyA2}
g\tx\notin\tA\;\;\mbox{ and }\;\;g\ty\notin\tA 
\qquad\iff\qquad 
\tA\cap g\tA=\emptyset,
\end{equation}
\begin{equation}\label{eq:gxyA3}
g\tx\in\tA\;\;\mbox{ and }\;\;g\ty\in\tA 
\qquad\iff\qquad g=\id.
\end{equation}
The same holds with $\tA$ replaced by $\tB$.
\end{lemma}

\begin{proof}
If $\alpha$ is a contractible embedded circle or 
not an embedded circle at all we have 
$\tA\cap g\tA=\emptyset$ whenever $g\ne\id$
and this implies~\eqref{eq:gxyA1}, \eqref{eq:gxyA2}
and~\eqref{eq:gxyA3}. Hence assume 
$\alpha$ is a noncontractible embedded circle.
Then we may also assume, 
without loss of generality, 
that $\pi(\R)=\alpha$, 
the map $\tz\mapsto\tz+1$ is a deck transformation, 
$\pi$ maps the interval $[0,1)$ bijectively onto~$\alpha$,
and $\tx,\ty\in\R=\talpha$ with $\tx<\ty$.  
Thus $\tA=[\tx,\ty]$ and, for every $k\in\Z$, 
$$
\tx+k\in[\tx,\ty]
\quad\iff\quad
0\le k\le \ty-\tx
\quad\iff\quad
\ty-k\in[\tx,\ty].
$$
Similarly, we have 
$$
\tx+k,\ty+k\notin[\tx,\ty]
\quad\iff\quad
[\tx+k,\ty+k]\cap[\tx,\ty]=\emptyset
$$
and
$$
\tx+k,\ty+k\in[\tx,\ty]
\quad\iff\quad
[\tx+k,\ty+k]\subset[\tx,\ty]
\quad\iff\quad k=0.
$$
This proves~\eqref{eq:gxyA1}, \eqref{eq:gxyA2}, 
and~\eqref{eq:gxyA3} for the deck transformation 
$\tz\mapsto\tz+k$.  If $g$ is any other deck transformation, 
then we have 
$
\talpha\cap g\talpha=\emptyset
$ 
and so~\eqref{eq:gxyA1}, \eqref{eq:gxyA2}, and~\eqref{eq:gxyA3}
are trivially satisfied. This proves Lemma~\ref{le:AB}.
\end{proof}

\begin{lemma}[{\bf Winding Number Comparison}]\label{le:wg} 
\index{winding number comparison}
Suppose $\Sigma$ is not diffeomorphic to the $2$-sphere.
Let $\Lambda$, $\pi$, $\Gamma$, $\tLa$
be as in Proposition~\ref{prop:gm}, and let
$\tA,\tB\subset\C$ be as in Lemma~\ref{le:AB}.
Then the following holds.

\smallskip\noindent{\bf (i)}
Equation~\eqref{eq:GM} holds for every $g\in\Gamma$
that satisfies $g\tx,g\ty\notin\tA\cup\tB$.

\smallskip\noindent{\bf (ii)}
If $\Lambda$ satisfies the arc condition then it also satisfies
the cancellation formula~\eqref{eq:gm}
for every $g\in\Gamma\setminus\{\id\}$.
\end{lemma}

\begin{proof}
We prove~(i).  Let $g\in\Gamma$ such that $g\tx,g\ty\notin\tA\cup\tB$ 
and let $\gamma_\talpha,\gamma_\tbeta$ be as in Lemma~\ref{le:AB}.  
Then $\widetilde\w(\tz)$ is the winding number of the loop 
$\gamma_\talpha-\gamma_\tbeta$ about the point 
$\tz\in\C\setminus(\tA\cup\tB)$. Moreover, the paths 
$g\gamma_\talpha,g\gamma_\tbeta:[0,1]\to\C$ 
connect the points $g\tx,g\ty\in\C\setminus(\tA\cup\tB)$. Hence
$$
\widetilde\w(g\ty)-\widetilde\w(g\tx)
= (\gamma_\talpha-\gamma_\tbeta)\cdot g\gamma_\talpha
= (\gamma_\talpha-\gamma_\tbeta)\cdot g\gamma_\tbeta.
$$
Similarly with $g$ replaced by $g^{-1}$. 
Moreover, it follows from Lemma~\ref{le:AB}, that
$$
\tA\cap g\tA =\emptyset,\qquad \tB\cap g^{-1}\tB=\emptyset.
$$
Hence
\begin{eqnarray*}
\widetilde\w(g\ty) - \widetilde\w(g\tx)
&=&  
\left(\gamma_\talpha-\gamma_\tbeta\right)\cdot g\gamma_\talpha \\
&=&  
g\gamma_\talpha\cdot\gamma_\tbeta \\
&=&
\gamma_\talpha\cdot g^{-1}\gamma_\tbeta \\
&=&
\left(\gamma_\talpha-\gamma_\tbeta\right)\cdot g^{-1}\gamma_\tbeta \\
&=&
\widetilde\w(g^{-1}\ty) - \widetilde\w(g^{-1}\tx)
\end{eqnarray*}
Here we have used the fact that every $g\in\Gamma$
is an orientation preserving diffeomorphism
of $\C$. Thus we have proved that
$$
\widetilde\w(g\tx) + \widetilde\w(g^{-1}\ty)
= \widetilde\w(g\ty) + \widetilde\w(g^{-1}\tx).
$$
Since $g\tx,g\ty\notin\tA\cup\tB$, 
we have 
$$
m_{g\tx}(\tLa)=4\widetilde\w(g\tx),\qquad
m_{g^{-1}\ty}(\tLa)=4\widetilde\w(g^{-1}\ty),
$$
and the same identities hold with $g$ replaced by $g^{-1}$.
This proves~(i).

We prove~(ii).
If $\Lambda$ satisfies the arc condition then $g\tA\cap\tA=\emptyset$
and $g\tB\cap\tB=\emptyset$ for every $g\in\Gamma\setminus\{\id\}$.
In particular, for every $g\in\Gamma\setminus\{\id\}$,
we have $g\tx,g\ty\notin\tA\cup\tB$ and hence~\eqref{eq:GM}
holds by~(i).  Hence it follows from Lemma~\ref{le:gm} 
that the cancellation formula~\eqref{eq:gm} holds 
for every $g\in\Gamma\setminus\{\id\}$.
This proves Lemma~\ref{le:wg}. 
\end{proof}

The next lemma deals with $(\alpha,\beta)$-traces
connecting a point $x\in\alpha\cap\beta$ to itself. 
An example on the annulus is depicted in 
Figure~\ref{fig:annulus3}.

\begin{lemma}[{\bf Isotopy Argument}]\label{le:zero} \index{isotopy argument}
Suppose $\Sigma$ is not diffeomorphic to the $2$-sphere.
Let $\Lambda$, $\pi$, $\Gamma$, $\tLa$
be as in Proposition~\ref{prop:gm}.  Suppose that there 
is a deck transformation $g_0\in\Gamma\setminus\{\id\}$
such that $\ty = g_0\tx$.  Then $\Lambda$ has Viterbo--Maslov 
index zero and $m_{g\tx}(\tLa)=0$ for every 
$g\in\Gamma\setminus\{\id,g_0\}$.
\end{lemma}

\begin{figure}[htp]
\centering 
\includegraphics[scale=0.275]{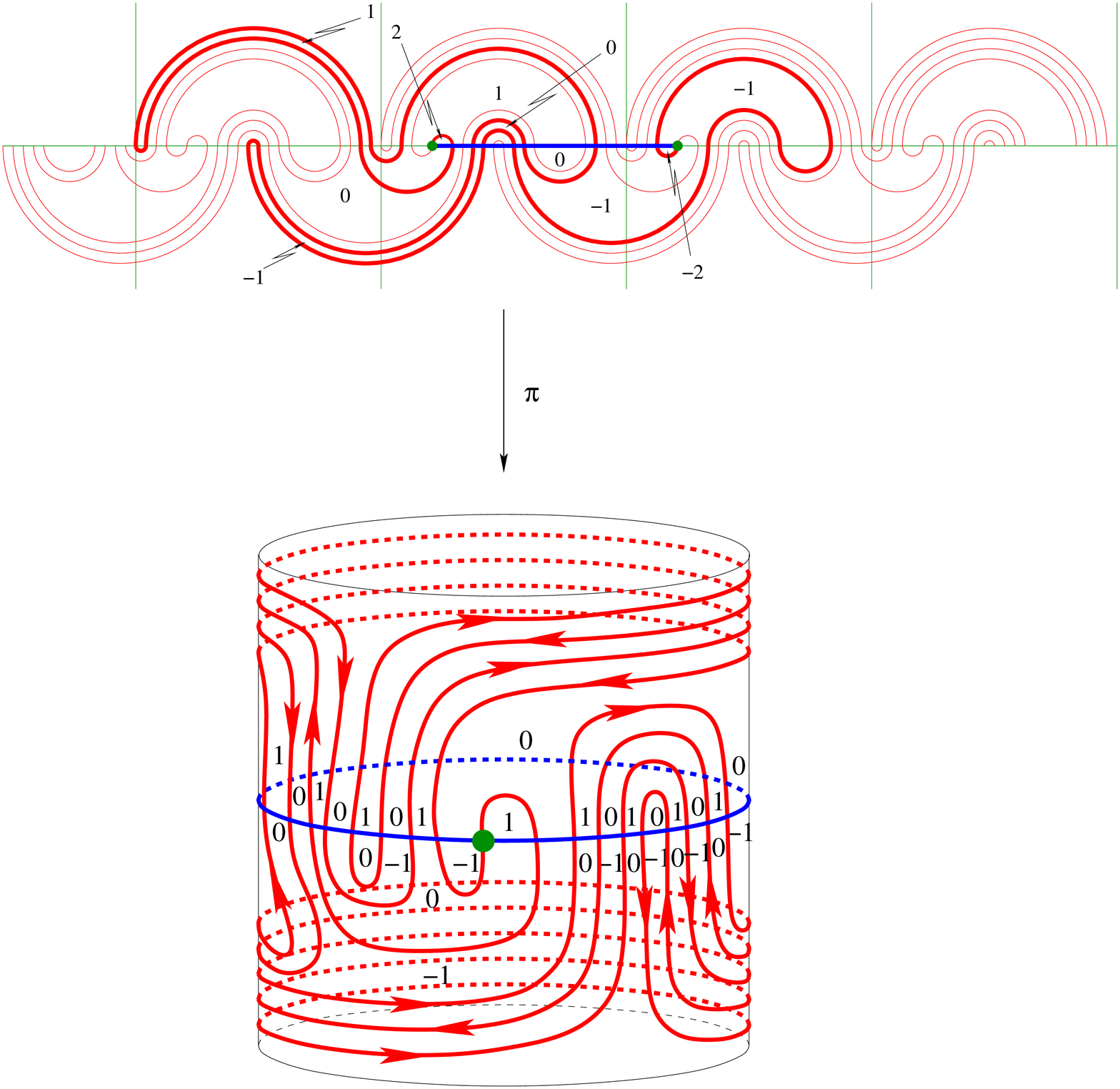} 
\caption{{An $(\alpha,\beta)$-trace on the annulus 
with $x=y$.}}
\label{fig:annulus3}
\end{figure}

\begin{proof}
By hypothesis, we have $\talpha=g_0\talpha$
and $\tbeta=g_0\tbeta$.  Hence $\alpha$ and $\beta$
are noncontractible embedded circles and some iterate 
of $\alpha$ is homotopic to some iterate of $\beta$.
Hence, by Lemma~\ref{le:dbae3}, $\alpha$ must 
be homotopic to $\beta$ (with some orientation).
Hence we may assume, without loss of generality, 
that $\pi(\R)=\alpha$, the map $\tz\mapsto\tz+1$ is a deck transformation, 
$\pi$ maps the interval $[0,1)$ bijectively onto~$\alpha$,
$\R=\talpha$, $\tx=0\in\talpha\cap\tbeta$, $\tbeta=\tbeta+1$,
and that $\ty=\ell>0$ is an integer.  
Then $g_0$ is the translation 
$$
g_0(\tz)=\tz+\ell.
$$ 
Let $\tA:=[0,\ell]\subset\talpha$ and let
$\tB\subset\tbeta$ be the arc connecting $0$ to $\ell$.
Then, for $\tz\in\C\setminus(\tA\cup\tB)$, the integer
$\widetilde\w(\tz)$ is the winding number of $\tA-\tB$ about $\tz$.
Define the projection $\pi_0:\C\to\C$ by
$$
\pi_0(\tz):=e^{2\pi\i\tz/\ell},
$$
denote $\alpha_0:=\pi_0(\talpha)=S^1$ and $\beta_0:=\pi(\tbeta)$,
and let $\Lambda_0=(1,1,\w_0)$ be the induced
$(\alpha_0,\beta_0)$-trace in $\C$ with
$\w_0(z) := \sum_{\tz\in\pi^{-1}(z)}\widetilde\w(\tz)$.
Then $\Lambda_0$ satisfies the conditions of Step~8, Case~3 
in the proof of Proposition~\ref{prop:maslovC} and its boundary is given
by $\nu_{\alpha_0}=\p\w_0|_{\alpha_0\setminus\beta_0}\equiv1$ and 
$\nu_{\beta_0}=\p\w_0|_{\beta_0\setminus\alpha_0}\equiv1$.  
Hence $\Lambda_0$ and $\tLa$ have Viterbo--Maslov index zero.

It remains to prove that $m_{g\tx}(\tLa)=0$ for every 
$g\in\Gamma\setminus\{\id,g_0\}$.  To see this we use 
the fact that the embedded loops $\alpha$
and $\beta$ are homotopic with fixed endpoint $x$.
Hence, by a Theorem of Epstein, they are
isotopic with fixed basepoint $x$
(see~\cite[Theorem~4.1]{EPSTEIN}).
\phantomsection\label{EPSTEIN0}  
Thus there exists a smooth map 
$f:\R/\Z\times[0,1]\to\Sigma$ such that
$$
f(s,0)\in\alpha,\qquad f(s,1)\in\beta,\qquad
f(0,t)=x,
$$
for all $s\in\R/\Z$ and $t\in[0,1]$,
and the map $\R/\Z\to\Sigma:s\mapsto f(s,t)$ 
is an embedding for every $t\in[0,1]$.
Lift this homotopy to the universal cover 
to obtain a map $\tf:\R\times[0,1]\to\C$
such that $\pi\circ\tf = f$ and
$$
\tf(s,0)\in[0,1],\quad
\tf(s,1)\in\tB_1,\quad
\tf(0,t)=\tx,\quad
\tf(s+k,t)=\tf(s,t)+k
$$
for all $s,t\in[0,1]$ and $k\in\Z$.  
Here $\tB_1\subset\tB$ denotes the arc in $\tB$ from $0$ to~$1$.
Since the map $\R/\Z\to\Sigma:s\mapsto f(s,t)$ is injective
for every $t$, we have 
$$
g\tx\notin\{\tx,\tx+1,\dots,\tx+\ell\}\qquad\implies\qquad
g\tx \notin \tf([0,\ell]\times[0,1])
$$
for every every $g\in\Gamma$. Now choose a smooth map 
$\tu:\D\to\C$ with $\Lambda_\tu=\tLa$ (see Theorem~\ref{thm:trace}).  
Define the homotopy $F_\tu:[0,\ell]\times[0,1]\to\C$
by $F_\tu(s,t):=\tu(-\cos(\pi s/\ell),t\sin(\pi s/\ell))$.
Then, by Theorem~\ref{thm:trace}, 
$F_\tu$ is homotopic to $\tf|_{[0,\ell]\times[0,1]}$
subject to the boundary conditions $\tf(s,0)\in\talpha=\R$,
$\tf(s,1)\in\tbeta$, $\tf(0,t)=\tx$, $\tf(\ell,t)=\ty$.
Hence, for every $\tz\in\C\setminus(\talpha\cup\tbeta)$, we have 
$$
\widetilde\w(\tz) = \deg(\tu,z) = \deg(F_\tu,\tz) = \deg(\tf,\tz).
$$
In particular, choosing $\tz$ near $g\tx$, we find
$
m_{g\tx}(\tLa) = 4\deg(\tf,g\tx) = 0
$
for every $g\in\Gamma$ that is not one of the translations 
$\tz\mapsto\tz+k$ for $k=0,1,\dots,\ell$.  This proves the
assertion in the case $\ell=1$.  

\bigbreak

If $\ell>1$ it remains to prove $m_k(\tLa)=0$ for $k=1,\dots,\ell-1$.
To see this, let $\tA_1:=[0,1]$, $\tB_1\subset\tB$ 
be the arc from $0$ to $1$, $\widetilde\w_1(\tz)$ be the winding 
number of $\tA_1-\tB_1$ about $\tz\in\C\setminus(\tA_1\cup\tB_1)$,
and define 
$
\tLa_1 := (0,1,\widetilde\w_1).
$
Then, by what we have already proved, the 
$(\talpha,\tbeta)$-trace $\tLa_1$ satisfies $m_{g\tx}(\tLa_1)=0$
for every $g\in\Gamma$ other than the translations by $0$ or $1$.
In particular, we have $m_j(\tLa_1)=0$ for every $j\in\Z\setminus\{0,1\}$
and also $m_0(\tLa_1)+m_1(\tLa_1)=2\mu(\tLa_1)=0$.
Since $\widetilde\w(\tz)=\sum_{j=0}^{\ell-1}\widetilde\w_1(\tz-j)$
for $\tz\in\C\setminus(\tA\cup\tB)$, we obtain
$$
m_k(\tLa) = \sum_{j=0}^{\ell-1}m_{k-j}(\tLa_1) = 0
$$
for every $k\in\Z\setminus\{0,\ell\}$. 
This proves Lemma~\ref{le:zero}.
\end{proof}

The next example shows that Lemma~\ref{le:wg} cannot 
be strengthened to assert the identity $m_{g\tx}(\tLa)=0$ 
for every $g\in\Gamma$ with $g\tx,g\ty\notin\tA\cup\tB$.

\begin{example}\label{ex:annulus1}\rm
Figure~\ref{fig:annulus1} depicts an $(\alpha,\beta)$-trace 
$\Lambda=(x,y,\w)$ on the annulus $\Sigma=\C/\Z$ 
that has Viterbo--Maslov index one and satisfies the arc condition. 
The lift satisfies $m_\tx(\tLa)=-3$,  $m_{\tx+1}(\tLa)=4$, 
$m_{\ty}(\tLa)=5$, and $m_{\ty-1}(\tLa)=-4$. 
Thus $m_x(\Lambda)=m_y(\Lambda)=1$.
\begin{figure}[htp]
\centering 
\includegraphics[scale=0.4]{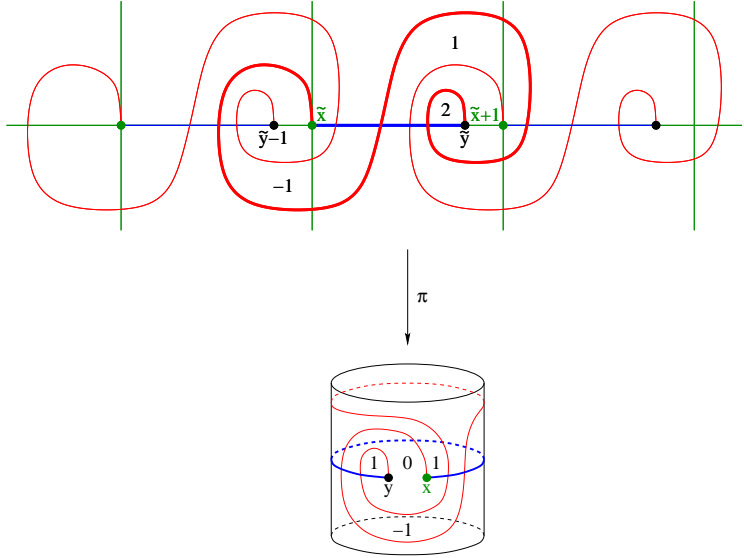} 
\caption{{An $(\alpha,\beta)$-trace on the annulus 
satisfying the arc condition.}}
\label{fig:annulus1}
\end{figure}
\end{example}

\begin{proof}[Proof of Proposition~\ref{prop:gm}] 
\phantomsection\label{proof:gm}
The proof has five steps.  

\medskip\noindent{\bf Step~1.}
{\it Let $\tA,\tB\subset\C$ be as in Lemma~\ref{le:AB}
and let $g\in\Gamma$ such that
$$
g\tx\in\tA\setminus\tB,\qquad g\ty\notin\tA\cup\tB.
$$ 
(An example is depicted in Figure~\ref{fig:torus}.)
Then~\eqref{eq:GM} holds.}

\begin{figure}[htp]
\centering 
\includegraphics[scale=0.2]{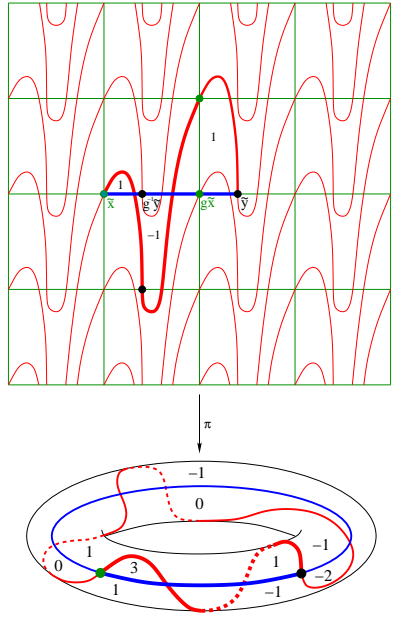} 
\caption{{An $(\alpha,\beta)$-trace on the torus 
not satisfying the arc condition.}}
\label{fig:torus}
\end{figure}

\medskip\noindent
The proof is a refinement of the winding number 
comparison argument in Lemma~\ref{le:wg}.
Since $g\tx\notin\tB$ we have $g\ne\id$ and, 
since $\tx,g\tx\in\tA\subset\talpha$,
it follows that $\alpha$ is a noncontractible embedded circle.
Hence we may choose the universal covering $\pi:\C\to\Sigma$
and the lifts $\talpha$, $\tbeta$, $\tLa$ such that $\pi(\R)=\alpha$, 
the map $\tz\mapsto\tz+1$ is a deck transformation, 
the projection $\pi$ maps the interval $[0,1)$ 
bijectively onto $\alpha$, and
$$
\talpha=\R,\qquad \tx=0\in\talpha\cap\tbeta,\qquad \ty>0.
$$
By hypothesis and Lemma~\ref{le:AB} 
there is an integer $k$ such that
$$
0<k<\ty,\qquad g\tx=k,\qquad g^{-1}\ty = \ty-k.
$$
Thus $g$ is the deck transformation $\tz\mapsto\tz+k$.

\smallbreak

Since $g\tx\notin\tB$ and $g\ty\notin\tB$ it follows from 
Lemma~\ref{le:AB} that $g^{-1}\ty\notin\tB$ 
and $g^{-1}\tx\notin\tB$ and hence, again
by Lemma~\ref{le:AB}, we have
$$
\tB\cap g\tB=\tB\cap g^{-1}\tB=\emptyset.
$$ 
With $\gamma_\talpha$ and $\gamma_\tbeta$ chosen as in 
Lemma~\ref{le:AB}, this implies
\begin{equation}\label{eq:zero}
\gamma_\tbeta\cdot(\gamma_\tbeta-k)
= (\gamma_\tbeta+k)\cdot\gamma_\tbeta
= 0.
\end{equation}
Since $k,-k,\ty+k,\ty-k\notin\tB$, there exists
a constant $\eps>0$ such that
$$
-\eps\le t\le\eps\qquad\implies\qquad
k+\i t,\;\;-k+\i t,\;\;
\ty-k+\i t,\;\;\ty+k+\i t\notin\tB.
$$
The paths $g\gamma_\talpha\pm\i\eps$ 
and $g\gamma_\tbeta\pm\i\eps$ both
connect the point $g\tx\pm\i\eps$ to $g\ty\pm\i\eps$.
Likewise, the paths $g^{-1}\gamma_\talpha\pm\i\eps$ 
and $g^{-1}\gamma_\tbeta\pm\i\eps$ both
connect the point $g^{-1}\tx\pm\i\eps$ to $g^{-1}\ty\pm\i\eps$.
Hence
\begin{eqnarray*}
\widetilde\w(g\ty\pm\i\eps)
- \widetilde\w(g\tx\pm\i\eps)
&=&
(\gamma_\talpha-\gamma_\tbeta)\cdot(g\gamma_\talpha\pm\i\eps) \\
&=&
(\gamma_\talpha-\gamma_\tbeta)\cdot(\gamma_\talpha+k\pm\i\eps) \\
&=&
(\gamma_\talpha+k\pm\i\eps)\cdot\gamma_\tbeta \\
&=&
\gamma_\talpha\cdot(\gamma_\tbeta-k\mp\i\eps) \\
&=&
(\gamma_\talpha-\gamma_\tbeta)\cdot(\gamma_\tbeta-k\mp\i\eps) \\
&=&
(\gamma_\talpha-\gamma_\tbeta)\cdot(g^{-1}\gamma_\tbeta\mp\i\eps) \\
&=&
\widetilde\w(g^{-1}\ty\mp\i\eps)
- \widetilde\w(g^{-1}\tx\mp\i\eps).
\end{eqnarray*}
Here the last but one equation follows from~\eqref{eq:zero}.
Thus we have proved
\begin{equation}\label{eq:step1}
\begin{split}
\widetilde\w(g\tx+\i\eps) + \widetilde\w(g^{-1}\ty-\i\eps)
&= \widetilde\w(g^{-1}\tx-\i\eps) + \widetilde\w(g\ty+\i\eps), \\
\widetilde\w(g\tx-\i\eps) + \widetilde\w(g^{-1}\ty+\i\eps)
&= \widetilde\w(g^{-1}\tx+\i\eps) + \widetilde\w(g\ty-\i\eps).
\end{split}
\end{equation}
Since
\begin{equation*}
\begin{split}
m_{g\tx}(\tLa) &= 2\widetilde\w(g\tx+\i\eps) 
+ 2\widetilde\w(g\tx-\i\eps), \\
m_{g\ty}(\tLa) &= 2\widetilde\w(g\ty+\i\eps) 
+ 2\widetilde\w(g\ty-\i\eps), \\
m_{g^{-1}\tx}(\tLa) &= 2\widetilde\w(g^{-1}\tx+\i\eps) 
+ 2\widetilde\w(g^{-1}\tx-\i\eps),\\
m_{g^{-1}\ty}(\tLa) &= 2\widetilde\w(g^{-1}\ty+\i\eps) 
+ 2\widetilde\w(g^{-1}\ty-\i\eps),
\end{split}
\end{equation*}
Step~1 follows by taking the sum 
of the two equations in~\eqref{eq:step1}.

\medskip\noindent{\bf Step~2.}
{\it Let $\tA,\tB\subset\C$ be as in Lemma~\ref{le:AB}
and let $g\in\Gamma$.  Suppose that either
$g\tx,g\ty\notin\tA$ or $g\tx,g\ty\notin\tB$.
Then~\eqref{eq:GM} holds.}

\medskip\noindent
If $g\tx,g\ty\notin\tA\cup\tB$ the assertion 
follows from Lemma~\ref{le:wg}.
If $g\tx\in\tA\setminus\tB$ and $g\ty\notin\tA\cup\tB$
the assertion follows from Step~1.
If $g\tx\notin\tA\cup\tB$ and $g\ty\in\tA\setminus\tB$
the assertion follows from Step~1 by interchanging $\tx$
and $\ty$.   Namely, \eqref{eq:GM} 
holds for $\tLa$ if and only if it holds for
the $(\talpha,\tbeta)$-trace
$
-\tLa:=(\ty,\tx,-\widetilde\w).
$
This covers the case $g\tx,g\ty\notin\tB$. 
If $g\tx,g\ty\notin\tA$ the assertion follows 
by interchanging $\tA$ and $\tB$.
Namely, \eqref{eq:GM} 
holds for $\tLa$ if and only if it holds for
the $(\tbeta,\talpha)$-trace
$
\tLa^*:=(\tx,\ty,-\widetilde\w).
$
This proves Step~2.

\medskip\noindent{\bf Step~3.}
{\it Let $\tA,\tB\subset\C$ be as in Lemma~\ref{le:AB}
and let $g\in\Gamma$ such that
$$
g\tx\in\tA\setminus\tB,\qquad g\ty\in\tB\setminus\tA.
$$ 
(An example is depicted in Figure~\ref{fig:annulus2}.)
Then the cancellation formula~\eqref{eq:gm} 
holds for $g$ and $g^{-1}$.}

\begin{figure}[htp]
\centering 
\includegraphics[scale=0.28]{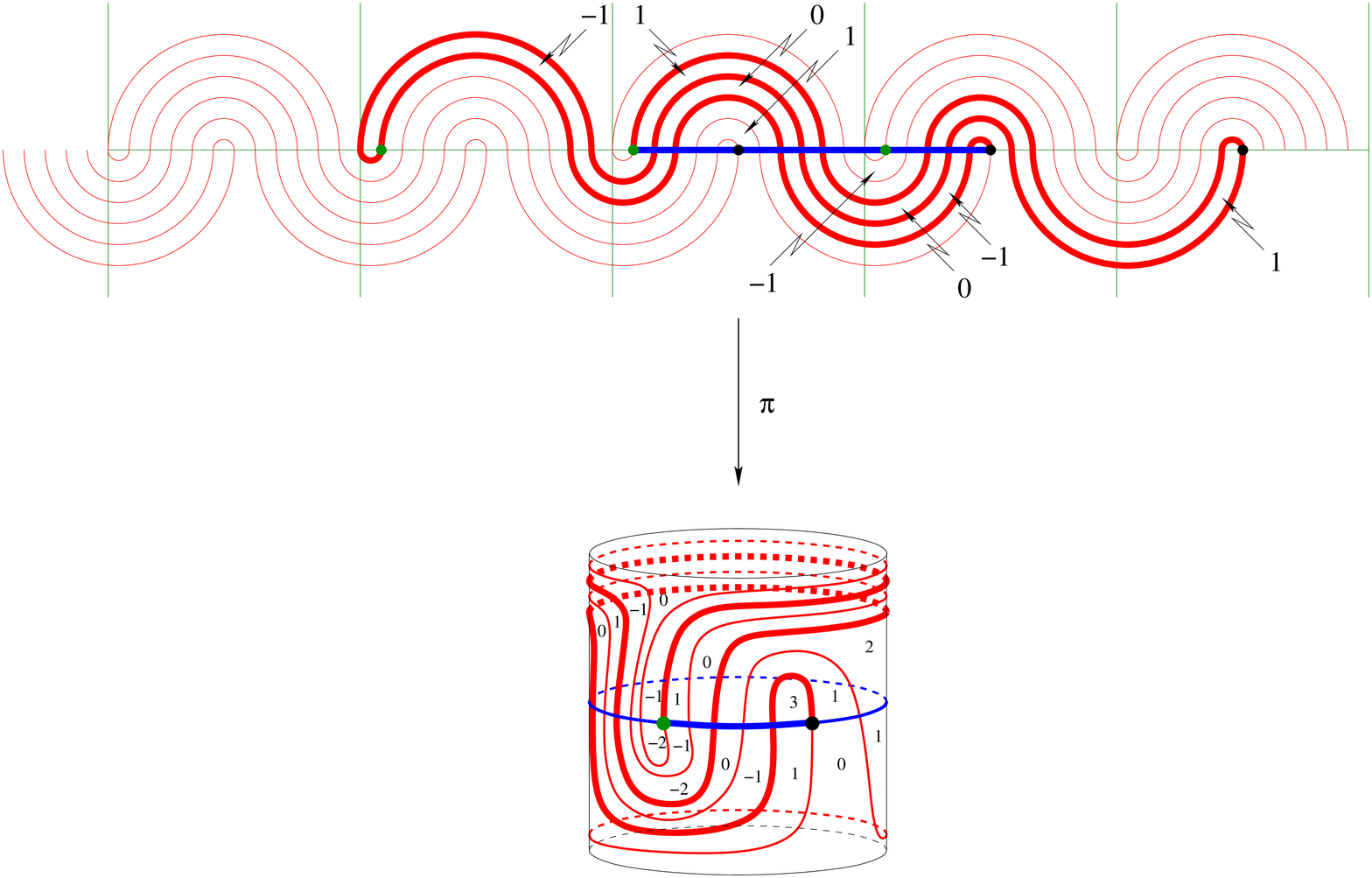} 
\caption{{An $(\alpha,\beta)$-trace on the annulus
with $g\tx\in\tA$ and $g\ty\in\tB$.}}
\label{fig:annulus2}
\end{figure}

\medskip\noindent
Since $g\tx\notin\tB$ 
(and $g\ty\notin\tA$) we have $g\ne\id$ and, 
since $\tx,g\tx\in\tA\subset\talpha$ and $\ty,g\ty\in\tB\subset\tbeta$,
it follows that $g\talpha=\talpha$ and $g\tbeta=\tbeta$.
Hence $\alpha$ and $\beta$ are noncontractible embedded circles
and some iterate of $\alpha$ is homotopic to some iterate of $\beta$.
So $\alpha$ is homotopic to~$\beta$ (with some orientation), 
by Lemma~\ref{le:dbae3}.

Choose the universal covering $\pi:\C\to\Sigma$
and the lifts ${\talpha,\tbeta,\tLa}$ such that $\pi(\R)=\alpha$, 
the map $\tz\mapsto\tz+1$ is a deck transformation, 
$\pi$~maps the interval $[0,1)$ bijectively onto $\alpha$, and 
$$
\talpha=\R,\qquad 
\tx=0\in\talpha\cap\tbeta,\qquad 
\ty>0.
$$
Thus $\tA=[0,\ty]$ is the arc in $\talpha$ from $0$ to $\ty$ 
and $\tB$ is the arc in $\tbeta$ from $0$ to $\ty$. 
Moreover, since $\alpha$ is homotopic to $\beta$, we have
$$
\tbeta=\tbeta+1
$$ 
and the arc in $\tbeta$ from $0$ to $1$ is a fundamental domain 
for $\beta$.  Since $g\talpha=\talpha$, the deck transformation $g$ 
is given by $\tz\mapsto\tz+\ell$ for some integer $\ell$.  
Since $g\tx\in\tA\setminus\tB$ and $g\ty\in\tB$,
we have $g^{-1}\ty\notin\tB$ and $g^{-1}\tx\in\tB$ 
by Lemma~\ref{le:AB}.  Hence
$$
0< \ell<\ty,\qquad \ell\notin\tB,\qquad \ty+\ell\in\tB,\qquad
\ty-\ell\notin\tB,\qquad -\ell\in\tB.
$$
This shows that, walking along $\tbeta$ from $0$ to $\ty$
(traversing $\tB$) one encounters some negative integer 
and therefore no positive integers.  Hence
$$
\tA\cap\Z=\left\{0,1,2,\cdots,k_\tA\right\},\qquad
\tB\cap\Z=\left\{0,-1,-2,\cdots,-k_\tB\right\},
$$
where $k_\tA$ is the number of fundamental domains
of $\talpha$ contained in $\tA$ and $k_\tB$ is the number of 
fundamental domains of $\tbeta$ contained in $\tB$ 
(see Figure~\ref{fig:annulus2}).
For $0\le k\le k_\tA$ let $\tA_k\subset\talpha$ and 
$\tB_k\subset\tbeta$ be the arcs from $0$ to $\ty-k$.
Thus $\tA_k$ is obtained from $\tA$ by removing $k$ 
fundamental domains at the end, while $\tB_k$
is obtained from $\tB$ by attaching $k$ fundamental 
domains at the end.  Consider the $(\talpha,\tbeta)$-trace
$$
\tLa_k:=(0,\ty-k,\widetilde\w_k),\qquad
\p\tLa_k:=(0,\ty-k,\tA_k,\tB_k),
$$
where $\widetilde\w_k:\C\setminus(\tA_k\cup\tB_k)\to\Z$ 
is the winding number of $\tA_k-\tB_k$. Then
$$
\tB_k\cap\Z = \left\{0,-1,-2,\cdots,-k_\tB-k\right\}
$$
and $\tLa_0=\tLa$.  We prove that, for each $k$,
the $(\talpha,\tbeta)$-trace $\tLa_k$ satisfies 
\begin{equation}\label{eq:jmk}
m_j(\tLa_k) + m_{\ty-k-j}(\tLa_k) = 0\qquad
\forall\,j\in\Z\setminus\{0\}.
\end{equation} 
If $\ty$ is an integer, then~\eqref{eq:jmk} 
follows from Lemma~\ref{le:zero}. 
Hence we may assume that $\ty$ is not an integer. 

\bigbreak

We prove equation~\eqref{eq:jmk} by reverse induction on $k$.
First let $k=k_\tA$.  Then we have $j,\ty-k+j\notin\tA_k$ 
for every $j\in\N$.  Hence it follows from Step~2 that
\begin{equation}\label{eq:JMK}
m_j(\tLa_k) + m_{\ty-k-j}(\tLa_k) = 
m_{-j}(\tLa_k) + m_{\ty-k+j}(\tLa)
\qquad\forall\,j\in\N.
\end{equation} 
Thus we can apply Lemma~\ref{le:gm}
to the projection of $\tLa_k$ to the quotient $\C/\Z$.
Hence $\tLa_k$ satisfies~\eqref{eq:jmk}.

Now fix an integer $k\in\{0,1,\dots,k_\tA-1\}$ and suppose,
by induction, that $\tLa_{k+1}$ satisfies~\eqref{eq:jmk}. 
Denote by $\tA'\subset\talpha$ and $\tB'\subset\tbeta$
the arcs from $\ty-k-1$ to $1$, and by
$\tA''\subset\talpha$ and $\tB''\subset\tbeta$
the arcs from $1$ to $\ty-k$.  
Then $\tLa_k$ is the catenation 
of the $(\talpha,\tbeta)$-traces
\begin{equation*}
\begin{split}
\tLa_{k+1} := (0,\ty-k-1,\widetilde\w_{k+1}),\quad
&\p\tLa_{k+1} = (0,\ty-k-1,\tA_{k+1},\tB_{k+1}),\\
\tLa' := (\ty-k-1,1,\widetilde\w'),\quad
&\p\tLa' = (\ty-k-1,1,\tA',\tB'),\\
\tLa'' :=  (1,\ty-k,\widetilde\w''),\quad
&\p\tLa'' =  (1,\ty-k,\tA'',\tB'').
\end{split}
\end{equation*}
Here $\widetilde\w'(\tz)$ is the winding number of 
the loop $\tA'-\tB'$ about $\tz\in\C\setminus(\tA'\cup\tB')$ 
and simiarly for $\widetilde\w''$. 
Note that $\tLa''$ is the shift of $\tLa_{k+1}$ by $1$. 
The catenation of $\tLa_{k+1}$ and $\tLa'$ 
is the $(\talpha,\tbeta)$-trace from $0$ to $1$.
Hence it has Viterbo--Maslov index zero, 
by Lemma~\ref{le:zero}, and satisfies
\begin{equation}\label{eq:step3a}
m_j(\tLa_{k+1})+m_j(\tLa')=0\qquad
\forall j\in\Z\setminus\{0,1\}.
\end{equation}
Since the catenation of $\tLa'$ and $\tLa''$ is the 
$(\talpha,\tbeta)$-trace from $\ty-k-1$ to $\ty-k$,
it also has Viterbo--Maslov index zero and satisfies
\begin{equation}\label{eq:step3b}
m_{\ty-k-j}(\tLa')+m_{\ty-k-j}(\tLa'') = 0
\qquad\forall j\in\Z\setminus\{0,1\}. 
\end{equation}
Moreover, by the induction hypothesis, we have 
\begin{equation}\label{eq:step3c}
m_j(\tLa_{k+1})+m_{\ty-k-1-j}(\tLa_{k+1}) = 0
\qquad\forall j\in\Z\setminus\{0\}. 
\end{equation}
Combining the equations~\eqref{eq:step3a}, \eqref{eq:step3b}, \eqref{eq:step3c} 
we find that, for $j\in\Z\setminus\{0,1\}$, 
\begin{eqnarray*}
m_j(\tLa_k)+m_{\ty-k-j}(\tLa_k) 
&=&
m_j(\tLa_{k+1})+m_j(\tLa')+m_j(\tLa'') \\
&&
+\,m_{\ty-k-j}(\tLa_{k+1}) + m_{\ty-k-j}(\tLa') +m_{\ty-k-j}(\tLa'') \\
&=&
m_j(\tLa_{k+1})+m_j(\tLa') \\
&&
+\, m_{\ty-k-j}(\tLa') + m_{\ty-k-j}(\tLa'')  \\
&&
+\, m_{j-1}(\tLa_{k+1}) + m_{\ty-k-j}(\tLa_{k+1})  \\
&=&
0.
\end{eqnarray*}
For $j=1$ we obtain
\begin{eqnarray*}
m_1(\tLa_k)+m_{\ty-k-1}(\tLa_k) 
&=&
m_1(\tLa_{k+1})+m_1(\tLa')+m_1(\tLa'') \\
&&
+\,m_{\ty-k-1}(\tLa_{k+1}) + m_{\ty-k-1}(\tLa') +m_{\ty-k-1}(\tLa'') \\
&=&
m_1(\tLa_{k+1}) + m_{\ty-k-2}(\tLa_{k+1})  \\
&&
+\, m_0(\tLa_{k+1}) + m_{\ty-k-1}(\tLa_{k+1}) \\
&&
+\, m_{\ty-k-1}(\tLa') + m_1(\tLa')  \\
&=&
2\mu(\tLa_{k+1}) + 2\mu(\tLa') \\
&=&
0.
\end{eqnarray*}
Here the last but one equation follows from 
equation~\eqref{eq:step3c} and Proposition~\ref{prop:maslovC},
and the last equation follows from Lemma~\ref{le:zero}.
Hence $\tLa_k$ satisfies~\eqref{eq:jmk}.
This completes the induction argument for the proof of Step~3.

\medskip\noindent{\bf Step~4.}
{\it Let $\tA,\tB\subset\C$ be as in Lemma~\ref{le:AB}
and let $g\in\Gamma$ such that
$$
g\tx\in\tA\cap\tB,\qquad g\ty\notin\tA\cup\tB.
$$ 
Then the cancellation formula~\eqref{eq:gm} holds for $g$ and $g^{-1}$.}

\medskip\noindent
The proof is by induction and catenation 
based on Step~2 and Lemma~\ref{le:zero}.  
Since $g\ty\notin\tA\cup\tB$ we have $g\ne\id$.
Since $g\tx\in\tA\cap\tB$ we have $\talpha=g\talpha$ and $\tbeta=g\tbeta$.  
Hence $\alpha$ and $\beta$ are noncontractible 
embedded circles, and they are homotopic 
(with some orientation) by Lemma~\ref{le:dbae3}.
Thus we may choose $\pi:\C\to\Sigma$, $\talpha$, $\tbeta$, $\tLa$ 
as in Step~3.  By hypothesis there is an integer $k\in\tA\cap\tB$. 
Hence $\tA$ and $\tB$ do not contain any negative integers.
Choose $k_\tA,k_\tB\in\N$ such that
$$
\tA\cap\Z=\left\{0,1,\dots,k_\tA\right\},\qquad
\tB\cap\Z=\left\{0,1,\dots,k_\tB\right\}.
$$
Assume without loss of generality that 
$$
k_\tA\le k_\tB.
$$ 
For $0\le k\le k_\tA$ denote by $\tA_k\subset\tA$ and 
$\tB_k\subset\tB$ the arcs from $0$ to $\ty-k$
and consider the $(\talpha,\tbeta)$-trace
$$
\tLa_k:=(0,\ty-k,\widetilde\w_k),\qquad
\p\tLa_k:=(0,\ty-k,\tA_k,\tB_k).
$$
In this case 
$$
\tB_k\cap\Z = \{0,1,\dots,k_\tB-k\}.
$$
As in Step~3, it follows by reverse induction on $k$ 
that $\tLa_k$ satisfies~\eqref{eq:jmk} for every $k$.   
We assume again that~$\ty$ is not an integer. 
(Otherwise~\eqref{eq:jmk} follows from Lemma~\ref{le:zero}).
If $k=k_\tA$ then $j,\ty-k+j\notin\tA_k$ for every $j\in\N$,
hence it follows from Step~2 that $\tLa_k$ satisfies~\eqref{eq:JMK},
and hence it follows from Lemma~\ref{le:gm} for the 
projection of $\tLa_k$ to the annulus $\C/\Z$ that $\tLa_k$
also satisfies~\eqref{eq:jmk}.  The induction step is verbatim the 
same as in Step~3 and will be omitted. 
This proves Step~4.

\medskip\noindent{\bf Step~5.}
{\it We prove Proposition~\ref{prop:gm}.}

\medskip\noindent
If both points $g\tx,g\ty$ are contained in $\tA$
(or in $\tB$) then $g=\id$ by Lemma~\ref{le:AB}, and in this 
case equation~\eqref{eq:GM} is a tautology.
If both points $g\tx,g\ty$ are not contained in $\tA\cup\tB$,
equation~\eqref{eq:GM} has been established in Lemma~\ref{le:wg}.
Moreover, we can interchange $\tx$ and $\ty$
or $\tA$ and $\tB$ as in the proof of Step~2.
Thus Steps~1 and~4 cover the case where precisely one 
of the points $g\tx,g\ty$ is contained in $\tA\cup\tB$
while Step~3 covers the case where $g\ne\id$ and 
both points $g\tx,g\ty$ are contained in $\tA\cup\tB$. 
This shows that equation~\eqref{eq:GM} holds for every 
$g\in\Gamma\setminus\{\id\}$. Hence, by Lemma~\ref{le:gm}, 
the cancellation formula~\eqref{eq:gm} holds 
for every $g\in\Gamma\setminus\{\id\}$.
This proves Proposition~\ref{prop:gm}.
\end{proof}

\begin{proof}[Proof of Theorem~\ref{thm:maslov} 
in the Non Simply Connected Case] 
\phantomsection\label{proof:maslov2}
Choose a universal covering 
$\pi:\C\to\Sigma$ and let $\Gamma$, $\talpha$, $\tbeta$, 
and $\tLa=(\tx,\ty,\widetilde\w)$ be as in Proposition~\ref{prop:gm}.  
Then
$$
m_x(\Lambda)+m_y(\Lambda)
-m_\tx(\tLa)-m_\ty(\tLa)
= \sum_{g\ne\id}\left(m_{g\tx}(\tLa) 
+ m_{g^{-1}\ty}(\tLa)\right)
= 0.
$$
Here the last equation follows from the cancellation formula 
in Proposition~\ref{prop:gm}. Hence, by 
Proposition~\ref{prop:maslovC}, we have
$$
\mu(\Lambda) = \mu(\tLa)
=\frac{m_\tx(\tLa)+m_\ty(\tLa)}{2}
=\frac{m_x(\Lambda)+m_y(\Lambda)}{2}.
$$
This proves the trace formula in the case where 
$\Sigma$ is not simply connected.
\end{proof}

\newpage
\part*{II. Combinatorial Lunes}
\addcontentsline{toc}{part}{Part II. Combinatorial Lunes}


\section{Lunes and Traces}\label{sec:LUNE}

We denote the universal covering of $\Sigma$ by\hypertarget{Part_II}{}
$$
\pi:\tSi\to\Sigma
$$
and, when $\Sigma$ is not diffeomorphic to the $2$-sphere,
we assume $\tSi=\C$. 

\begin{definition}[{\bf Smooth Lunes}]\label{def:lune}\rm 
\index{smooth!lune}\index{lune!smooth} 
Assume~\hyperlink{property_H}{(H)}.  
A {\bf smooth $(\alpha,\beta)$-lune} is an orientation 
preserving immersion $u:\D\to\Sigma$ such that
$$
u(\D\cap\R)\subset\alpha,\qquad
u(\D\cap S^1)\subset\beta,
$$
Three examples of smooth lunes are depicted 
in Figure~\ref{fig:3lunes}.
Two lunes are said to be {\bf equivalent} 
\index{equivalent!lunes}\index{lunes!equivalent}
iff there is an orientation preserving diffeomorphism
$\phi:\D\to\D$ such that
$$
\phi(-1)=-1,\qquad
\phi(1)=1,\qquad
u'=u\circ\phi.
$$
The equivalence class of $u$ is denoted by $[u]$.
That $u$ is an immersion means that $u$ is smooth and $du$
is injective in all of~$\D$, even at the corners $\pm1$.
The set $u(\D\cap\R)$
is called the {\bf bottom boundary} of the lune, \index{bottom boundary}
and the set $u(\D\cap S^1)$
is called the {\bf top boundary}. \index{top boundary}
The points  
$$
x=u(-1),\qquad y=u(1)
$$
are called respectively the left and right 
{\bf endpoints} of the lune. \index{endpoints of a lune}
\index{lune!endpoints of} The locally constant function
$$
\Sigma\setminus u(\p\D)\to\N: z\mapsto\# u^{-1}(z)
$$
is called the {\bf counting function} of the lune. 
\index{counting function of a lune}\index{lune!counting function of}
(This function is locally constant because a proper 
local homeomorphism is a covering projection.)  
A smooth lune is said to be {\bf embedded} 
iff the map $u$ is injective. 
\index{embedded lune}\index{lune!embedded}
These notions depend only on the equivalence 
class $[u]$ of the smooth lune $u$.
\end{definition}

Our objective is to characterize smooth lunes in terms 
of their boundary behavior, i.e.\ to say when a pair of immersions 
${u_\alpha:(\D\cap\R,-1,1)\to(\alpha,x,y)}$ and 
${u_\beta:(\D\cap S^1,-1,1)\to(\beta,x,y)}$ extends 
to a smooth $(\alpha,\beta)$-lune $u$.  
Recall the following definitions and theorems from~Part I.
\begin{figure}[htp] 
\centering 
\includegraphics[scale=.9]{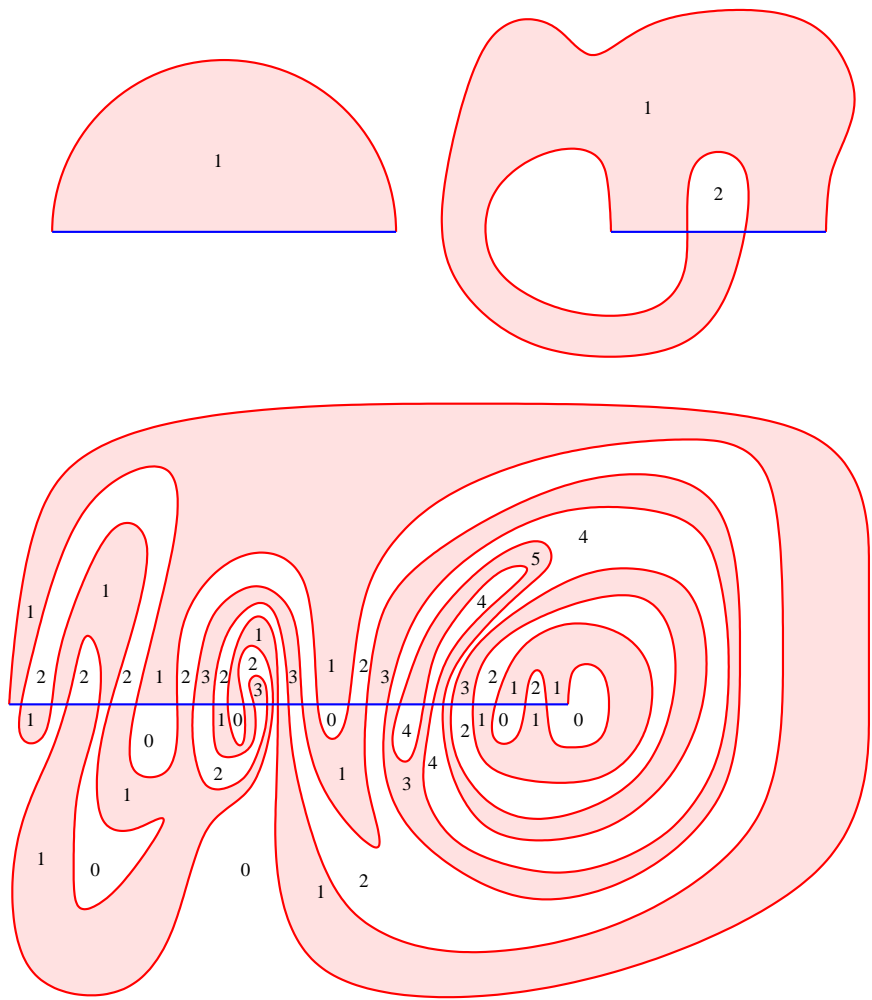}      
\caption{{Three lunes.}}\label{fig:3lunes}      
\end{figure}

\begin{definition}[{\bf Traces}]
\label{def:trace-CF}\rm \index{trace}
Assume~\hyperlink{property_H}{(H)}. 
An {\bf $(\alpha,\beta)$-trace} is a triple
$$
\Lambda=(x,y,\w)
$$
such that $x,y\in\alpha\cap\beta$ and 
$\w:\Sigma\setminus(\alpha\cup\beta)\to\Z$
is a locally constant function such that 
there exists a smooth map $u:\D\to\Sigma$ 
satisfying 
\begin{equation}\label{eq:bc}
u(\D\cap\R)\subset\alpha,\qquad u(\D\cap S^1)\subset\beta,
\end{equation}
\begin{equation}\label{eq:xy}
u(-1)=x,\qquad u(1)=y,
\end{equation}
\begin{equation}\label{eq:w-CF}
\w(z) = \deg(u,z), \qquad z\in\Sigma\setminus(\alpha\cup\beta).
\end{equation}
The $(\alpha,\beta)$-trace associated to a smooth map
$u:\D\to\Sigma$ satisfying~\eqref{eq:bc} is denoted by $\Lambda_u$. 

The {\bf boundary of an $(\alpha,\beta)$-trace} 
$\Lambda=(x,y,\w)$ is the triple 
$$
\p\Lambda:=(x,y,\p\w).  
$$
Here
$$
\p\w:(\alpha\setminus\beta)
\cup(\beta\setminus\alpha)\to\Z
$$
is the locally constant function
that assigns to $z\in\alpha\setminus\beta$ 
the value of $\w$ slightly to the left of $\alpha$  
minus the value of $\w$ slightly to the right of $\alpha$ 
near $z$, and to $z\in\beta\setminus\alpha$ 
the value of $\w$ slightly to the right of $\beta$ 
minus the value of $\w$ slightly to the left of $\beta$ 
near $z$.
\end{definition}

In Lemma~\ref{le:boundary} above it was shown that,  
if $\Lambda=(x,y,\w)$ is the $(\alpha,\beta)$-trace 
of a smooth map $u:\D\to\Sigma$ that satisfies~\eqref{eq:bc}, 
then $\p\Lambda_u=(x,y,\nu)$, where the function 
$\nu:=\p\w:(\alpha\setminus\beta)
\cup(\beta\setminus\alpha)\to\Z$
is given by
\begin{equation}\label{eq:nuu}
\nu(z) = 
\left\{\begin{array}{rl}
\deg(u\big|_{\p\D\cap\R\;\,}:\p\D\cap\R\;\to\alpha,z),&
\mbox{for } z\in\alpha\setminus\beta, \\
-\deg(u\big|_{\p\D\cap S^1}:\p\D\cap S^1\to\beta,z),&
\mbox{for } z\in\beta\setminus\alpha.
\end{array}\right.
\end{equation}
Here we orient the one-manifolds $\D\cap\R$ and 
$\D\cap S^1$ from $-1$ to $+1$. 
Moreover, in Theorem~\ref{thm:trace} above it was shown 
that the homotopy class of a smooth map $u:\D\to\Sigma$ 
satisfying the boundary condition~\eqref{eq:bc} 
is uniquely determined by its trace $\Lambda_u=(x,y,\w)$. 
If $\Sigma$ is not diffeomorphic to the $2$-sphere 
then its universal cover is diffeomorphic to the $2$-plane. 
In this situation it was also shown in~Theorem~\ref{thm:trace} 
that the homotopy class of $u$ and the degree function $\w$ 
are uniquely determined by the triple $\p\Lambda_u=(x,y,\nu)$.

\begin{remark}[{\bf The Viterbo--Maslov index}]
\label{rmk:VM}\rm 
\index{Viterbo--Maslov index}\index{Maslov index}
\index{Viterbo--Maslov index!of an $(\alpha,\beta)$-trace}
Let $\Lambda=(x,y,\w)$ be an $(\alpha,\beta)$-trace
and denote by $\mu(\Lambda)$ its Viterbo--Maslov index 
as defined in~\ref{def:maslov} above 
(see also~\cite{VITERBO}).\phantomsection\label{V3}
For $z\in\alpha\cap\beta$ let $m_z(\Lambda)$
be the sum of the four values of the function 
$\w$ encountered when walking along a small 
circle surrounding~$z$. In Theorem~\ref{thm:maslov}
it was shown that the Viterbo--Maslov index of $\Lambda$
is given by the trace\index{trace!formula}
formula\index{Viterbo--Maslov index!trace formula} 
\begin{equation}\label{eq:VM1}
\mu(\Lambda) = \frac{m_x(\Lambda)+m_y(\Lambda)}{2}.
\end{equation}
Let $\Lambda'=(y,z,\w')$ be another $(\alpha,\beta)$-trace.
The {\bf catenation of $\Lambda$ and $\Lambda'$} 
is defined by \index{catenation}
$$
\Lambda\#\Lambda' := (x,z,\w+\w').
$$
It is again an $(\alpha,\beta)$-trace
and has Viterbo--Maslov index
\begin{equation}\label{eq:VM2}
\mu(\Lambda\#\Lambda')=\mu(\Lambda)+\mu(\Lambda').
\end{equation}
For a proof see~\cite{VITERBO,RS2}.
\phantomsection\label{V4}\label{RS2d}
\end{remark}

\begin{definition}[{\bf Arc Condition}]
\label{def:arc-CF}\rm \index{arc condition}
Let $\Lambda=(x,y,\w)$ be an $(\alpha,\beta)$-trace
and 
$$
\nu_\alpha:=\p\w|_{\alpha\setminus\beta},\qquad
\nu_\beta:=-\p\w|_{\beta\setminus\alpha}.
$$
$\Lambda$ is said to satisfy the {\bf arc condition} if
\begin{equation}\label{eq:arc-CF}
x\ne y,\qquad \min\Abs{\nu_\alpha} = \min\Abs{\nu_\beta}=0.
\end{equation}
When $\Lambda$ satisfies the arc condition
there are arcs $A\subset\alpha$ and $B\subset\beta$ 
from $x$ to $y$ such that
\begin{equation}\label{eq:nuAB-CF}
\nu_\alpha(z) = \left\{\begin{array}{rl}
\pm1,&\mbox{if }z\in A,\\
0,&\mbox{if }z\in\alpha\setminus\overline A,
\end{array}\right.\;\;
\nu_\beta(z) = \left\{\begin{array}{rl}
\pm1,&\mbox{if }z\in B,\\
0,&\mbox{if }z\in\beta\setminus\overline B.
\end{array}\right.
\end{equation}
Here the plus sign is chosen iff the orientation of $A$ from 
$x$ to $y$ agrees with that of $\alpha$, respectively the 
orientation of $B$ from $x$ to $y$ agrees with that of~$\beta$.
In this situation the quadruple $(x,y,A,B)$ and the triple
$(x,y,\p\w)$ determine one another and we also write
$$
\p\Lambda = (x,y,A,B)
$$ 
for the boundary of $\Lambda$.
When $u:\D\to\Sigma$ is a smooth map satisfying~\eqref{eq:bc} 
and $\Lambda_u=(x,y,\w)$ satisfies the arc condition 
and $\p\Lambda_u=(x,y,A,B)$ then the path
$
s\mapsto u(-\cos(\pi s),0)
$
is homotopic in $\alpha$ 
to a path traversing $A$ and the path 
$
s\mapsto u(-\cos(\pi s),\sin(\pi s))
$
is homotopic in~$\beta$ to a path traversing $B$.
\end{definition}

\begin{theorem}\label{thm:arc}
Assume~\hyperlink{property_H}{(H)}. If $u:\D\to\Sigma$ is a smooth 
$(\alpha,\beta)$-lune then its $(\alpha,\beta)$-trace 
$\Lambda_u$ satisfies the arc condition.
\end{theorem}

\begin{proof}
See Section~\ref{sec:ARC} page~\pageref{proof:arc}.
\end{proof}

\begin{definition}[{\bf Combinatorial Lunes}]
\label{def:clune} 
\index{combinatorial!lune}\index{lune!combinatorial}
Assume~\hyperlink{property_H}{(H)}.
A {\bf combinatorial $(\alpha,\beta)$-lune} is
an $(\alpha,\beta)$-trace $\Lambda=(x,y,\w)$
with boundary $\p\Lambda=:(x,y,A,B)$ that satisfies 
the arc condition and the following.
\begin{description}
\item[(I)]
$\w(z)\ge 0$ for every 
$z\in\Sigma\setminus(\alpha\cup\beta)$.
\item[(II)] 
The intersection index of $A$ and $B$ 
at $x$ is $+1$ and at $y$ is $-1$.
\item[(III)]
$\w(z)\in\{0,1\}$ for $z$ sufficiently
close to $x$ or $y$.
\end{description}
Condition~(II) says
that the angle from $A$ to $B$
at $x$ is between zero and $\pi$
and the angle from $B$ to $A$
at $y$ is also between zero and $\pi$.
\end{definition}

\begin{figure}[htp]
\centering 
\includegraphics[scale=0.5]{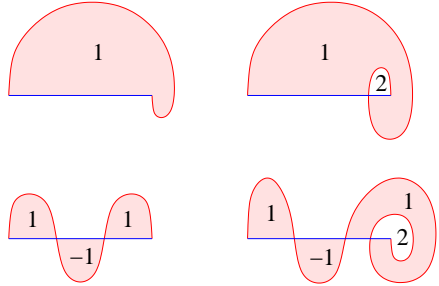}      
\caption{{$(\alpha,\beta)$-traces which satisfy 
the arc condition but are not lunes.}}
\label{fig:prelune}
\end{figure}

\begin{theorem}[{\bf Existence}]\label{thm:lune1} 
\index{existence of a lune}\index{lune!existence}
Assume~\hyperlink{property_H}{(H)} and let 
$
\Lambda=(x,y,\w)
$ 
be an $(\alpha,\beta)$-trace.
Consider the following three conditions.

\smallskip\noindent{\bf (i)}
There exists a smooth $(\alpha,\beta)$-lune $u$
such that $\Lambda_u=\Lambda$.

\smallskip\noindent{\bf (ii)}
$\w\ge 0$ and $\mu(\Lambda)=1$.

\medskip\noindent{\bf (iii)}
$\Lambda$ is a combinatorial $(\alpha,\beta)$-lune.

\medskip\noindent
Then~(i)$\implies$(ii)$\iff$(iii).
If $\Sigma$ is simply connected then all three conditions
are equivalent.
\end{theorem}

\begin{proof}
See Section~\ref{sec:CL} page~\pageref{proof:lune1}.
\end{proof}

\begin{theorem}[{\bf Uniqueness}]\label{thm:lune2} 
\index{uniqueness of lunes}\index{lune!uniqueness}
Assume~\hyperlink{property_H}{(H)}.  
If two smooth $(\alpha,\beta)$-lunes
have the same trace then they are equivalent.
\end{theorem}

\begin{proof}
See Section~\ref{sec:CL} page~\pageref{proof:lune2}.
\end{proof}

\begin{corollary}\label{cor:lune}
Assume~\hyperlink{property_H}{(H)} and let 
$$
\Lambda=(x,y,\w)
$$
be an $(\alpha,\beta)$-trace. 
Choose a universal covering $\pi:\tSi\to\Sigma$, a point 
$$
\tx\in\pi^{-1}(x), 
$$
and lifts $\talpha$ and~$\tbeta$ of $\alpha$ and $\beta$ such that
$$
\tx\in\talpha\cap\tbeta.
$$
Let 
$$
\tLa=(\tx,\ty,\widetilde\w)
$$
be the lift of $\Lambda$ to the universal cover. 

\smallskip\noindent{\bf (i)}
If $\tLa$ is a combinatorial $(\talpha,\tbeta)$-lune
then $\Lambda$ is a combinatorial $(\alpha,\beta)$-lune.

\smallskip\noindent{\bf (ii)}
$\tLa$ is a combinatorial $(\talpha,\tbeta)$-lune
if and only if there exists a smooth $(\alpha,\beta)$-lune $u$
such that $\Lambda_u=\Lambda$. 
\end{corollary}

\begin{proof}
Lifting defines a one-to-one correspondence between 
smooth $(\alpha,\beta)$-lunes with trace $\Lambda$
and smooth $(\talpha,\tbeta)$-lunes with trace $\tLa$. 
Hence the assertions follow from Theorem~\ref{thm:lune1}.
\end{proof}

\begin{remark}\label{rmk:lune}\rm
Assume~\hyperlink{property_H}{(H)} 
and let $\Lambda$ be an $(\alpha,\beta)$-trace.  
We conjecture that the three conditions in Theorem~\ref{thm:lune1} 
are equivalent, even when $\Sigma$ is not simply connected, i.e.\ 

\medskip\centerline{\it 
If $\Lambda$ is a combinatorial $(\alpha,\beta)$-lune}
\centerline{\it 
then there exists a smooth $(\alpha,\beta)$-lune 
$u$ such that $\Lambda=\Lambda_u$.}

\medskip\noindent
Theorem~\ref{thm:lune1} shows that
this conjecture is equivalent to the following.

\medskip\centerline{\it 
If $\Lambda$ is a combinatorial $(\alpha,\beta)$-lune}
\centerline{\it 
then $\tLa$ is a combinatorial $(\talpha,\tbeta)$-lune.}

\medskip\noindent
The hard part is to prove that $\tLa$ satisfies~(I), i.e.\ that
the winding numbers are nonnegative.  
\end{remark}

\begin{remark}\label{rmk:algorithm}\rm
Assume~\hyperlink{property_H}{(H)}. 
Corollary~\ref{cor:lune} and Theorem~\ref{thm:lune2}
suggest the following algorithm for finding a smooth
$(\alpha,\beta)$-lune. 

\medskip\noindent{\bf 1.}
Fix two points $x,y\in\alpha\cap\beta$
with opposite intersection indices,
and two oriented embedded arcs 
$A\subset\alpha$ and $B\subset\beta$ 
from $x$ to~$y$ so that~(II) holds.

\medskip\noindent{\bf 2.}
If $A$ is not homotopic to $B$
with fixed endpoints discard this pair.\footnote{
This problem is solvable via Dehn's algorithm.
See the wikipedia article
\href{http://en.wikipedia.org/wiki/Small_cancellation_theory}{Small Cancellation Theory}
and the references cited therein.}
Otherwise $(x,y,A,B)$ is the boundary of 
an $(\alpha,\beta)$-trace $\Lambda=(x,y,\w)$ 
satisfying the arc condition and~(II) 
(for a suitable function $\w$ to be chosen below). 

\medskip\noindent{\bf 3a.}
If $\Sigma$ is diffeomorphic to the $2$-sphere let
$
\w:\Sigma\setminus(A\cup B)\to\Z
$
be the winding number of the loop $A-B$ 
in $\Sigma\setminus\{z_0\}$, where 
$z_0\in\alpha\setminus A$ is chosen close to $x$.
Check if $\w$ satisfies~(I) and~(III).  
If yes, then $\Lambda=(x,y,\w)$ is a 
combinatorial $(\alpha,\beta)$-lune and hence, 
by Theorems~\ref{thm:lune1} and~\ref{thm:lune2}, 
gives rise to a smooth $(\alpha,\beta)$-lune 
$u$, unique up to isotopy.

\medskip\noindent{\bf 3b.}
If $\Sigma$ is not diffeomorphic to the $2$-sphere choose 
lifts $\tA$ of $A$ and $\tB$ of~$B$ to a universal covering 
$\pi:\C\to\Sigma$ connecting $\tx$ and $\ty$
and let 
$$
\widetilde\w:\C\setminus(\tA\cup\tB)\to\Z
$$
be the winding number of $\tA-\tB$.  
Check if $\widetilde\w$ satisfies~(I) and~(III).  
If yes, then $\tLa:=(\tx,\ty,\widetilde\w)$ 
is a combinatorial $(\talpha,\tbeta)$-lune and hence, 
by Theorem~\ref{thm:lune1}, gives rise to a smooth
$(\alpha,\beta)$-lune $u$ such that
$$
\Lambda_u=\Lambda:=(x,y,\w),\qquad
\w(z):=\sum_{\tz\in\pi^{-1}(z)}\widetilde\w(\tz).
$$
By Theorem~\ref{thm:lune2}, the $(\alpha,\beta)$-lune $u$
is uniquely determined by $\Lambda$ up to isotopy.
\end{remark}

\begin{proposition}\label{prop:trace1}
Assume~\hyperlink{property_H}{(H)} 
and let $\Lambda=(x,y,\w)$ be an $(\alpha,\beta)$-trace 
that satisfies the arc condition and let $\p\Lambda=:(x,y,A,B)$.
Let $S$ be a connected component of $\Sigma\setminus(A\cup B)$
such that $\w|_S\not\equiv0$. Then $S$ is diffeomorphic 
to the open unit disc in $\C$.
\end{proposition}

\begin{proof}
By Definition~\ref{def:trace-CF}, there is a smooth map
$u:\D\to\Sigma$ satisfying~\eqref{eq:bc}
such that $\Lambda_u=\Lambda$. By a homotopy argument we may assume, 
without loss of generality, that $u(\D\cap\R)=A$ and $u(\D\cap S^1)=B$.
Let $S$ be a connected component of $\Sigma\setminus(A\cup B)$
such that $\w$ does not vanish on $S$.  We prove in two steps 
that $S$ is diffeomorphic to the open unit disc in $\C$.  

\medskip\noindent{\bf Step~1.}
{\it If $S$ is not diffeomorphic to the open unit disc in $\C$ 
then there is an embedded loop $\gamma:\R/\Z\to S$ 
and a loop $\gamma':\R/\Z\to\Sigma$ with intersection number 
$\gamma\cdot\gamma'=1$.}

\medskip\noindent
If $S$ has positive genus there are in fact two embedded loops in $S$ 
with intersection number one.  If $S$ has genus zero but is not 
diffeomorphic to the disc it is diffeomorphic to a multiply connected 
subset of $\C$, i.e.\ a disc with at least one hole cut out. 
Let $\gamma:\R/\Z\to S$ be an embedded loop encircling one 
of the holes and choose an arc in $\overline S$ which 
connects two boundary points and has intersection number
one with $\gamma$. (For an elegant construction 
of such a loop in the case of an open subset of $\C$ 
see Ahlfors~\cite{AHLFORS}.)\phantomsection\label{AHLFORS}
Since $\Sigma\setminus S$ is connected the arc 
can be completed to a loop in $\Sigma$ which still has
intersection number one with $\gamma$. 
This proves Step~1.  

\medskip\noindent{\bf Step~2.}
{\it $S$ is diffeomorphic to the open unit disc in $\C$.}

\medskip\noindent
Assume, by contradiction, that this is false 
and choose $\gamma$ and $\gamma'$ as in Step~1.
By transversality theory we may assume that $u$ is transverse to $\gamma$. 
Since $C:=\gamma(\R/\Z)$ is disjoint from $u(\p\D)=A\cup B$ it follows that
$\Gamma := u^{-1}(C)$ is a disjoint union of embedded circles
in $\Delta:=u^{-1}(S)\subset\D$.  Orient $\Gamma$ such that the degree 
of $u|_\Gamma:\Gamma\to C$ agrees with the degree 
of $u|_\Delta:\Delta\to S$.  More precisely,  
let $z\in\Gamma$ and $t\in\R/\Z$ such that $u(z)=\gamma(t)$.
Call a nonzero tangent vector $\hat z\in T_z\Gamma$ positive if 
the vectors $\dot\gamma(t),du(z)\i\hat z$ form a positively oriented basis
of $T_{u(z)}\Sigma$.  Then, if $z\in\Gamma$ is a regular point of both
$u|_\Delta:\Delta\to S$ and $u|_\Gamma:\Gamma\to C$, the linear 
map $du(z):\C\to T_{u(z)}\Sigma$ has the same sign as 
its restriction $du(z):T_z\Gamma\to T_{u(z)}C$. 
Thus $u|_\Gamma:\Gamma\to C$ has nonzero degree. 
Choose a connected component $\Gamma_0$ of $\Gamma$ such that
$u|_{\Gamma_0}:\Gamma_0\to C$ has degree $d\ne 0$.  
Since $\Gamma_0$ is a loop in $\D$ it follows that
the $d$-fold iterate of $\gamma$ is contractible.  
Hence $\gamma$ is contractible by~\ref{le:dbae2} 
in Appendix~\ref{app:path}.
This proves Step~2 and Proposition~\ref{prop:trace1}.
\end{proof}


\section{Arcs}\label{sec:ARC}

In this section we prove Theorem~\ref{thm:arc}. 
The first step is to prove the arc condition 
under the hypothesis that $\alpha$ and $\beta$ are not 
contractible (Proposition~\ref{prop:arc}). The second step 
is to characterize embedded lunes in terms of their traces
(Proposition~\ref{prop:simple}).  The third step is to prove 
the arc condition for lunes in the two-sphere 
(Proposition~\ref{prop:S2arc}).

\begin{proposition}\label{prop:arc}
Assume~\hyperlink{property_H}{(H)}, 
suppose $\Sigma$ is not simply connected,
and choose a universal covering $\pi:\C\to\Sigma$.  
Let $\Lambda=(x,y,\w)$ be an $(\alpha,\beta)$-trace
and denote 
$$
\nu_\alpha:=\p\w|_{\alpha\setminus\beta},\qquad
\nu_\beta:=-\p\w|_{\beta\setminus\alpha}.
$$
Choose lifts $\talpha$, $\tbeta$, and
$
\tLa=(\tx,\ty,\widetilde\w)
$
of $\alpha$, $\beta$, and $\Lambda$ 
such that $\tLa$ is an $(\talpha,\tbeta)$-trace.  
Thus $\tx,\ty\in\talpha\cap\tbeta$ and the path 
from $\tx$ to $\ty$ in $\talpha$ (respecively~$\tbeta$) 
determined by $\p\widetilde\w$ is the lift of the path 
from $x$ to $y$ in $\alpha$ (respectively $\beta$) 
determined by $\p\w$.  Assume
$$
\widetilde\w\ge 0,\qquad \widetilde\w\not\equiv 0.
$$
Then the following holds

\smallskip\noindent{\bf (i)}
If $\alpha$ is a noncontractible embedded circle then there 
exists an oriented arc $A\subset\alpha$ from $x$ to~$y$
(equal to $\{x\}$ in the case $x=y$) such that
\begin{equation}\label{eq:A}
\nu_\alpha(z)=\left\{\begin{array}{rl}
\pm1,&\mbox{for }z\in A\setminus\beta,\\
0,&\mbox{for }z\in\alpha\setminus(A\cup\beta).
\end{array}\right.
\end{equation}
Here the plus sign is chosen if and only if the orientations of $A$ 
and $\alpha$ agree.  If $\beta$ is a noncontractible embedded circle
the same holds for $\nu_\beta$.  

\smallskip\noindent{\bf (ii)}
If $\alpha$ and $\beta$ are both noncontractible
embedded circles then $\Lambda$ satisfies
the arc condition. 
\end{proposition}

\begin{proof}
We prove~(i).
The universal covering
$\pi:\C\to\Sigma$ and the lifts $\talpha$, $\tbeta$, and
$
\tLa=(\tx,\ty,\widetilde\w)
$
can be chosen such that
$$
\talpha=\R,\qquad
\tx=0,\qquad
\ty=a\ge 0,\qquad
\pi(\tz+1)=\pi(\tz),
$$
and $\pi$ maps the interval $[0,1)$
bijectively onto $\alpha$.  
Denote by 
$$
\tB\subset\tbeta
$$
the closure of the support of 
$$
\nu_\tbeta:=-\p\widetilde\w|_{\tbeta\setminus\talpha}.
$$ 
If $\beta$ is noncontractible then $\tB$ is the unique 
arc in $\tbeta$ from~$0$ to~$a$. If $\beta$ is contractible
then $\tbeta\subset\C$ is an embedded circle and $\tB$ is either
an arc in $\tbeta$ from $0$ to $a$ or is equal to $\tbeta$.
We must prove that~$A:=\pi([0,a])$ is an arc or,
equivalently, that $a<1$.  

\begin{figure}[htp]
\centering 
\includegraphics[scale=0.9]{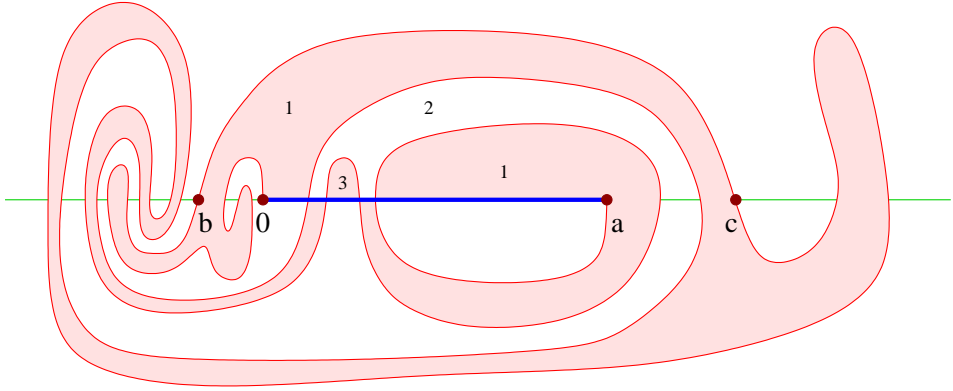}      
\caption{{The lift of an $(\alpha,\beta)$-trace 
with $\widetilde\w\ge0$.}}\label{fig:lift}
\end{figure}

Let $\Gamma$ be the set of connected components $\gamma$ of
$\tB\cap(\R\times[0,\infty))$ such that the function
$\widetilde\w$ is zero on one side of $\gamma$ and
positive on the other. If $\gamma\in\Gamma$, 
neither end point of $\gamma$ can lie in the open 
interval $(0,a)$ since the function $\widetilde\w$ 
is at least one above this interval.  
We claim that there exists a connected component
$\gamma\in \Gamma$ whose endpoints $b$ and $c$
satisfy
\begin{equation}\label{eq:b}
b\le 0\le a\le c, \qquad \p\gamma=\{b,c\}.
\end{equation}
(See Figure~\ref{fig:lift}.)
To see this walk slightly above the real axis towards
zero, starting at $-\infty$.
Just before the first crossing $b_1$ with $\tB$
turn left and follow the arc in $\tB$
until it intersects the real axis again at $c_1$.
The two intersections $b_1$ and $c_1$ are the endpoints
of an  element $\gamma_1$ of $\Gamma$.
Obviously $b_1\le 0$ and, as noted above,
$c_1$ cannot lie in the interval $(0,a)$.
For the same reason $c_1$ cannot be equal to zero.
Hence either $c_1<0$ or $c_1\ge a$.  In the latter
case $\gamma_1$ is the required arc $\gamma$.
In the former case we continue walking towards
zero along the real axis until the next intersection
with $\tB$ and repeat the above procedure.
Because the set of intersection points
of $\tB$ with $\talpha=\R$ is finite the
process must terminate after finitely  many steps.
Thus we have proved the existence of an arc 
$\gamma\in\Gamma$ satisfying~(\ref{eq:b}).

Assume that
$$
c\ge b+1.
$$
If $c=b+1$ then $c\in\tbeta\cap(\tbeta+1)$ and hence $\tbeta=\tbeta+1$.  
It follows that the intersection numbers of $\R$ and $\tbeta$ 
at $b$ and $c$ agree. But this contradicts the fact that
$b$ and $c$ are the endpoints of an arc in $\tbeta$ contained
in the closed upper halfplane.  Thus we have
$
c> b+1.
$ 
When this holds the arc $\gamma$ and its translate
$\gamma+1$ must intersect and their intersection 
does not contain the endpoints $b$ and $c$.  
We denote by 
$
\zeta\in\gamma\setminus\{b,c\}
$ 
the first point in $\gamma+1$ we encounter 
when walking along $\gamma$ from $b$ to~$c$.  
Let 
$$
U_0\subset\tbeta,\qquad
U_1\subset\tbeta+1
$$
be sufficiently small connected open neighborhoods 
of $\zeta$, so that $\pi:U_0\to\beta$ and $\pi:U_1\to\beta$ 
are embeddings and their images agree.  
Thus 
$$
\pi(U_0)=\pi(U_1)\subset\beta
$$
is an open neighborhood of $z:=\pi(\zeta)$ in~$\beta$.  
Hence it follows from a lifting argument that 
$U_0=U_1\subset\gamma+1$ and this contradicts 
our choice of $\zeta$.  This contradiction shows that 
our hypothesis $c\ge b+1$ must have been wrong.  
Thus we have proved that 
$$
b\le 0\le a\le c< b+1\le 1.
$$
Hence 
$
0\le a< 1
$
and so $A=\pi([0,a])$ is an arc, as claimed.
In the case $a=0$ we obtain the trivial arc from $x=y$ 
to itself.  This proves~(i).  

We prove~(ii).  Assume that $\alpha$ and $\beta$ 
are noncontractible embedded circles.  Then it follows from~(i)
that there exist oriented arcs $A\subset\alpha$ and $B\subset\beta$
from $x$ to $y$ such that $\nu_\alpha$ and $\nu_\beta$ are given 
by~\eqref{eq:nuAB-CF}.  If $x=y$ it follows also from~(i) that $A=B=\{x\}$,
hence 
$$
\nu_\talpha\equiv0,\qquad \nu_\tbeta\equiv 0,
$$ 
and hence $\widetilde\w\equiv0$, in contradiction to our hypothesis.
Thus $x\ne y$ and so $\Lambda$ satisfies the arc condition. 
This proves~(ii) and Proposition~\ref{prop:arc}.
\end{proof}

\begin{example}\label{ex:arc}\rm
Let $\alpha\subset\Sigma$ be a noncontractible embedded circle
and $\beta\subset\Sigma$ be a contractible embedded circle
intersecting $\alpha$ transversally.  Suppose $\beta$ is oriented 
as the boundary of an embedded disc $\Delta\subset\Sigma$.  
Let 
$$
x=y\in\alpha\cap\beta,\qquad
\nu_\alpha\equiv 0,\qquad \nu_\beta\equiv 1,
$$
and define 
$$
\w(z):=\left\{\begin{array}{ll}
1,&\mbox{for }z\in\Delta\setminus(\alpha\cup\beta),\\
0,&\mbox{for }z\in\Sigma\setminus(\Delta\cup\alpha\cup\beta).
\end{array}\right.
$$
Then $\Lambda=(x,y,\nu_\alpha,\nu_\beta,\w)$ is an $(\alpha,\beta)$-trace
that satisfies the hypotheses of Proposition~\ref{prop:arc}~(i)
with $x=y$ and $A=\{x\}$.
\end{example}

\begin{definition}\label{def:primitive}\rm
An $(\alpha,\beta)$-trace 
$
\Lambda=(x,y,\w)
$
is called {\bf primitive} if it satisfies the arc condition \index{primitive trace}
with boundary $\p\Lambda=:(x,y,A,B)$ and
$$
A\cap\beta=\alpha\cap B=\{x,y\}.
$$
A smooth $(\alpha,\beta)$-lune $u$ is called {\bf primitive}
if its $(\alpha,\beta)$-trace $\Lambda_u$ is primitive.  
It is called {\bf embedded} if $u:\D\to\Sigma$ is injective. \index{embedded lune}
\end{definition}

The next proposition is the special case of Theorems~\ref{thm:lune1}
and~\ref{thm:lune2} for embedded lunes.  
It shows that isotopy classes of primitive smooth 
$(\alpha,\beta)$-lunes are in one-to-one correspondence 
with the simply connected components of 
$\Sigma\setminus(\alpha\cup\beta)$ with two corners.
We will also call such a component a 
{\bf primitive $(\alpha,\beta)$-lune}.

\begin{proposition}[{\bf Embedded lunes}]\label{prop:simple}
Assume~\hyperlink{property_H}{(H)} and let 
$\Lambda=(x,y,\w)$ be an $(\alpha,\beta)$-trace.
The following are equivalent.

\smallskip\noindent{\bf (i)}
$\Lambda$ is a combinatorial lune 
and its boundary $\p\Lambda=(x,y,A,B)$
satisfies 
$$
A\cap B=\{x,y\}.
$$

\smallskip\noindent{\bf (ii)}
There exists an embedded $(\alpha,\beta)$-lune $u$ such that 
$\Lambda_u=\Lambda$.

\medskip\noindent
If $\Lambda$ satisfies~(i)
then any two smooth $(\alpha,\beta)$-lunes~$u$ and~$v$
with $\Lambda_u=\Lambda_v=\Lambda$ are equivalent.
\end{proposition}

\begin{proof}
We prove that~(ii) implies~(i).
Let $u:\D\to\Sigma$ be an embedded $(\alpha,\beta)$-lune
with $\Lambda_u=\Lambda$. Then $u|_{\D\cap\R}:\D\cap\R\to\alpha$ 
and $u|_{\D\cap S^1}:\D\cap S^1\to\beta$ 
are embeddings.  Hence $\Lambda$
satisfies the arc condition and 
$\p\Lambda=(x,y,A,B)$ with $A=u(\D\cap\R)$ 
and $B=u(\D\cap S^1)$.  Since $\w$ is the counting 
function of $u$ it takes only the values zero and one. 
If $z\in A\cap B$ then $u^{-1}(z)$ 
contains a single point which must lie 
in $\D\cap\R$ and $\D\cap S^1$,
hence is either $-1$ or $+1$, 
and so $z=x$ or $z=y$.
The assertion about the intersection indices follows
from the fact that $u$ is an immersion.
Thus we have proved that~(ii) implies~(i).

We prove that~(i) implies~(ii).  
This relies on the following.

\medskip\noindent{\bf Claim.}
{\it Let $\Lambda=(x,y,\w)$ be an $(\alpha,\beta)$-trace
that satisfies the arc condition and $\p\Lambda=:(x,y,A,B)$
with ${A\cap B=\{x,y\}}$. Then $\Sigma\setminus(A\cup B)$ has 
two components and one of these is homeomorphic to the disc.}

\medskip\noindent
To prove the claim, let $\Gamma\subset\Sigma$ be an embedded 
circle obtained from $A\cup B$ by smoothing the corners.
Then $\Gamma$ is contractible and hence,
by a theorem of Epstein~\cite{EPSTEIN}, 
\phantomsection\label{EPSTEIN1}
bounds a disc. This proves the claim. 

Now suppose that $\Lambda=(x,y,\w)$ is an 
$(\alpha,\beta)$-trace that satisfies~(i)
and let $\p\Lambda=:(x,y,A,B)$. By the claim, 
the complement $\Sigma\setminus(A\cup B)$
has two components, one of which is homeomorphic to the disc.  
Denote the components by $\Sigma_0$ and $\Sigma_1$.
Since $\Lambda$ is a combinatorial lune, it follows that
$\w$ only takes the values zero and one.  
Hence we may choose the indexing such that
$$
\w(z)=\left\{\begin{array}{ll}
0,&\mbox{for }z\in\Sigma_0\setminus(\alpha\cup\beta),\\
1,&\mbox{for }z\in\Sigma_1\setminus(\alpha\cup\beta).
\end{array}\right.
$$
We prove that $\Sigma_1$ is homeomorphic to the disc.
Suppose, by contradiction, that $\Sigma_1$ is not homeomorphic 
to the disc.  Then $\Sigma$ is not diffeomorphic to the $2$-sphere
and, by the claim, $\Sigma_0$ is homeomorphic to the disc.
By Definition~\ref{def:arc-CF}, there is a smooth map
$u:\D\to\Sigma$ that satisfies the boundary 
condition~\eqref{eq:bc} such that $\Lambda_u=\Lambda$.
Since $\Sigma$ is not diffeomorphic to the $2$-sphere, 
the homotopy class of $u$ is uniquely determined by
the quadruple $(x,y,A,B)$ (see~Theorem~\ref{thm:trace} above).
Since $\Sigma_0$ is homeomorphic to the disc we may choose $u$ 
such that $u(\D)=\overline\Sigma_0$ and hence $\w(z)=\deg(u,z)=0$ 
for $z\in\Sigma_1\setminus(\alpha\cup\beta)$,
in contradiction to our choice of indexing. This shows 
that $\Sigma_1$ must be homeomorphic to the disc.
Let $N$ denote the closure of $\Sigma_1$:
$$
N:=\overline\Sigma_1 = \Sigma_1\cup A\cup B.
$$
Then the orientation of $\p N=A\cup B$ agrees with the 
orientation of $A$ and is opposite to the orientation of $B$,
i.e.\ $N$ lies to the left of $A$ and to the right of $B$.
Since the intersection index of $A$ and $B$ at $x$ is
$+1$ and at $y$ is $-1$, it follows that the angles
of $N$ at $x$ and $y$ are between zero and $\pi$
and hence $N$ is a $2$-manifold with two corners.
Since $N$ is simply connected there exists 
a diffeomorphism $u:\D\to N$ such that
$$
u(-1)=x,\qquad u(1)=y,\qquad 
u(\D\cap\R)=A,\qquad u(\D\cap S^1)=B.
$$
This diffeomorphism is the required
embedded $(\alpha,\beta)$-lune.

We prove that the embedded $(\alpha,\beta)$-lune in~(ii)
is unique up to equivalence.  Let $v:\D\to\Sigma$ be another 
smooth $(\alpha,\beta)$-lune such that $\Lambda_v=\Lambda$. 
Then $v$ maps the boundary of $\D$ bijectively onto $A\cup B$,
because $A\cap B=\{x,y\}$.  Moreover, $\w$ is the counting 
function of $v$ and $\#v^{-1}(z)$ is constant on each component 
of $\Sigma\setminus(A\cup B)$.  Hence $\#v^{-1}(z)=0$ for 
$z\in\Sigma_0$ and  $\#v^{-1}(z)=1$ for $z\in\Sigma_1$.
This shows that $v$ is injective and
$
v(\D) = N = u(\D).
$
Since~$u$ and~$v$ are embeddings the composition
$
\phi:=u^{-1}\circ v:\D\to\D
$
is an orientation preserving diffeomorphism 
such that $\phi(\pm1)=\pm1$.  Hence $v=u\circ\phi$ is 
equivalent to $u$. This proves Proposition~\ref{prop:simple}.
\end{proof}

\begin{lemma}\label{le:slune1}
Assume~\hyperlink{property_H}{(H)} and let 
$
u:\D\to\Sigma
$ 
be a smooth $(\alpha,\beta)$-lune.

\smallskip\noindent{\bf (i)}
Let $S$ be a connected component of $\Sigma\setminus(\alpha\cup\beta)$.
If $S\cap u(\D)\ne\emptyset$ then $S\subset u(\D)$
and $S$ is diffeomorphic to the open unit  disc in $\C$.

\smallskip\noindent{\bf (ii)}
Let  $\Delta$ be a connected component 
of  $\D\setminus u^{-1}(\alpha\cup\beta)$.
Then  $\Delta$  is diffeomorphic to the open unit disc 
and the restriction of  $u$ to $\Delta$ 
is a diffeomorphism onto  the open set 
$S := u(\Delta)\subset\Sigma$. 
\end{lemma}

\begin{proof}
That $S\cap u(\D)\ne\emptyset$ implies $S\subset u(\D)$
follows from the fact that $u$ is an immersion.
That this implies that $S$ is diffeomorphic 
to the open unit disc in $\C$ follows 
as in Proposition~\ref{prop:trace1}. This proves~(i).  
By~(i) the open set $S:=u(\Delta)$ in~(ii) 
is diffeomorphic to the disc and hence is simply connected. 
Since $u:\Delta\to S$ is a proper covering it follows 
that $u:\Delta\to S$ is a diffeomorphism. 
This proves Lemma~\ref{le:slune1}.
\end{proof}

Let $u:\D\to\Sigma$ be a smooth $(\alpha,\beta)$-lune.
The image under $u$ of the connected component of
$\D\setminus u^{-1}(\alpha\cup\beta)$ whose closure 
contains $-1$ is called the {\bf left end of~$u$}.
\index{left end of a lune}\index{lune!left end of}
The image under $u$ of the connected component 
of\index{end of a lune}
$\D\setminus u^{-1}(\alpha\cup\beta)$ whose closure 
contains $+1$ is called the {\bf right end of~$u$}. 
\index{right end of a lune}\index{lune!right end of}

\begin{lemma}\label{le:slune2}
Assume~\hyperlink{property_H}{(H)} 
and let $u$ be a smooth $(\alpha,\beta)$-lune.
If there is a primitive $(\alpha,\beta)$-lune with the same left 
or right end as $u$ it is equivalent to~$u$.
\end{lemma}

\begin{proof}
If $u$ is not a primitive lune its ends have 
at least three corners. To see this, walk along $\D\cap\R$
(respectively $\D\cap S^1$) from $-1$ to $1$ 
and let $z_0$ (respectively $z_1$) be the first 
intersection point with $u^{-1}(\beta)$
(respectively $u^{-1}(\alpha)$). Then 
$u(-1)$, $u(z_0)$, $u(z_1)$ are corners 
of the left end of~$u$. Hence the hypotheses of 
Lemma~\ref{le:slune2} imply that $u$ is a primitive lune.  
Two primitive lunes with the same ends 
are equivalent by Proposition~\ref{prop:simple}.  
This proves Lemma~\ref{le:slune2}.
\end{proof}

\begin{proposition}\label{prop:S2arc}
Assume~\hyperlink{property_H}{(H)} and suppose 
that $\Sigma$ is diffeomorphic to the $2$-sphere.  
If $u$ is a smooth $(\alpha,\beta)$-lune 
then $\Lambda_u$ satisfies the arc condition.
\end{proposition}

\begin{proof}
The proof is by induction on the number 
of intersection points of $\alpha$ and $\beta$.
It has three steps.

\medskip\noindent{\bf Step~1.}
{\it Let  $u$ be a smooth $(\alpha,\beta)$-lune whose 
$(\alpha,\beta)$-trace
$$
\Lambda=\Lambda_u=(x,y,\w)
$$
does not satisfy the arc condition. Suppose there is a primitive 
$(\alpha,\beta)$-lune with endpoints in $\Sigma\setminus\{x,y\}$.
Then there is an embedded loop $\beta'$, isotopic to~$\beta$
and transverse to $\alpha$, and a smooth $(\alpha,\beta')$-lune 
$u'$ with endpoints $x,y$ such that
$\Lambda_{u'}$ does not satisfy the arc condition 
and $\#(\alpha\cap\beta')<\#(\alpha\cap\beta)$.}

\smallbreak

\medskip\noindent
By Proposition~\ref{prop:simple}, there exists a 
primitive smooth $(\alpha,\beta)$-lune $u_0:\D\to\Sigma$ 
whose endpoints $x_0:=u_0(-1)$ and $y_0:=u_0(+1)$
are contained in $\Sigma\setminus\{x,y\}$.  
Use this lune to remove the intersection points
$x_0$ and $y_0$ by an isotopy of~$\beta$, supported
in a small neighborhood of the image of $u_0$.
More precisely, extend~$u_0$ to an embedding
(still denoted by $u_0$) of the open set
$$
\D_\eps:=\{z\in\C\,|\,\IM\,z>-\eps,|z|<1+\eps\}
$$
for $\eps>0$ sufficiently small such that
$$
u_0(\D_\eps)\cap\beta = u_0(\D_\eps\cap S^1),\qquad
u_0(\D_\eps)\cap\alpha = u_0(\D_\eps\cap\R).
$$
Choose a smooth cutoff function $\rho:\D_\eps\to[0,1]$ which 
vanishes near the boundary and is equal to one on $\D$.
Consider the vector field $\xi$ on $\Sigma$
that vanishes outside $u_0(\D_\eps)$ and satisfies
$$
u_0^*\xi(z) = -\rho(z)\i.
$$
Let $\psi_t:\Sigma\to\Sigma$ 
be the isotopy generated by $\xi$ and, 
for $T>0$ sufficiently large, define
$$
\beta':=\psi_T(\beta),\qquad
\Lambda':=(x,y,\nu_\alpha,\nu_{\beta'},\w').
$$
Here $\nu_{\beta'}:\beta'\setminus\alpha\to\Z$ is the 
unique one-chain equal to $\nu_\beta$ on 
$\beta\setminus u_0(\D_\eps)$ and 
$\w':\Sigma\setminus(\alpha\cup\beta')\to\Z$ 
is the unique two-chain equal to $\w$ on 
$\Sigma\setminus u_0(\D_\eps)$. Since~$\Lambda$ does
not satisfy the arc condition, neither does $\Lambda'$.  
Let $U\subset\D$ be the union of the components of
$u^{-1}(u_0(\D_\eps))$ that contain an arc
in $\D\cap S^1$ and define the map 
$u':\D\to\Sigma$ by
$$
u'(z) := \left\{\begin{array}{rl}
\psi_T(u(z)),&\mbox{if }z\in U,\\
u(z),&\mbox{if }z\in\D\setminus U.
\end{array}\right.
$$
We prove that $U\cap\R=\emptyset$. 
To see this, note that the restriction of $u$
to each connected component of $U$ 
is a diffeomorphism onto its image which is
either equal to $u_0(\{z\in\D_\eps\,|\,\Abs{z}\ge 1\})$
or equal to $u_0(\{z\in\D_\eps\,|\,\Abs{z}\le 1\})$ 
(see Figure~\ref{fig:deformation} below).  Thus 
$$
u(U)\cap(\alpha\setminus\{x_0,y_0\})\subset\mathrm{int}(u(U))
$$ 
and hence $U\cap\R=\emptyset$ as claimed.
This implies that $u'$ is a smooth $(\alpha,\beta')$-lune
such that ${\Lambda_{u'}=\Lambda'}$.  Hence $\Lambda_{u'}$
does not satisfy the arc condition. 
This proves Step~1.

\begin{figure}[htp]
\centering 
\includegraphics[scale=0.50]{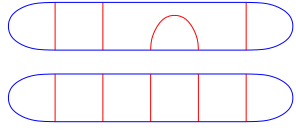}      
\caption{{$\alpha$ encircles at least two primitive lunes.}}
\label{fig:primlune1}
\end{figure}

\medskip\noindent{\bf Step~2.}
{\it Let $u$ be a smooth $(\alpha,\beta)$-lune with 
endpoints $x,y$ and suppose that every primitive 
$(\alpha,\beta)$-lune has $x$  or $y$ as one of its endpoints.
Then $\Lambda_u$ satisfies the arc condition.}

\medskip\noindent
Both connected components of $\Sigma\setminus\alpha$ are discs, 
and each of these discs contains at least two primitive 
$(\alpha,\beta)$-lunes.  If it contains more than two there is one with 
endpoints in $\Sigma\setminus\{x,y\}$.  Hence, under 
the assumptions of Step~2, each connected component 
of $\Sigma\setminus\alpha$ contains precisely two primitive 
$(\alpha,\beta)$-lunes. (See Figure~\ref{fig:primlune1}.)  
Thus each connected component of $\Sigma\setminus(\alpha\cup\beta)$ 
is either a quadrangle or a primitive $(\alpha,\beta)$-lune and 
there are precisely four primitive $(\alpha,\beta)$-lunes, 
two in each connected component of $\Sigma\setminus\alpha$. 
At least two primitive $(\alpha,\beta)$-lunes contain~$x$  
and at least two contain~$y$. 
(See Figure~\ref{fig:primlune2}.)
\begin{figure}[htp]
\centering 
\includegraphics[scale=0.35]{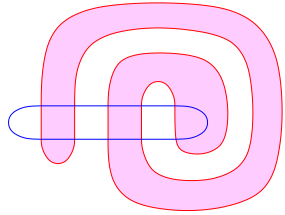} 
\caption{{Four primitive lunes in the 2-sphere.}}
\label{fig:primlune2}
\end{figure}

As $\Sigma$ is diffeomorphic to $S^2$, the number of intersection points 
of $\alpha$ and~$\beta$ is even. Write
$
\alpha\cap\beta = \{x_0,\dots,x_{2n-1}\},
$
where the ordering is chosen along~$\alpha$, starting at $x_0=x$.
Then $x_0$ is contained in two primitive $(\alpha,\beta)$-lunes,
one with endpoints $x_0,x_{2n-1}$ and one with endpoints $x_0,x_1$.  
Each connected component of $\Sigma\setminus\alpha$ 
determines an equivalence relation on $\alpha\cap\beta$:
distinct points are equivalent iff they are connected 
by a $\beta$-arc in this component.  Let $A$ be the connected 
component containing the $\beta$-arc from $x_0$ to $x_{2n-1}$ 
and $B$ be the connected component containing the $\beta$-arc 
from $x_0$ to $x_1$.  Then $x_{j-1}\sim_A x_{2n-j}$ and 
$x_{j+1}\sim_B x_{2n-j}$ for  $j=1,\dots,n$.
Thus the only other intersection point contained in 
two primitive $(\alpha,\beta)$-lunes is $y=x_n$. 
Moreover, $\alpha$ and $\beta$ have opposite intersection indices 
at $x_i$ and $x_{i+1}$ for each~$i$, because the arcs in $\alpha$ 
from $x_{i-1}$ to $x_i$ and from $x_i$ to $x_{i+1}$ are contained 
in different connected components of $\Sigma\setminus\beta$.  
Since $\alpha$ and $\beta$ have opposite intersection 
indices at $x$ and $y$ it follows that $n$ is odd.
Now the image of a neighborhood of $\R\cap\D$ under $u$ 
is contained in either $\overline{A}$ or $\overline{B}$.
Hence, when $n=2k+1\ge 3$, Figure~\ref{fig:primlune3} 
shows that one of the ends of $u$ is a quadrangle 
and the other end is a primitive $(\alpha,\beta)$-lune, 
in contradiction to Lemma~\ref{le:slune2}.
Hence the number of intersection points is $2n=2$,
each component of $\Sigma\setminus(\alpha\cup\beta)$
is a primitive $(\alpha,\beta)$-lune, and all four primitive 
$(\alpha,\beta)$-lunes contain $x$ and $y$.  
By Lemma~\ref{le:slune2}, one of them is equivalent to $u$. 
This proves Step~2.

\begin{figure}[htp]
\centering 
\includegraphics[scale=0.50]{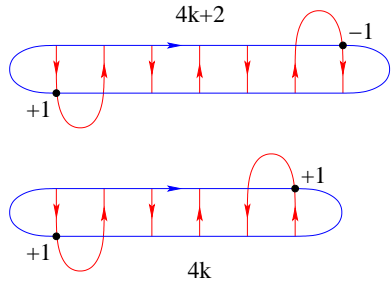}      
\caption{{$\alpha$ intersects $\beta$ in $4k$ or $4k+2$ points.}}
\label{fig:primlune3}
\end{figure}

\medskip\noindent{\bf Step~3.}
{\it We prove the proposition.}

\medskip\noindent
Assume, by contradiction, that there is a smooth
$(\alpha,\beta)$-lune $u$ such that $\Lambda_u$ 
does not satisfy the arc condition.
By Step~1 we can reduce the number of intersection points
of $\alpha$ and $\beta$ until there are no primitive 
$(\alpha,\beta)$-lunes with endpoints in $\Sigma\setminus\{x,y\}$. 
Once this algorithm terminates the resulting lune still does not
satisfy the arc condition, in contradiction to Step~2.
This proves Step~3 and Proposition~\ref{prop:S2arc}.
\end{proof}

\begin{proof}[Proof of Theorem~\ref{thm:arc}] 
\phantomsection\label{proof:arc}
Let 
$
u:\D\to\Sigma
$ 
be a smooth $(\alpha,\beta)$-lune with $(\alpha,\beta)$-trace
$
\Lambda_u=:(x,y,\w)
$
and denote $A:=u(\D\cap\R)$ and $B:=u(\D\cap S^1)$.
Since $u$ is an immersion, $\alpha$ and $\beta$ 
have opposite intersection indices at $x$ and~$y$,
and hence $x\ne y$.  We must prove that $A$ and $B$
are arcs.  It is obvious that $A$ is an arc whenever $\alpha$
is not compact, and $B$ is an arc whenever $\beta$ 
is not compact. It remains to show that $A$ and $B$ 
are arcs in the remaining cases. We prove this in four steps.

\medskip\noindent{\bf Step~1.}
{\it If $\alpha$ is not a contractible embedded
circle then $A$ is an arc.}

\medskip\noindent
This follows immediately from Proposition~\ref{prop:arc}.

\medskip\noindent{\bf Step~2.}
{\it If $\alpha$ and $\beta$ are contractible embedded circles
then $A$ and $B$ are arcs.}

\medskip\noindent
If $\Sigma$ is diffeomorphic to $S^2$ this 
follows from Proposition~\ref{prop:S2arc}. 
Hence assume that $\Sigma$ is not diffeomorphic to $S^2$.  
Then the universal cover of $\Sigma$ is diffeomorphic 
to the complex plane.  Choose a universal covering 
$\pi:\C\to\Sigma$ and a point 
$
\tx\in\pi^{-1}(x).
$
Choose lifts $\talpha,\tbeta\subset\C$
of $\alpha,\beta$ such that 
$
\tx\in\talpha\cap\tbeta.
$
Then $\talpha$ and $\tbeta$ 
are embedded loops in $\C$ and $u$ lifts to 
a smooth $(\talpha,\tbeta)$-lune 
$\tu:\D\to\C$ such that 
$
\tu(-1)=\tx.
$
Compactify  $\C$  to get the 2-sphere.
Then, by Proposition~\ref{prop:S2arc}, the subsets 
$\tA:=\tu(\D\cap\R)\subset\talpha$ and 
$\tB:=\tu(\D\cap S^1)\subset\tbeta$ are arcs. 
Since the restriction of $\pi$ to $\talpha$
is a diffeomorphism from $\talpha$ to $\alpha$
it follows that $A\subset\alpha$ is an arc.
Similarly for $B$.  This proves Step~2.

\medskip\noindent{\bf Step~3.}
{\it If $\alpha$ is not a contractible embedded circle 
and $\beta$ is a contractible embedded circle 
then $A$ and $B$ are arcs.}

\medskip\noindent
That $A$ is an arc follows from Step~1.
To prove that $B$ is an arc choose a universal covering
$\pi:\C\to\Sigma$ with $\pi(0)=x$ and lifts $\talpha$, $\tbeta$, 
$\tu$ with $0\in\talpha\cap\tbeta$ and $\tu(-1)=0$
as in the proof of Step~2. Then $\tbeta\subset\C$
is an embedded loop and we may assume without
loss of generality that $\talpha=\R$ and $\tA=[0,a]$
with $0<a<1$.  (If $\alpha$ is a noncontractible embedded 
circle we choose the lift such that $\tz\mapsto\tz+1$ is a covering
transformation and $\pi$ maps the interval $[0,1)$ 
bijectively onto $\alpha$; if $\alpha$ is not compact we choose 
the universal covering such that $\pi$ maps the interval $[0,a]$
bijectively onto $A$ and $\tbeta$ is transverse to~$\R$,
and then replace $\talpha$ by $\R$.)  In the Riemann sphere 
$
S^2\cong\overline{\C}=\C\cup\{\infty\}
$
the real axis $\talpha=\R$ compactifies to a great circle.
Hence it follows from Proposition~\ref{prop:S2arc} 
that $\tB$ is an arc.  Since 
$
\pi:\tbeta\to\beta
$
is a diffeomorphism it follows that $B$ 
is an arc. This proves Step~3. 

\medskip\noindent{\bf Step~4.}
{\it If $\beta$ is not a contractible embedded circle 
then $A$ and $B$ are arcs.}

\medskip\noindent
That $B$ is an arc follows from Step~1 by interchanging 
$\alpha$ and $\beta$ and replacing $u$ with the 
smooth $(\beta,\alpha)$-lune
$$
v(z) := u\left(\frac{\i-z}{1-\i z}\right).
$$
If $\alpha$ is not a contractible embedded circle 
then $A$ is an arc by Step~1.
If $\alpha$ is a contractible embedded circle 
then $A$ is an arc by Step~3
with $\alpha$ and $\beta$ interchanged.
This proves Step~4. The assertion of Theorem~\ref{thm:arc}
follows from Steps~2, 3, and~4.
\end{proof}


\section{Combinatorial Lunes}\label{sec:CL}

In this section we prove Theorems~\ref{thm:lune1} and~\ref{thm:lune2}. 
Proposition~\ref{prop:simple} establishes the equivalence 
of~(i) and~(iii) in Theorem~\ref{thm:lune1} under the additional 
hypothesis that $\Lambda=(x,y,A,B,\w)$ satisfies the 
arc condition and $A\cap B=\{x,y\}$.  In this case the hypothesis
that $\Sigma$ is simply connected can be dropped.
The induction argument for the proof of Theorems~\ref{thm:lune1}
and~\ref{thm:lune2} is the content of the next three lemmas.

\begin{lemma}\label{le:Lambda0}
Assume~\hyperlink{property_H}{(H)} and suppose 
that $\Sigma$ is simply connected. 
Let $\Lambda=(x,y,\w)$ be a combinatorial
$(\alpha,\beta)$-lune with boundary $\p\Lambda=(x,y,A,B)$ 
such that 
$$
A\cap B\ne\{x,y\}. 
$$
Then there exists a combinatorial $(\alpha,\beta)$-lune 
$
\Lambda_0=(x_0,y_0,\w_0)
$
with boundary $\p\Lambda_0=(x_0,y_0,A_0,B_0)$
such that $\w\ge\w_0$ and
\begin{equation}\label{eq:Lambda0}
A_0\subset A\setminus\{x,y\},\quad
B_0\subset B\setminus\{x,y\},\quad
A_0\cap B = A\cap B_0 = \{x_0,y_0\}.
\end{equation}
\end{lemma}

\begin{proof}
Let $\prec$ denote the order relation on $A$ determined
by the orientation from $x$ to $y$. Denote the intersection 
points of $A$ and $B$ by
$$
x = x_0 \prec x_1 \prec \cdots
\prec x_{n-1} \prec x_n = y.
$$
Define a function
$
\sigma:\{0,\dots,n-1\}\to\{1,\dots,n\}
$
as follows.  Walk along $B$ towards $y$,
starting at $x_i$ and denote the next intersection
point encountered by $x_{\sigma(i)}$.
This function $\sigma$ is bijective.
Let $\eps_i\in\{\pm1\}$ be the intersection index
of $A$ and $B$ at~$x_i$. Thus
$$
\eps_0 = 1,\qquad \eps_n = -1,\qquad
\sum_{i=0}^n\eps_i = 0.
$$
Consider the set
$$
I := \{i\in\N\,|\,0\le i\le n-1,\,
\eps_i=1,\,\eps_{\sigma(i)}=-1\}.
$$
We prove that this set has the following properties.
\begin{description}
\item[(a)] $I\ne\emptyset$.
\item[(b)] If $i\in I$, $i<j<\sigma(i)$, and $\eps_j=1$,
then $j\in I$ and $i<\sigma(j)<\sigma(i)$.
\item[(c)] If $i\in I$, $\sigma(i)<j<i$, and $\eps_j=1$,
then $j\in I$ and $\sigma(i)<\sigma(j)<i$.
\item[(d)] $0\in I$ if and only if $n\in\sigma(I)$ 
if and only if $n=1=\sigma(0)$.
\end{description}

\bigbreak

\noindent
To see this, denote by $m_i$ the value of $\w$
in the right upper quadrant near~$x_i$. Thus
$$
m_j = m_0+\sum_{i=1}^j\eps_i
$$
for $j=1,\dots,n$ and
\begin{equation}\label{eq:Sigma}
m_{\sigma(i)} = m_i + \eps_{\sigma(i)}
\end{equation}
for $i=0,\dots,n-1$. (See Figure~\ref{fig:reduce}.)

\begin{figure}[htp]
\centering 
\includegraphics[scale=0.9]{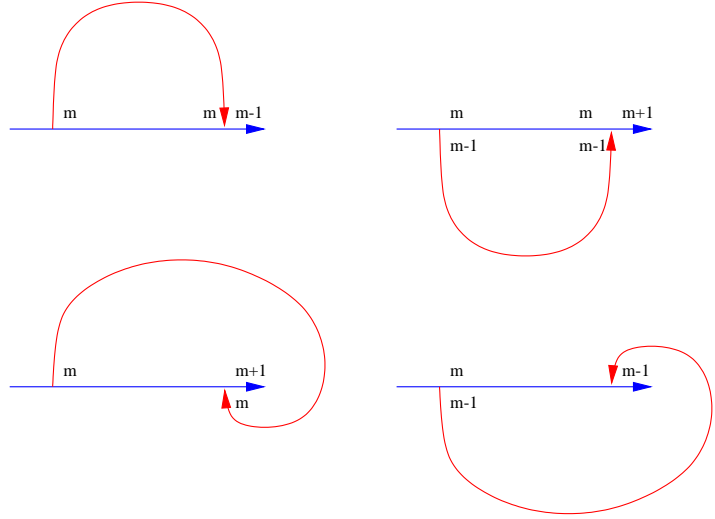}      
\caption{{Simple arcs.}}\label{fig:reduce}
\end{figure}

We prove that $I$ satisfies~(a).
Consider the sequence
$$
i_0:=0,\qquad
i_1:=\sigma(i_0),\qquad
i_2:=\sigma(i_1),\dots.
$$
Thus the points $x_i$ are encountered in the order
$$
x = x_0 = x_{i_0},
x_{i_1},\dots,x_{i_{n-1}},
x_{i_n} = x_n = y
$$
when walking along $B$ from $x$ to $y$.
By~(\ref{eq:Sigma}), we have
$$
\eps_{i_0} = 1,\qquad
\eps_{i_n}=-1,\qquad
m_{i_k} = m_{i_{k-1}}+\eps_{i_k}.
$$
Let $k\in\{0,\dots,n-1\}$ be the largest integer
such that $\eps_{i_k}=1$.  Then we have
$
\eps_{\sigma(i_k)}=\eps_{i_{k+1}}=-1
$
and hence $i_k\in I$.  Thus $I$ is nonempty.

We prove that $I$ satisfies~(b) and~(c).  
Let $i\in I$ such that $\sigma(i)>i$. 
Then $\eps_i=1$ and $\eps_{\sigma(i)}=-1$.  
Hence
$$
m_{\sigma(i)} = m_i + \eps_{\sigma(i)} = m_i-1,
$$
and hence, in the interval $i<j<\sigma(i)$,
the numbers of intersection points with positive
and with negative intersection indices agree.
Consider the arcs $A_i\subset A$ and $B_i\subset B$
that connect $x_i$ to $x_{\sigma(i)}$.
Then $A\cap B_i=\{x_i,x_{\sigma(i)}\}$.
Since $\Sigma$ is simply connected the piecewise smooth
embedded loop $A_i-B_i$ is contractible.  This implies that 
the complement $\Sigma\setminus(A_i\cup B_i)$ has two 
connected components. Let $\Sigma_i$ be the connected component
of $\Sigma\setminus(A_i\cup B_i)$ that contains the points 
slightly to the left of~$A_i$. 
Then any arc on $B$ that starts at $x_j\in A_i$ with $\eps_j=1$
is trapped in $\Sigma_i$ and hence must exit it through~$A_i$.  
Hence
$$
x_j\in A_i,\quad \eps_j=1\qquad\implies\qquad
x_{\sigma(j)}\in A_i,\quad \eps_{\sigma(j)}=-1.  
$$
Thus we have proved that $I$ satisfies~(b).  
That it satisfies~(c) follows by a similar argument.

We prove that $I$ satisfies~(d).  
Here we use the fact that $\Lambda$ satisfies~(III) or, 
equivalently, $m_0=1$ and $m_n=0$. If $0\in I$ then 
$
m_{\sigma(0)}=m_0+\eps_{\sigma(0)}=0.
$
Since $m_i>0$ for $i<n$ this implies $\sigma(0)=n=1$.
Conversely, suppose that $n\in\sigma(I)$ and let 
$i:=\sigma^{-1}(n)\in I$.  Then 
$
m_i=m_n-\eps_{\sigma(i)}=1.
$
Since $m_i>1$ for $i\in I\setminus\{0\}$ this implies $i=0$. 
Thus $I$ satisfies~(d).

It follows from~(a), (b), and (c) by induction 
that there exists a point $i\in I$ such that 
$\sigma(i)\in\{i-1,i+1\}$.  Assume first that $\sigma(i)=i+1$,
denote by $A_i$ the arc in $A$ from $x_i$ to $x_{i+1}$,
and denote by $B_i$ the arc in $B$ from $x_i$ to $x_{i+1}$.
If~$i=0$ it follows from~(d) that $x_i=x_0=x$ and $x_{i+1}=x_n=y$, 
in contradiction to $A\cap B\ne\{x,y\}$.  Hence $i\ne 0$ and 
it follows from~(d) that 
$
0<i<i+1<n.
$ 
The arcs $A_i$ and $B_i$ satisfy 
$$
A_i\cap B=A\cap B_i=\{x_i,x_{i+1}\}.
$$
Let $D_i$ be the connected component of 
$\Sigma\setminus(A\cup B)$ that contains the points
slightly to the left of $A_i$.  This component is bounded
by $A_i$ and $B_i$.  Moreover, the function $\w$ 
is positive on $D_i$. Hence it follows from 
Propostion~\ref{prop:trace1} that $D_i$ 
is diffeomorphic to the open unit disc in $\C$.  
Let $\w_i(z):=1$ for $z\in D_i$ and $\w_i(z):=0$ 
for $z\in\Sigma\setminus\overline D_i$. 
Then the combinatorial lune 
$$
\Lambda_i:=(x_i,x_{i+1},A_i,B_i,\w_i)
$$
satisfies~\eqref{eq:Lambda0} and $\w_i\le\w$.

\bigbreak

Now assume $\sigma(i)=i-1$, denote by $A_i$ the arc in $A$ 
from $x_{i-1}$ to $x_i$, and denote by $B_i$ the arc in $B$ 
from $x_{i-1}$ to $x_i$. Thus the orientation of $A_i$ 
(from $x_{i-1}$ to $x_i$) agrees with the orientation of $A$ 
while the orientation of $B_i$ is opposite to the orientation of $B$.
Moreover, we have 
$
0<i-1<i<n.
$
The arcs $A_i$ and $B_i$ satisfy 
$$
A_i\cap B=A\cap B_i=\{x_{i-1},x_i\}.
$$
Let $D_i$ be the connected component of 
$\Sigma\setminus(A\cup B)$ that contains the points
slightly to the left of $A_i$.  This component is again bounded
by $A_i$ and $B_i$, the function $\w$ is positive on $D_i$,
and so $D_i$ is diffeomorphic to the open unit disc 
in $\C$ by Propostion~\ref{prop:trace1}.  
Let $\w_i(z):=1$ for $z\in D_i$ and $\w_i(z):=0$ 
for $z\in\Sigma\setminus\overline D_i$. 
Then the combinatorial lune 
$$
\Lambda_i:=(x_{i-1},x_i,A_i,B_i,\w_i)
$$
satisfies~\eqref{eq:Lambda0} and $\w_i\le\w$.
This proves Lemma~\ref{le:Lambda0}.
\end{proof}

\begin{lemma}\label{le:one}
Assume~\hyperlink{property_H}{(H)}.
Let $u$ be a smooth $(\alpha,\beta)$-lune
whose $(\alpha,\beta)$-trace
$
\Lambda_u=(x,y,\w)
$
is a combinatorial $(\alpha,\beta)$-lune. 
Let $\gamma:[0,1]\to\D$ be a smooth path such that
$$
\gamma(0)\in(\D\cap\R)\setminus\{\pm1\},\quad
\gamma(1)\in(\D\cap S^1)\setminus\{\pm1\},\quad
u(\gamma(t)) \notin A
$$
for $0<t<1$.  Then $\w(u(\gamma(t)))=1$ for $t$ near $1$. 
\end{lemma}

\begin{proof}
Denote $A:=u(\D\cap\R)$.
Since $\Lambda$ is a combinatorial $(\alpha,\beta)$-lune 
we have $x,y\notin u(\mathrm{int}(\D))$.
Hence $u^{-1}(A)$ is a union of embedded arcs, each connecting   
two points in $\D\cap S^1$.  If $\w(u(\gamma(t))\ge 2$ 
for $t$ close to $1$, then $\gamma(1)\in\D\cap S^1$ is separated from 
$\D\cap\R$ by one these arcs in $\D\setminus\R$.
This proves Lemma~\ref{le:one}.
\end{proof}

For each combinatorial $(\alpha,\beta)$-lune
$\Lambda$ the integer $\nu(\Lambda)$ denotes
the number of equivalence classes of smooth 
$(\alpha,\beta)$-lunes $u$ with $\Lambda_u=\Lambda$.
\begin{figure}[htp]
\centering 
\includegraphics[scale=0.5]{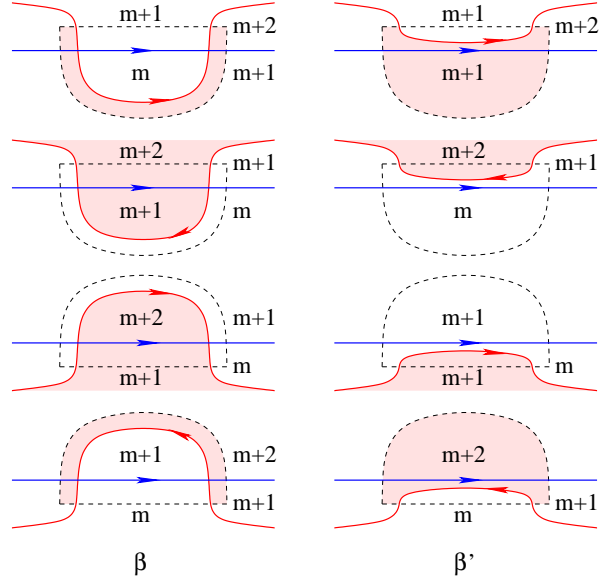}      
\caption{{Deformation of a lune.}}
\label{fig:deformation}
\end{figure}

\begin{lemma}\label{le:move}
Assume~\hyperlink{property_H}{(H)} 
and suppose that $\Sigma$ is simply connected.
Let $\Lambda=(x,y,\w)$ be a combinatorial 
$(\alpha,\beta)$-lune with boundary $\p\Lambda=(x,y,A,B)$
such that $A\cap B\ne\{x,y\}$.
Then there exists an embedded loop $\beta'$, 
isotopic to $\beta$ and transverse 
to $\alpha$, and a combinatorial $(\alpha,\beta')$-lune 
$\Lambda'=(x,y,\w')$ with boundary $\p\Lambda'=(x,y,A,B')$
such that 
$$
\#(A'\cap B') < \#(A\cap B),\qquad
\nu(\Lambda) = \nu(\Lambda').
$$
\end{lemma}

\begin{proof}
By Lemma~\ref{le:Lambda0} there exists 
a combinatorial $(\alpha,\beta)$-lune 
$$
\Lambda_0=(x_0,y_0,\w_0),\qquad
\p\Lambda_0=(x_0,y_0,A_0,B_0),
$$
that satisfies $\w\ge\w_0$ and~\eqref{eq:Lambda0}.  
In particular, we have
$$
A_0\cap B_0=\{x_0,y_0\}
$$
and so, by Proposition~\ref{prop:simple}, there is
an embedded smooth lune $u_0:\D\to\Sigma$
with bottom boundary $A_0$ and top boundary $B_0$.
As in the proof of Step~1 in Proposition~\ref{prop:S2arc}
we use this lune to remove the intersection points~$x_0$ 
and~$y_0$ by an isotopy of $B$, supported
in a small neighborhood of the image of~$u_0$.
This isotopy leaves the number $\nu(\Lambda)$ unchanged.
More precisely, extend $u_0$ to an embedding
(still denoted by $u_0$) of the open set
$$
\D_\eps:=\{z\in\C\,|\,\IM\,z>-\eps,|z|<1+\eps\}
$$
for $\eps>0$ sufficiently small such that
$$
u_0(\D_\eps)\cap B = u_0(\D_\eps\cap S^1),\qquad
u_0(\D_\eps)\cap A = u_0(\D_\eps\cap\R),
$$
$$
u(\{z\in\D_\eps\,|\,|z|>1\})\cap\beta=\emptyset,\quad
u(\{z\in\D_\eps\,|\,\RE\,z<0\})\cap\alpha=\emptyset.
$$
Choose a smooth cutoff function $\rho:\D_\eps\to[0,1]$
that vanishes near the boundary of $\D_\eps$
and is equal to one on $\D$. 
Consider the vector field $\xi$ on $\Sigma$
that vanishes outside $u_0(\D_\eps)$ and satisfies
$$
u_0^*\xi(z) = -\rho(z)\i.
$$
Let $\psi_t:\Sigma\to\Sigma$ 
be the isotopy generated by $\xi$ and, 
for $T>0$ sufficiently large, define
$$
\beta':=\psi_T(\beta),\qquad
B':=\psi_T(B),\qquad
\Lambda':=(x,y,A,B',\w'),
$$
where $\w':\Sigma\setminus(\alpha\cup\beta')\to\Z$ 
is the unique two-chain that agrees with $\w$ on 
$\Sigma\setminus u_0(\D_\eps)$.  Thus $\w'$ corresponds 
to the homotopy from $A$ to $B$ determined by $\w$
followed by the homotopy $\psi_t$ from $B$ to $B'$.
Then $\Lambda'$ is a combinatorial $(\alpha,\beta')$-lune.
If $u:\D\to\Sigma$ is a smooth $(\alpha,\beta)$-lune
let $U\subset\D$ be the unique component of
$u^{-1}(u_0(\D_\eps))$ that contains an arc
in $\D\cap S^1$.  Then $U$ 
does not intersect $\D\cap\R$. 
(See Figure~\ref{fig:deformation}.)
Hence the map $u':\D\to\Sigma$, defined by
$$
u'(z) := \left\{\begin{array}{rl}
\psi_T(u(z)),&\mbox{if }z\in U,\\
u(z),&\mbox{if }z\in\D\setminus U,
\end{array}\right.
$$
is a smooth $(\alpha,\beta')$-lune such that
$\Lambda_{u'}=\Lambda'$.  

We claim that the map 
$u\mapsto u'$ defines a one-to-one correspondence
between smooth $(\alpha,\beta)$-lunes $u$ such that $\Lambda_u=\Lambda$
and smooth $(\alpha,\beta')$-lunes $u'$ such that $\Lambda_{u'}=\Lambda'$.
The map $u\mapsto u'$ is obviously injective.  To prove that
it is surjective we choose a smooth $(\alpha,\beta')$-lune 
$u'$ such that $\Lambda_{u'}=\Lambda'$.  Denote by 
$$
U'\subset\D
$$ 
the unique connected component of 
${u'}^{-1}(u_0(\D_\eps))$ that contains an arc in $\D\cap S^1$.
There are four cases as depicted in Figure~\ref{fig:deformation}.
In two of these cases (second and third row) we have 
$u'(U')\cap \alpha=\emptyset$ and hence $U'\cap\R=\emptyset$.  
In the casees where $u'(U')\cap \alpha\ne\emptyset$
it follows from an orientation argument (fourth row)
and from Lemma~\ref{le:one} (first row)
that $U'$ cannot intersect $\D\cap\R$.  
Thus we have shown that $U'$ does not intersect 
$\D\cap\R$ in all four cases.  This implies that $u'$ 
is in the image of the map $u\mapsto u'$.  
Hence the map $u\mapsto u'$ is bijective as claimed,
and hence 
$$
\nu(\Lambda)=\nu(\Lambda').
$$
This proves Lemma~\ref{le:move}.
\end{proof}

\begin{proof}[Proof of Theorems~\ref{thm:lune1} and~\ref{thm:lune2}] 
\phantomsection\label{proof:lune1}
Assume first that $\Sigma$ is simply connected.
We prove that~(iii) implies~(i) in Theorem~\ref{thm:lune1}. 
Let $\Lambda$ be a combinatorial $(\alpha,\beta)$-lune.
By Lemma~\ref{le:move}, reduce the number of intersection 
points of $\Lambda$, while leaving the number $\nu(\Lambda)$ 
unchanged.  Continue by induction until reaching an embedded 
combinatorial lune in $\Sigma$. By Proposition~\ref{prop:simple}, 
such a lune satisfies $\nu=1$. Hence $\nu(\Lambda)=1$.
In other words, there is a smooth $(\alpha,\beta)$-lune 
$u:\D\to\Sigma$, unique up to equivalence, 
such that $\Lambda_u=\Lambda$.
Thus we have proved that~(iii) implies~(i).
We have also proved, in the simply connected case, that
$u$ is uniquely determined by $\Lambda_u$ up to equivalence.
From now on we drop the hypothesis that $\Sigma$ is simply connected.

We prove Theorem~\ref{thm:lune2}.
\phantomsection\label{proof:lune2}
Let $u:\D\to\Sigma$ and $u':\D\to\Sigma$ be smooth
$(\alpha,\beta)$-lunes such that 
$$
\Lambda_u=\Lambda_{u'}.
$$
Let $\tu:\D\to\tSi$ and $\tu':\D\to\tSi$ be lifts to the universal
cover such that 
$$
\tu(-1)=\tu'(-1).  
$$
Then $\Lambda_{\tu}=\Lambda_{\tu'}$.
Hence, by what we have already proved, $\tu$ is equivalent to $\tu'$
and hence $u$ is equivalent to $u'$.  This proves Theorem~\ref{thm:lune2}.

We prove that~(i) implies~(ii) in Theorem~\ref{thm:lune1}.
Let $u:\D\to\Sigma$ be a smooth $(\alpha,\beta)$-lune and denote by
$\Lambda_u=:(x,y,A,B,\w)$ be its $(\alpha,\beta)$-trace.
Then 
$
\w(z)=\#u^{-1}(z)
$
is the counting function of $u$ and hence is nonnegative.  
For $0<t\le 1$ define the curve
$\lambda_t:[0,1]\to\RP^1$ by 
$$
du(-\cos(\pi s),t\sin(\pi s)) \lambda_t(s) 
:= \R \frac{\p}{\p s} u(-\cos(\pi s),t\sin(\pi s)),\quad
0\le s\le 1.
$$
For $t=0$ use the same definition for $0<s<1$ and extend 
the curve continuously to the closed interval $0\le s\le 1$.
Then
\begin{equation*}
\begin{split}
\lambda_0(s) &= du(-\cos(\pi s),0)^{-1}
T_{u(-\cos(\pi s),0)}\alpha,\\
\lambda_1(s) &= du(-\cos(\pi s),\sin(\pi s))^{-1}
T_{u(-\cos(\pi s),\sin(\pi s))}\beta,
\end{split}
\qquad
0\le s\le 1.
\end{equation*}
The Viterbo--Maslov index $\mu(\Lambda_u)$
is, by definition, the relative Maslov index of the 
pair of Lagrangian paths $(\lambda_0,\lambda_1)$,
denoted by $\mu(\lambda_0,\lambda_1)$
(see Definition~\ref{def:maslov} above or~\cite{VITERBO,RS2}).
\phantomsection\label{V5}\label{RS2e}  Hence it follows 
from the homotopy axiom for the relative Maslov index that
$$
\mu(\Lambda_u) 
= \mu(\lambda_0,\lambda_1)
= \mu(\lambda_0,\lambda_t)
$$ 
for every $t>0$.  Choosing $t$ sufficiently close
to zero we find that $\mu(\Lambda_u)=1$.

\bigbreak

We prove that~(iii) implies~(ii) in Theorem~\ref{thm:lune1}. 
If $\Lambda=(x,y,\w)$ is a combinatorial $(\alpha,\beta)$-lune, 
then $\w\ge0$ by~(I) in Definition~\ref{def:clune}. 
Moreover, by the trace formula~\eqref{eq:VM1}, we have
$$
\mu(\Lambda) = \frac{m_x(\Lambda)+m_y(\Lambda)}{2}.
$$
It follows from~(II) and~(III) 
in Definition~\ref{def:clune} that
$
m_x(\Lambda)=m_y(\Lambda)=1
$
and hence $\mu(\Lambda)=1$.  
Thus we have proved that~(iii) implies~(ii).

We prove that~(ii) implies~(iii) in Theorem~\ref{thm:lune1}. 
Let $\Lambda=(x,y,\w)$ be an $(\alpha,\beta)$-trace 
such that $\w\ge 0$ and $\mu(\Lambda)=1$.
Denote $\nu_\alpha:=\p\w|_{\alpha\setminus\beta}$ and 
$\nu_\beta:=-\p\w|_{\beta\setminus\alpha}$. 
Reversing the orientation of $\alpha$ or $\beta$, if necessary, 
we may assume that $\nu_\alpha\ge0$ and $\nu_\beta\ge 0$.
Let $\eps_x,\eps_y\in\{\pm1\}$ be the intersection indices 
of $\alpha$ and $\beta$ at $x,y$ with these orientations, 
and let 
$$
n_\alpha := \min\nu_\alpha\ge 0,\qquad
n_\beta:=\min\nu_\beta\ge 0.
$$ 
As before, denote by $m_x$ (respectively $m_y$) 
the sum of the four values of $\w$ encountered
when walking along a small circle surrounding~$x$ 
(respectvely~$y$).  Since the Viterbo--Maslov 
index of $\Lambda$ is odd, we have 
$\eps_x\ne\eps_y$ and thus $x\ne y$. 
This shows that~$\Lambda$ satisfies the arc condition
if and only if $n_\alpha=n_\beta=0$.

We prove that~$\Lambda$ satisfies~(II).
Suppose, by contradiction, that $\Lambda$ does
not satisfy~(II).  Then $\eps_x=-1$ and $\eps_y=1$.  
This implies that the values of~$\w$ near~$x$
are given by $k$, $k+n_\alpha+1$, 
$k+n_\alpha+n_\beta+1$, $k+n_\beta+1$
for some integer~$k$.  Since $\w\ge 0$ 
these numbers are all nonnegative.
Hence $k\ge 0$ and hence $m_x\ge 3$.
The same argument shows that $m_y\ge 3$
and, by the trace formula~\eqref{eq:VM1}, we have 
$
\mu(\Lambda)=(m_x(\Lambda)+m_y(\Lambda))/2\ge 3,
$
in contradiction to our hypothesis. 
This shows that $\Lambda$ satisfies~(II).

We prove that $\Lambda$ satisfies the arc condition and~(III).
By~(II) we have $\eps_x=1$. Hence the values of $\w$ near $x$
in counterclockwise order are given by 
$
k_x,
$
$
k_x+n_\alpha+1,
$
$
k_x+n_\alpha-n_\beta,
$
$
k_x-n_\beta
$
for some integer $k_x\ge n_\beta\ge0$.   
This implies
$$
m_x(\Lambda)=4k_x-2n_\beta+2n_\alpha+1
$$
and, similarly,
$
m_y(\Lambda)=4k_y-2n_\beta+2n_\alpha+1
$
for some integer $k_y\ge n_\beta\ge 0$.
Hence, by the trace formula~\eqref{eq:VM1}, we have 
$$
1 = \mu(\Lambda) 
= \frac{m_x(\Lambda)+m_y(\Lambda)}{2}
= k_x + (k_x-n_\beta) + k_y + (k_y-n_\beta) + 2n_\alpha + 1.
$$
Hence
$
k_x=k_y=n_\alpha=n_\beta=0
$ 
and so $\Lambda$ satisfies the arc condition and~(III). 
Thus we have shown that~(ii) implies~(i).
This proves Theorem~\ref{thm:lune1}.
\end{proof}

\begin{example}\label{ex:maslov}\rm
The arguments in the proof of Theorem~\ref{thm:lune1} 
can be used to show that, if~$\Lambda$ is an 
$(\alpha,\beta)$-trace with 
$
\mu(\Lambda)=1,
$
then (I)$\implies$(III)$\implies$(II).
Figure~\ref{fig:maslov1} shows three 
$(\alpha,\beta)$-traces that satisfy the arc condition 
and have Viterbo--Maslov index one but do not satisfy~(I); 
one that still satisfies~(II) and~(III), 
one that satisfies~(II) but not~(III),
and one that satisfies neither~(II) nor~(III).
Figure~\ref{fig:maslov2} shows an $(\alpha,\beta)$-trace
of Viterbo--Maslov index two that satisfies~(I) and~(III) 
but not~(II). Figure~\ref{fig:maslov3} shows 
an $(\alpha,\beta)$-trace of Viterbo--Maslov index three 
that satisfies~(I) and~(II) but not~(III).
\end{example}

\begin{figure}[htp]
\centering 
\includegraphics[scale=0.7]{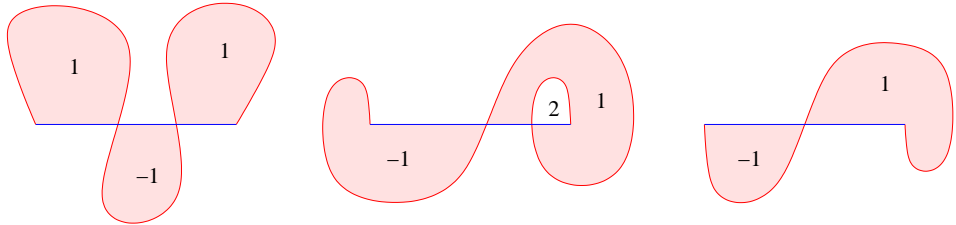} 
\caption{{Three $(\alpha,\beta)$-traces 
with Viterbo--Maslov index one.}}
\label{fig:maslov1}
\end{figure}

\begin{figure}[htp]
\centering 
\includegraphics[scale=0.7]{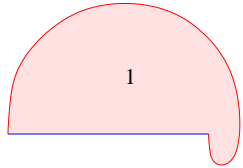} 
\caption{{An $(\alpha,\beta)$-trace 
with Viterbo--Maslov index two.}}
\label{fig:maslov2}
\end{figure}

\begin{figure}[htp]
\centering 
\includegraphics[scale=0.7]{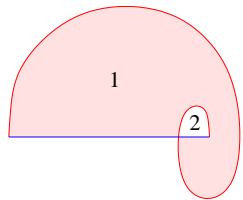} 
\caption{{An $(\alpha,\beta)$-trace 
with Viterbo--Maslov index three.}}
\label{fig:maslov3}
\end{figure}

We close this section with two results about lunes
that will be useful below.

\begin{proposition}\label{prop:luni}
Assume~\hyperlink{property_H}{(H)} and suppose that $\alpha$ and $\beta$ 
are noncontractible nonisotopic transverse embedded circles
and let 
$
x,y\in\alpha\cap\beta.
$
Then there is at most one $(\alpha,\beta)$-trace
from $x$ to $y$ that satisfies the arc condition.  
Hence, by Theorems~\ref{thm:arc} and~\ref{thm:lune2}, 
there is at most one equivalence class of smooth 
$(\alpha,\beta)$-lunes from $x$ to $y$.
\end{proposition}

\begin{proof}
Let 
$$
\alpha=\alpha_1\cup\alpha_2,\qquad
\beta=\beta_1\cup\beta_2, 
$$
where $\alpha_1$ and $\alpha_2$
are the two arcs of~$\alpha$ with endpoints $x$ and $y$,
and similarly for~$\beta$.  Assume that the quadruple 
$(x,y,\alpha_1,\beta_1)$ is an $(\alpha,\beta)$-trace. 
Then $\alpha_1$ is homotopic to $\beta_1$ with fixed endpoints.
Since $\alpha$ is not contractible,
$\alpha_2$ is not homotopic to~$\beta_1$ with fixed endpoints.
Since $\beta$ is not contractible,
$\beta_2$ is not homotopic to~$\alpha_1$ with fixed endpoints.
Since $\alpha$ is not isotopic to~$\beta$,
$\alpha_2$ is not homotopic to~$\beta_2$ with fixed endpoints.
Hence the quadruple $(x,y,\alpha_i,\beta_j)$ 
is not an $(\alpha,\beta)$-trace unless $i=j=1$. 
This proves Proposition~\ref{prop:luni}.
\end{proof}

The hypotheses that the loops $\alpha$ and $\beta$
are not contractible and not isotopic to each other cannot 
be removed in Proposition~\ref{prop:luni}.
A pair of isotopic circles with precisely two intersection 
points is an example.  Another example is a pair 
consisting of a contractible and a non-contractible loop,
again with precisely two intersection points.

\begin{proposition}\label{prop:embedded}
Assume~\hyperlink{property_H}{(H)}. 
If there is a smooth $(\alpha,\beta)$-lune 
then there is a primitive $(\alpha,\beta)$-lune.
\end{proposition}

\begin{proof}
The proof has three steps.

\medskip\noindent{\bf Step~1.}
{\it If $\alpha$ or $\beta$ is a contractible embedded circle
and $\alpha\cap\beta\ne\emptyset$ 
then there exists a primitive $(\alpha,\beta)$-lune.}

\medskip\noindent
Assume $\alpha$ is a contractible embedded circle.   
Then, by a theorem of 
Epstein~\cite{DBAE},\phantomsection\label{DBAE1} 
there exists an embedded closed disc 
$D\subset\Sigma$ with boundary $\p D=\alpha$.
Since $\alpha$ and $\beta$ intersect transversally, 
the set $D\cap\beta$ is a finite union of arcs.
Let $\cA$ be the set of all arcs $A\subset\alpha$
which connect the endpoints of an arc $B\subset D\cap\beta$.
Then~$\cA$ is a nonempty finite set, partially ordered
by inclusion.  Let $A_0\subset\alpha$ be a minimal 
element of $\cA$ and $B_0\subset D\cap\beta$ be the 
arc with the same endpoints as~$A_0$.  Then $A_0$ and $B_0$
bound a primitive $(\alpha,\beta)$-lune. 
This proves Step~1 when $\alpha$ is a contractible embedded circle.
When $\beta$ is a contractible embedded circle the proof is analogous. 

\medskip\noindent{\bf Step~2.}
{\it Assume $\alpha$ and $\beta$ are not contractible 
embedded circles.  If there exists a smooth $(\alpha,\beta)$-lune
then there exists an embedded $(\alpha,\beta)$-lune $u$
such that $u^{-1}(\alpha)=\D\cap\R$.}

\medskip\noindent
Let $v:\D\to\Sigma$ be a smooth $(\alpha,\beta)$-lune.
Then the set
$$
X := v^{-1}(\alpha) \subset \D
$$
is a smooth $1$-manifold with boundary
$
\p X = v^{-1}(\alpha)  \cap S^1.
$
The interval $\D\cap\R$ is one component of $X$
and no component of $X$ is a circle.
(If $X_0\subset X$ is a circle,
then $v|_{X_0}:X_0\to\Sigma$ is a contractible 
loop covering~$\alpha$ finitely many times.  
Hence, by Lemma~\ref{le:dbae2} in the appendix,
it would follow that $\alpha$ is a contractible embedded 
circle, in contradiction to the hypothesis of Step~2.)
Write 
$$
\p X = \{e^{\i\theta_1},\dots,e^{\i\theta_n}\},\qquad
\pi=\theta_1 > \theta_2 >\cdots > \theta_{n-1} > \theta_n=0.
$$
Then there is a permutation $\sigma\in S_n$ such that
the arc of $v^{-1}(\alpha)$ that starts at $e^{\i\theta_j}$
ends at $e^{\i\theta_{\sigma(j)}}$.  This permutation
satisfies $\sigma\circ\sigma = \id$, $\sigma(1)=n$, and
$$
j < k < \sigma(j)
\qquad\implies\qquad
j<\sigma(k)<\sigma(j).
$$
Hence, by induction, there exists a 
$j\in\{1,\dots,n-1\}$ such that $\sigma(j)=j+1$.
Let $X_0\subset X$ be the submanifold
with boundary points $e^{\i\theta_j}$
and $e^{\i\theta_{j+1}}$ and denote
$$
Y_0 := \left\{e^{\i\theta}\,|\,
\theta_{j+1}\le\theta\le\theta_j\right\}.
$$
Then the closure of the domain $\Delta\subset\D$ 
bounded by~$X_0$ and $Y_0$ is diffeomorphic 
to the half disc and $\overline{\Delta}\cap v^{-1}(\alpha)=X_0$.  
Hence there exists an orientation preserving embedding 
$\phi:\D\to\D$  that maps $\D\cap\R$ 
onto~$X_0$ and maps $\D\cap S^1$ onto $Y_0$. 
It follows that 
$$
u:=v\circ\phi:\D\to\Sigma
$$ 
is a smooth 
$(\alpha,\beta)$-lune such that 
$$
u^{-1}(\alpha)=\phi^{-1}(\overline{\Delta}\cap v^{-1}(\alpha)) 
= \phi^{-1}(X_0)=\D\cap\R.
$$
Moreover, $\Lambda_u=(x,y,A,B,\w)$ with
$$
x:=v(e^{\i\theta_j}),\qquad y:=v(e^{\i\theta_{j+1}}),\qquad
A:=v(X_0),\qquad B:=v(Y_0).
$$  
Since $A\cap B=\alpha\cap B=\{x,y\}$,
it follows from Proposition~\ref{prop:simple}
that $u$ is an embedding.
This proves Step~2.

\medskip\noindent{\bf Step~3.}
{\it Assume $\alpha$ and $\beta$ are not contractible embedded circles.
If there exists an embedded $(\alpha,\beta)$-lune~$u$
such that $u^{-1}(\alpha)=\D\cap\R$ then there exists
a primitive $(\alpha,\beta)$-lune.}

\medskip\noindent
Repeat the argument in the proof of Step~2 with 
$v$ replaced by $u$ and the set $v^{-1}(\alpha)$ replaced 
by the $1$-manifold 
$$
Y:=u^{-1}(\beta)\subset\D
$$
with boundary 
$$
\p Y=u^{-1}(\beta)\cap\R.
$$ 
The argument produces an arc $Y_0\subset Y$ 
with boundary points $a<b$ such that
the closed interval ${X_0:=[a,b]}$ intersects $Y$
only in the endpoints.  Hence the arcs $A_0:=u(X_0)$ and 
${B_0:=u(Y_0)}$ bound a primitive $(\alpha,\beta)$-lune.  
This proves Step~3 and Proposition~\ref{prop:embedded}.
\end{proof}

\newpage
\part*{III. Floer Homology}
\addcontentsline{toc}{part}{Part III. Floer Homology}


\section{Combinatorial Floer Homology}\label{sec:FLOER}

We assume throughout this section that $\Sigma$
is an oriented 2-manifold without boundary
and that $\alpha,\beta\subset\Sigma$ are 
noncontractible nonisotopic transverse embedded circles.  
We orient $\alpha$ and $\beta$.\hypertarget{Part_III}{}   
There are three ways we can count the number of points 
in their intersection:
\begin{itemize}
\item 
The {\bf numerical intersection number} 
$\num{\alpha}{\beta}$\index{numerical intersection number}
is the actual number of intersection 
points.\index{intersection number!numerical}
\item 
The \hypertarget{geometric_intersection_number}{{\bf geometric intersection number}}
$\geo{\alpha}{\beta}$\index{geometric intersection number} is defined as the 
minimum \index{intersection number}
of the numbers $\num{\alpha}{\beta'}$ over all embedded
loops $\beta'$ that are transverse to $\alpha$
and isotopic to $\beta$.\index{intersection number!geometric}
\item 
The {\bf algebraic intersection number} $\alg{\alpha}{\beta}$ is the sum
\index{algebraic intersection number}\index{intersection number!algebraic}
$$
\alpha\cdot\beta=\sum_{x\in\alpha\cap\beta}\pm1
$$
where the plus sign is chosen iff the orientations match
in the direct sum 
$
T_x\Sigma=T_x\alpha\oplus T_x\beta.
$
\end{itemize}
Note that
$
\Abs{\alg{\alpha}{\beta}}\le \geo{\alpha}{\beta}\le\num{\alpha}{\beta}.
$

\begin{theorem}\label{thm:floer1}
Define a chain complex\index{Floer chain complex}
$
\p:\CF(\alpha,\beta)\to \CF(\alpha,\beta)
$
by
\begin{equation}\label{eq:floer}
\CF(\alpha,\beta) =\bigoplus_{x\in\alpha\cap\beta}\Z_2 x,
\qquad
\p x = \sum_y n(x,y)y,
\end{equation}
where $n(x,y)$ denotes the number modulo two 
of equivalence classes of smooth $(\alpha,\beta)$-lunes 
from $x$ to $y$. Then
$$
\p\circ\p = 0.
$$
The homology group of this chain complex is denoted by
$$
\HF(\alpha,\beta) := \ker\p/\im\p
$$
and is called 
the \hypertarget{floer_homology}{{\bf Combinatorial Floer Homology}} 
of the pair $(\alpha,\beta)$.\index{combinatorial!Floer homology}
\index{Floer homology!combinatorial}
\index{Floer homology}
\end{theorem}

\begin{proof} 
See Section~\ref{sec:HEART} 
page~\pageref{proof:floer1}. 
\end{proof}

\bigbreak

\begin{theorem}\label{thm:floer2}
Combinatorial Floer homology is invariant under isotopy: 
If  $\alpha',\beta'\subset\Sigma$ are noncontractible 
transverse embedded circles such that $\alpha$ 
is isotopic to $\alpha'$ and $\beta$ is isotopic 
to $\beta'$ then 
$$
\HF(\alpha,\beta)\cong \HF(\alpha',\beta').
$$
\end{theorem}

\begin{proof} 
See Section~\ref{sec:ISOTOPY} page~\pageref{proof:floer2}.
\end{proof}

\begin{theorem}\label{thm:floer3}
Combinatorial Floer homology is isomorphic to
the original analytic Floer homology.
In fact, the two chain complexes agree.
\end{theorem}

\begin{proof} 
See Section~\ref{sec:LS} page~\pageref{proof:floer3}.
\end{proof}

\begin{corollary}\label{cor:nolune}
If $\geo{\alpha}{\beta} = \num{\alpha}{\beta}$
there is no smooth $(\alpha,\beta)$-lune.
\end{corollary}

\begin{proof}
If there exists a smooth $(\alpha,\beta)$-lune
then, by Proposition~\ref{prop:embedded},
there exists a primitive $(\alpha,\beta)$-lune
and hence there exists an embedded curve $\beta'$
that is isotopic to $\beta$ and satisfies
$\num{\alpha}{\beta'}<\num{\alpha}{\beta}$.
This contradicts our hypothesis.  
\end{proof}

\begin{corollary}\label{cor:HF}
\index{Floer homology!dimension}
$\dim\HF(\alpha,\beta) = \geo{\alpha}{\beta}$.
\end{corollary}

\begin{proof}
By Theorem~\ref{thm:floer2} we may assume that
$\num{\alpha}{\beta}=\geo{\alpha}{\beta}$.
In this case there is no $(\alpha,\beta)$-lune
by Corollary~\ref{cor:nolune}. 
Hence the Floer boundary operator is zero,
and hence the dimension of
$$
\HF(\alpha,\beta)\cong\CF(\alpha,\beta)
$$
is the geometric intersection number
$\geo{\alpha}{\beta}$.
\end{proof}

\begin{corollary}\label{cor:LUNE}
If $\geo{\alpha}{\beta} < \num{\alpha}{\beta}$
there is a primitive $(\alpha,\beta)$-lune.
\end{corollary}

\begin{proof}
By Corollary~\ref{cor:HF}, the
Floer homology group has dimension
$$
\dim\,\HF(\alpha,\beta)=\geo{\alpha}{\beta}.
$$
Since the Floer chain complex has dimension
$$
\dim\,\CF(\alpha,\beta)=\num{\alpha}{\beta}
$$ 
it follows that the Floer boundary operator is nonzero.
Hence there exists a smooth $(\alpha,\beta)$-lune
and hence, by Proposition~\ref{prop:embedded},
there exists a primitive $(\alpha,\beta)$-lune.
\end{proof}

\begin{remark}[{\bf Action Filtration}]
\label{rmk:action}\rm \index{action filtration}
Consider the space
$$
    \Om_{\alpha,\beta}
    := \left\{x\in\Cinf([0,1],\Sigma)\,|\,
       x(0)\in\alpha,\,x(1)\in\beta\right\}
$$
of paths connecting $\alpha$ to $\beta$.
Every intersection point $x\in\alpha\cap\beta$ determines
a constant path in $\Om_{\alpha,\beta}$ and hence a component 
of $\Om_{\alpha,\beta}$. In general, $\Om_{\alpha,\beta}$ is not 
connected and different intersection points may determine 
different components (see~\cite{POZNIAK} for 
the case $\Sigma=\T^2$).\phantomsection\label{POZNIAK}
By Proposition~\ref{prop:dbae} in Appendix~\ref{app:path},
each component of $\Om_{\alpha,\beta}$ is simply connected.
Now fix a positive area form $\om$ on $\Sigma$ and define
a $1$-form $\Theta$ on $\Om_{\alpha,\beta}$ by
$$
\Theta(x;\xi) := \int_0^1\om(\dot x(t),\xi(t))\,dt
$$
for $x\in\Om_{\alpha,\beta}$ and
$\xi\in T_x\Om_{\alpha,\beta}$.
This form is closed and hence exact.
Let
$$
\cA:\Om_{\alpha,\beta}\to\R
$$
be a function whose differential is $\Theta$.
Then the critical points of $\cA$ are the zeros of $\Theta$.
These are the constant paths and hence
the intersection points of $\alpha$ and $\beta$.
If $x,y\in\alpha\cap\beta$ belong to the same connected
component of $\Om_{\alpha,\beta}$ then
$$
\cA(x) - \cA(y) = \int u^*\om
$$
where $u:[0,1]\times[0,1]\to\Sigma$ is any smooth function
that satisfies 
$$
u(0,t)=x(t),\qquad u(0,1)=y(t),\qquad
u(s,0)\in\alpha,\qquad  u(s,1)\in\beta
$$
for all $s$ and $t$ (i.e.\ the map $s\mapsto u(s,\cdot)$
is a path in $\Om_{\alpha,\beta}$ connecting
$x$ to $y$). In particular, if $x$ and $y$ are the
endpoints of a smooth lune then $\cA(x)-\cA(y)$
is the area of that lune.  Figure~\ref{fig:large}
shows that there is no upper bound (independent 
of~$\alpha$ and~$\beta$ in fixed isotopy classes) 
on the area of a lune.
\begin{figure}[htp]
\centering 
\includegraphics[scale=0.9]{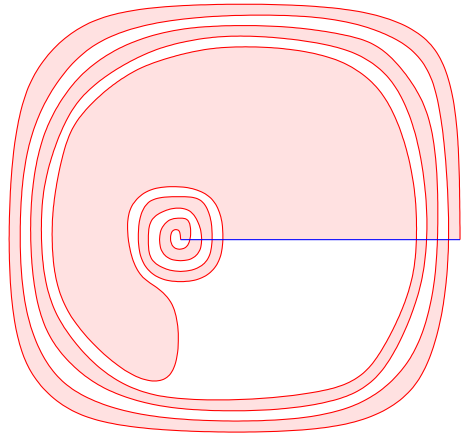} 
\caption{{A lune of large area.}}
\label{fig:large}
\end{figure}
\end{remark}

\begin{proposition} \label{prop:prec}
Define a relation $\prec$ on $\alpha\cap\beta$
by 
$
x\prec y
$
if and only if there is a sequence
$
x=x_0,x_1,\dots,x_n=y
$ 
in $\alpha\cap\beta$ such that, for each $i$,
there is a lune from $x_i$ to $x_{i-1}$
(see Figure~\ref{fig:lune3}).
Then $\prec$ is a strict partial order.
\end{proposition}

\begin{proof}
Since there is an $(\alpha,\beta)$-lune from $x_i$ to $x_{i-1}$
we have $\cA(x_{i-1})<\cA(x_i)$ for every $i$ and hence,
by induction, $\cA(x_0)<\cA(x_k)$.
\end{proof}
\begin{figure}[htp]
\centering 
\includegraphics[scale=1.4]{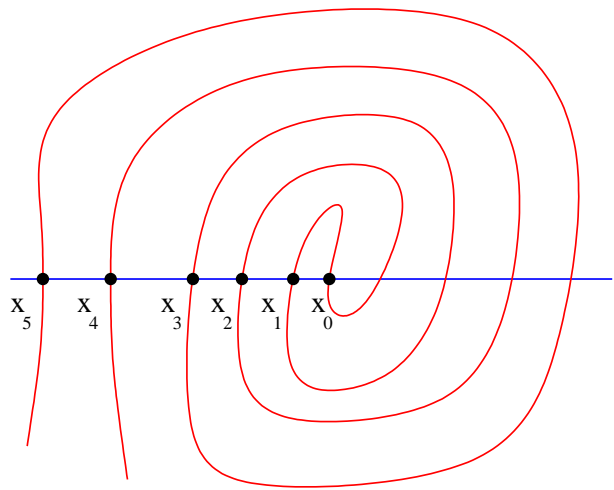}      
\caption{{Lunes from $x_i$ to $x_{i-1}$.}}
\label{fig:lune3}
\end{figure}

\begin{remark}[{\bf Mod Two Grading}]
\label{rmk:Z2grading} \rm \index{mod two grading} 
\index{grading!mod two}  The endpoints of a lune 
have opposite intersection indices.
Thus we may choose a $\Z/2\Z$-grading of $\CF(\alpha,\beta)$
by first choosing orientations of $\alpha$ and $\beta$
and then defining $\CF_0(\alpha,\beta)$ to be generated 
by the intersection points with intersection index $+1$ 
and $\CF_1(\alpha,\beta)$ to be generated 
by the intersection points with intersection index $-1$.
Then the boundary operator interchanges these two subspaces
and we have
$$
\alg{\alpha}{\beta} 
= \dim\HF_0(\alpha,\beta)-\dim\HF_1(\alpha,\beta).
$$
\end{remark}

\begin{remark}[{\bf Integer Grading}]\label{rmk:Zgrading}\rm
Since each component of the path space $\Om_{\alpha,\beta}$ 
is simply connected the $\Z/2\Z$-grading in 
Remark~\ref{rmk:Z2grading} \index{integer!grading} 
\index{grading!integer} can be refined to an integer grading. 
The grading is only well defined up to a global shift 
and the relative grading is given by the Viterbo--Maslov index. 
Then we obtain
$$
\alg{\alpha}{\beta} = \chi(\HF(\alpha,\beta)) 
= \sum_{i\in\Z}(-1)^i\dim\HF_i(\alpha,\beta).
$$
Figure~\ref{fig:lune3} shows that there is no upper or lower bound 
on the relative index in the combinatorial Floer 
chain complex. Figure~\ref{fig:lune5} 
shows that there is no upper bound on 
the dimension of $\CF_i(\alpha,\beta)$. 

In the case of the $2$-torus $\Sigma=\T^2=\R^2/\Z^2$
the shift in the integer grading can be fixed using 
Seidel's notion of a graded Lagrangian 
submanifold~\cite{SEIDEL}.\phantomsection\label{SEIDEL}
Namely, the tangent bundle of $\T^2$ is trivial 
so that each tangent space is equipped 
with a canonical isomorphism to $\R^2$.
Hence every embedded circle $\alpha\subset\T^2$ determines 
a map $\alpha\to\RP^1:z\mapsto T_z\alpha$.
A grading of $\alpha$ is a lift of this map 
to the universal cover of $\RP^1$. 
A choice of gradings for $\alpha$ and~$\beta$
can be used to fix an integer grading of 
the combinatorial Floer homology.
\end{remark}

\begin{figure}[htp]
\centering 
\includegraphics[scale=1.07]{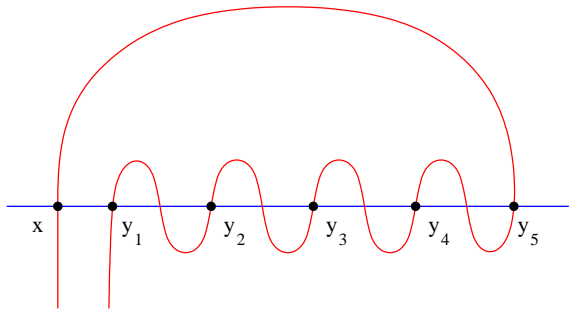}      
\caption{{Lunes from $x$ to $y_i$.}}
\label{fig:lune5}
\end{figure}

\begin{remark}[{\bf Integer Coefficients}]
\label{rmk:Zcoeff}\rm \index{integer!coefficients}
One can define combinatorial Floer homology with integer
coefficients as follows.  Fix an orientation of $\alpha$.
Then each $(\alpha,\beta)$-lune 
$
u:\D\to\Sigma
$
comes with a sign
$$
\nu(u) := \left\{\begin{array}{ll}
+1,&\mbox{if the arc }u|_{\D\cap\R}:\D\cap\R\to\alpha
\mbox{ is orientation preserving},\\
-1,&\mbox{if the arc }u|_{\D\cap\R}:\D\cap\R\to\alpha
\mbox{ is orientation reversing}.
\end{array}\right.
$$
Now define the chain complex by
$$
\CF(\alpha,\beta;\Z) =\bigoplus_{x\in\alpha\cap\beta}\Z x,
$$
and 
$$
\p x := \sum_{y\in\alpha\cap\beta} n(x,y;\Z)y,\qquad
n(x,y;\Z) := \sum_{[u]}\nu(u),
$$
where the sum runs over all equivalence classes $[u]$
of smooth $(\alpha,\beta)$-lunes from $x$ to $y$. 
The results of Section~\ref{sec:HEART} show that 
Theorem~\ref{thm:floer1} remains valid with this refinement, 
and the results of Section~\ref{sec:ISOTOPY} show 
that Theorem~\ref{thm:floer2} also remains valid.  
We will not discuss here any orientation issue for 
the analytic Floer theory and leave it to others 
to investigate the validity of Theorem~\ref{thm:floer3} 
with integer coefficients.
\end{remark}


\section{Hearts}\label{sec:HEART}

\begin{definition}\label{def:b-heart}\rm
Let $x,z\in\alpha\cap\beta$.  
A \hypertarget{broken_hearts}{{\bf broken $(\alpha,\beta)$-heart}}
\index{broken heart}\index{heart!broken} from $x$ to $z$ is a triple
$$
h = (u,y,v)
$$
such that $y\in\alpha\cap\beta$,
$u$ is a smooth $(\alpha,\beta)$-lune from $x$ to $y$,
and $v$ is a smooth $(\alpha,\beta)$-lune from $y$ to $z$.
The point $y$ is called the {\bf midpoint} of the heart.
\index{midpoint of a heart}\index{heart!midpoint of}
By Theorem~\ref{thm:lune2} the broken $(\alpha,\beta)$-heart 
$h$ is uniquely determined by the septuple
$$
\Lambda_h := (x,y,z,u(\D\cap\R),v(\D\cap\R),
u(\D\cap S^1),v(\D\cap S^1)).
$$
Two broken $(\alpha,\beta)$-hearts $h=(u,y,z)$ and $h'=(u',y',z')$
from $x$ to $z$ are called {\bf equivalent} if $y'=y$,  
\index{equivalent!broken hearts}\index{hearts!equivalent}
$u'$ is equivalent to $u$, and $v'$ is equivalent to~$v$.   
The equivalence class of $h$ is denoted by $[h]=([u],y,[v])$.
The set of equivalence classes of broken 
$(\alpha,\beta)$-hearts from $x$ to $z$
will be denoted by $\cH(x,z)$.
\end{definition}

\begin{proposition}\label{prop:heart}
Let $h=(u,y,v)$ be a broken $(\alpha,\beta)$-heart from
$x$ to $z$ and write
$
     \Lambda_h=:(x,y,z,A_{xy},A_{yz},B_{xy},B_{yz}).
$
Then exactly one of the following four alternatives
(see Figure~\ref{fig:hearts}) holds:
\begin{equation*}
\begin{split}
{\bf(a)}\;
A_{xy}\cap A_{yz}=\{y\},\;
B_{yz}\subsetneq B_{xy}.\quad
&
{\bf(b)}\;
A_{xy}\cap A_{yz}=\{y\},\;
B_{xy}\subsetneq B_{yz}.
\\
{\bf(c)}\;
B_{xy}\cap B_{yz}=\{y\},\;
A_{yz}\subsetneq A_{xy}.\quad
&
{\bf(d)}\;
B_{xy}\cap B_{yz}=\{y\},\;
A_{xy}\subsetneq A_{yz}.
\end{split}
\end{equation*}
\end{proposition}

\begin{figure}[htp]
\centering 
\includegraphics[scale=0.7]{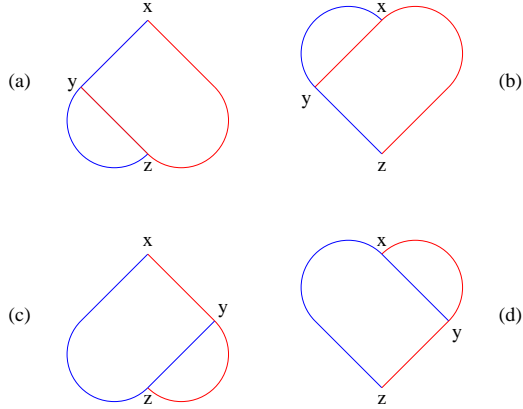}  
\caption{{Four broken hearts.}}\label{fig:hearts}
\end{figure}

\begin{proof}
The combinatorial $(\alpha,\beta)$-lunes
$\Lambda_{xy}:=\Lambda_u$ and $\Lambda_{yz}:=\Lambda_v$
have boundaries
$$
\p\Lambda_{xy} = (x,y,A_{xy},B_{xy}),\qquad
\p\Lambda_{yz} = (y,z,A_{yz},B_{yz})
$$
and their catenation 
$
\Lambda_{xz} := \Lambda_{xy}\#\Lambda_{yz}
= (x,z,\w_{xy}+\w_{yz})
$
has Viterbo--Maslov index two, by~\eqref{eq:VM2}.
Hence $m_x(\Lambda_{xy})+m_x(\Lambda_{yz})
+ m_z(\Lambda_{xy})+m_z(\Lambda_{yz}) = 4$,
by the trace formula~\eqref{eq:VM1}. Since 
$m_x(\Lambda_{xy})=m_z(\Lambda_{yz})=1$ 
this implies
\begin{equation}\label{eq:xyz}
m_x(\Lambda_{yz})+m_z(\Lambda_{xy})=2.
\end{equation}
By Proposition~\ref{prop:prec} we have $x\ne z$. 
Hence $x$ cannot be an endpoint of~$\Lambda_{yz}$.
Thus $m_x(\Lambda_{yz})\ge 2$ whenever $x\in A_{yz}\cup B_{yz}$
and $m_x(\Lambda_{yz})\ge 4$ whenever $x\in A_{yz}\cap B_{yz}$.
The same holds for $m_z(\Lambda_{xy})$. 
Hence it follows from~\eqref{eq:xyz} that
$\Lambda_{xy}$ and $\Lambda_{yz}$ satisfy
precisely one of the following conditions.
\begin{equation*}
\begin{split}
\mbox{\bf (a)}\;
x\notin A_{yz}\cup B_{yz},\;
z\in B_{xy}\setminus A_{xy}.\quad &
\mbox{\bf (b)}\;
x\in B_{yz}\setminus A_{yz},\;
z\notin A_{xy}\cup B_{xy}.\\
\mbox{\bf (c)}\;
x\notin A_{yz}\cup B_{yz},\;
z\in A_{xy}\setminus B_{xy}.\quad&
\mbox{\bf (d)}\;
x\in A_{yz}\setminus B_{yz},\;
z\notin A_{xy}\cup B_{xy}.
\end{split}
\end{equation*}
This proves Proposition~\ref{prop:heart}.
\end{proof}

Let $N\subset\C$ be an embedded convex half disc such that
$$
      [0,1]\cup \i[0,\eps)\cup (1+\i[0,\eps))\subset\p N,\qquad
      N\subset [0,1]+\i[0,1]
$$
for some $\eps>0$ and define
$$
      H := ([0,1]+\i[0,1])\cup(\i+N)\cup(1+\i-\i N).
$$
(See Figure~\ref{fig:s-heart}.)
The boundary of $H$ decomposes as
$$
      \p H = \p_0H\cup\p_1H
$$
where $\p_0H$ denotes the boundary arc from $0$ to $1+\i$
that contains the horizontal interval $[0,1]$
and $\p_1H$ denotes the arc from $0$ to $1+\i$
that contains the vertical interval $\i[0,1]$.
\begin{figure}[htp]
\centering 
\includegraphics[scale=0.7]{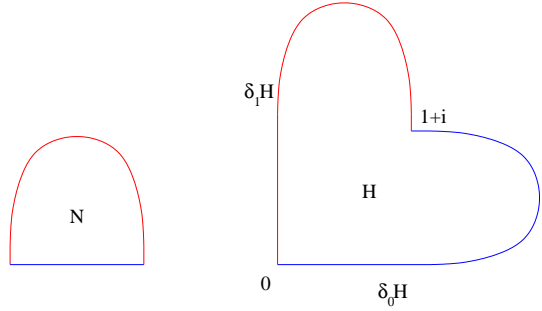}      
\caption{{The domains $N$ and $H$.}}
\label{fig:s-heart}
\end{figure}

\begin{definition}\label{def:s-heart}\rm
Let $x,z\in\alpha\cap\beta$.\index{heart!smooth}
A {\bf smooth $(\alpha,\beta)$-heart of type~(ac) from $x$ to $z$} 
\index{smooth!heart}\index{hearts!of type (ac) and (bd)}
is an orientation preserving immersion 
$
w:H\to\Sigma
$ 
that satisfies
$$
     w(0) = x,\qquad w(1+\i)=z,\qquad
     w(\p_0H)\subsetneq\alpha,\qquad
     w(\p_1H)\subsetneq\beta.
     \eqno(ac)
$$
Two smooth $(\alpha,\beta)$-hearts $w,w':H\to\Sigma$
are called {\bf equivalent} iff there exists 
an orientation \index{equivalent!smooth hearts}
\index{hearts!equivalent} preserving diffeomorphism 
$\chi:H\to H$ such that
$$
     \chi(0)=0,\qquad
     \chi(1+\i)=1+\i,\qquad
     w'=w\circ\chi.
$$
A {\bf smooth $(\alpha,\beta)$-heart of type~(bd) from $x$ to $z$}
is a smooth $(\beta,\alpha)$-heart of type~(ac) from $z$ to $x$.
Let $w$ be a smooth $(\alpha,\beta)$-heart
of type~(ac) from $x$ to $y$ and $h=(u,y,v)$ be a
broken $(\alpha,\beta)$-heart from
$x$ to $y$ of type~(a) or~(c).
The broken heart $h$ is called {\bf compatible} 
\index{compatible hearts}\index{hearts!compatible}
with the smooth heart $w$ if there exist
orientation preserving embeddings
$\phi:\D\to H$ and $\psi:\D\to H$ such that
\begin{equation}\label{eq:1-1}
      \phi(-1)=0,\qquad \psi(1)=1+\i,
\end{equation}
\begin{equation}\label{eq:phipsi}
      H = \phi(\D)\cup\psi(\D),\qquad
      \phi(\D)\cap\psi(\D)=\phi(\p\D)\cap\psi(\p\D),
\end{equation}
\begin{equation}\label{eq:uvw}
     u = w\circ\phi,\qquad
     v = w\circ\psi.
\end{equation}
\end{definition}

\begin{lemma}\label{le:heart}
Let $h=(u,y,v)$ be a broken $(\alpha,\beta)$-heart of type~(a),
write 
$$
\Lambda_h=:(x,y,z,A_{xy},A_{yz},B_{xy},B_{yz}),
$$
and define $A_{xz}$ and $B_{xz}$ by
$$
     A_{xz} := A_{xy}\cup A_{yz},\qquad
     B_{xy} =: B_{xz}\cup B_{yz},\qquad
     B_{xz}\cap B_{yz} = \{z\}.
$$
Let $w$ be a smooth $(\alpha,\beta)$-heart of type~(ac)
from $x$ to $z$ that is compatible with $h$ and
let $\phi,\psi:\D\to H$ be embeddings
that satisfy~(\ref{eq:1-1}), (\ref{eq:phipsi}),
and~(\ref{eq:uvw}). Then
\begin{equation}\label{eq:phi-z}
     \phi(e^{\i\theta_1}) = 1+\i,
\end{equation}
where $\theta_1\in[0,\pi]$ is defined by $u(e^{\i\theta_1})=z$,
and
\begin{equation}\label{eq:phicappsi}
     \phi(\D)\cap\psi(\D) = \psi(\D\cap S^1),
\end{equation}
\begin{equation}\label{eq:xz}
     w(\p_0H) = A_{xz},\qquad w(\p_1H)=B_{xz}.
\end{equation}
\end{lemma}

\begin{proof}
By definition of a smooth heart of type~(ac),
$w(\p_0H)$ is arc in $\alpha$ connecting
$x$ to $z$ and $w(\p_1H)$ is an arc in
$\beta$ connecting $x$ to $z$.
Moreover, by~(\ref{eq:uvw}),
$$
     \phi(\D\cap\R)\cup\psi(\D\cap\R)
     \subset w^{-1}(\alpha),\qquad
     \phi(\D\cap S^1)\cup\psi(\D\cap S^1)
     \subset w^{-1}(\beta).
$$
Now $w^{-1}(\alpha)$
is a union of disjoint embedded arcs,
and so is $w^{-1}(\beta)$.
One of the arcs in $w^{-1}(\alpha)$ contains $\p_0H$
and one of the arcs in $w^{-1}(\beta)$ contains $\p_1H$.
Since $\phi(-1)=0\in\p_0H$ and the arc
$$
w\circ\phi(\D\cap\R)=u(\D\cap\R)=A_{xy}
$$
does not contain $z$ we have
$$
     \phi(\D\cap\R)\subset \p_0H,\qquad
     A_{xy}\subset w(\p_0H).
$$
This implies the first equation in~(\ref{eq:xz}).
Since $w\circ\phi(\D\cap S^1)=u(\D\cap S^1)=B_{xy}$
is an arc containing $z$ we have
$$
     \p_1H\subset\phi(\D\cap S^1),\qquad
     w(\p_1H)\subset B_{xy}.
$$
This implies the second equation in~(\ref{eq:xz}).

We prove~(\ref{eq:phicappsi}). Choose $\theta_1\in[0,\pi]$
so that $u(e^{\i\theta_1})=z$ and denote
$$
     S_0 := \left\{e^{\i\theta}\,|\,
     0\le\theta\le\theta_1\right\},\qquad
     S_1 := \left\{e^{\i\theta}\,|\,
     \theta_1\le\theta\le\pi\right\},
$$
So that
$$
     \D\cap S^1 = S_0\cup S_1,\qquad
     u(S_0) = B_{yz},\qquad
     u(S_1) = B_{xz}.
$$
Hence $w\circ\phi(S_1)=u(S_1)=B_{xz}=w(\p_1H)$ and
$0\in\phi(S_1)\cap\p_1H$.  Since $w$ is an immersion
it follows that
$$
     \phi(S_1) = \p_1H,\qquad
     \phi(e^{\i\theta_1}) = 1+\i = \psi(1).
$$
This proves~(\ref{eq:phi-z}).
Moreover, by~(\ref{eq:uvw}),
$$
     w\circ\phi(S_0)
     = u(S_0)
     = B_{yz}
     = v(\D\cap S^1)
     = w\circ\psi(\D\cap S^1)
$$
and $1+\i$ is an endpoint of both
arcs $\phi(S_0)$ and $\psi(\D\cap S^1)$.
Since $w$ is an immersion it follows that
$$
     \psi(\D\cap S^1)
     = \phi(S_0)
     \subset\phi(\D)\cap\psi(\D).
$$
To prove the converse inclusion,
let $\zeta\in\phi(\D)\cap\psi(\D)$.
Then, by definition of a smooth heart,
$\zeta\in\phi(\p\D)\cap\psi(\p\D)$.
If $\zeta\in\phi(\D\cap\R)\cap\psi(\D\cap\R)$
then $w(\zeta)\in A_{xy}\cap A_{yz}=\{y\}$
and hence $\zeta=\psi(-1)\in\psi(\D\cap S^1)$.
Now suppose $\zeta=\phi(e^{\i\theta})\in\psi(\D\cap\R)$
for some $\theta\in[0,\pi]$.
Then we claim that $\theta\le\theta_1$.
To see this, consider the curve $\psi(\D\cap\R)$.
By~(\ref{eq:uvw}), this curve is mapped
to $A_{yz}$ under $w$ and it contains
the point $\psi(1)=1+\i$.  Hence
$\psi(\D\cap\R)\subset\p_0H\setminus\{0\}$.
But if $\theta>\theta_1$ then
$\phi(e^{\i\theta})\in\p_1H\setminus\{1+\i\}$
and this set does not intersect $\p_0H\setminus\{0\}$.
Thus we have proved that $\theta\le\theta_1$
and hence
$\phi(e^{\i\theta})\in\phi(S_0)=\psi(\D\cap S^1)$,
as claimed. This proves Lemma~\ref{le:heart}.
\end{proof}

\begin{proposition}\label{prop:heartbreaker}
{\bf (i)}
Let $h=(u,y,v)$ be a broken $(\alpha,\beta)$-heart 
of type~(a) or~(c) from $x$ to $z$. Then there exists 
a smooth $(\alpha,\beta)$-heart $w$ of type~(ac) from $x$ to $z$, 
unique up to equivalence, that is compatible with $h$.

\smallskip\noindent{\bf (ii)}
Let $w$ be a smooth $(\alpha,\beta)$-heart
of type~(ac) from $x$ to $z$.
Then there exists precisely one equivalence class of 
broken $(\alpha,\beta)$-hearts of type~(a) from~$x$ to~$z$
that are compatible with $w$, and precisely one equivalence class
of broken $(\alpha,\beta)$-hearts of type~(c) from $x$ to~$z$
that are compatible with $w$.
\end{proposition}

\begin{proof}
We prove~(i). Write
$$
\Lambda_h=:(x,y,z,A_{xy},A_{yz},B_{xy},B_{yz})
$$
and assume first that
$\Lambda_h$ satisfies~(a).
We prove the existence of $[w]$.
Choose a Riemannian metric on $\Sigma$ such that
the direct sum decompositions
$$
    T_y\Sigma = T_y\alpha\oplus T_y\beta,\qquad
    T_z\Sigma = T_z\alpha\oplus T_z\beta
$$
are orthogonal and $\alpha$ intersects small neighborhoods
of $y$ and $z$ in geodesic arcs.  Choose a diffeomorphism
$\gamma:[0,1]\to B_{yz}$ such that $\gamma(0)=y$ and $\gamma(1)=z$
and let $\zeta(t)\in T_{\gamma(t)}\Sigma$ be a unit normal
vector field pointing to the right. Then
there are orientation preserving embeddings
$\phi,\psi:\D\to H$ such that
$$
     \phi(\D) = ([0,1]+\i[0,1])\cup(\i+N),\qquad
     \phi(\D\cap\R) = [0,1],
$$
$$
     \psi(\D) = 1+\i-\i N,\qquad
     \psi(\D\cap S^1) = 1+\i[0,1],
$$
and
$$
     u\circ\phi^{-1}(1+s+\i t)=\exp_{\gamma(t)}(s\zeta(t))
$$
for $0\le t\le 1$ and small $s\le 0$, and
$$
     v\circ\psi^{-1}(1+s+\i t)=\exp_{\gamma(t)}(s\zeta(t))
$$
for $0\le t\le 1$ and small $s\ge 0$.
The function $w:H\to\Sigma$, defined by
$$
     w(z) := \left\{
     \begin{array}{rl}
     u\circ\phi^{-1}(z),&\mbox{if }z\in\phi(\D),\\
     v\circ\psi^{-1}(z),&\mbox{if }z\in\psi(\D),
     \end{array}\right.
$$
is a smooth $(\alpha,\beta)$-heart of type~(ac) from $x$ to $z$
that is compatible with~$h$.

We prove the uniqueness of $[w]$.
Suppose that $w':H\to\Sigma$ is another
smooth $(\alpha,\beta)$-heart of type~(ac)
that is compatible with $h$.
Let $\phi':\D\to H$ and $\psi':\D\to H$
be embeddings that satisfy~(\ref{eq:1-1}) and~(\ref{eq:phipsi})
and suppose that $w'$ is given by~(\ref{eq:uvw})
with $\phi$ and $\psi$ replaced by $\phi'$
and $\psi'$. Then
\begin{eqnarray*}
     w'\circ\phi'\circ\phi^{-1}(1+\i t)
&= &
     u\circ\phi^{-1}(1+\i t) \\
&= &
     v\circ\psi^{-1}(1+\i t) \\
&= &
     w'\circ\psi'\circ\psi^{-1}(1+\i t)
\end{eqnarray*}
for $0\le t\le 1$ and, by~(\ref{eq:1-1}) and~(\ref{eq:phi-z}),
$$
     \phi'\circ\phi^{-1}(1+\i)
     = 1+\i 
     = \psi'\circ\psi^{-1}(1+\i).
$$
Since $w'$ is an immersion it follows that
\begin{equation}\label{eq:phiphi'}
     \phi'\circ\phi^{-1}(1+\i t)
     = \psi'\circ\psi^{-1}(1+\i t)
\end{equation}
for $0\le t\le 1$.
Consider the map $\chi:H\to H$ given by
$$
     \chi(\zeta)
     := \left\{\begin{array}{rl}
        \phi'\circ\phi^{-1}(\zeta),&\mbox{for }\zeta\in\phi(\D),\\
        \psi'\circ\psi^{-1}(\zeta),&\mbox{for }\zeta\in\psi(\D).
        \end{array}\right.
$$
By~(\ref{eq:phiphi'}), this map is well defined.
Since $H=\phi'(\D)\cup\psi'(\D)$, the map $\chi$ is surjective.
We prove that $\chi$ is injective.
Let $\zeta,\zeta'\in H$ such that $\chi(\zeta)=\chi(\zeta')$.
If $\zeta,\zeta'\in\phi(\D)$ or $\zeta,\zeta'\in\psi(\D)$
then it is obvious that $\zeta=\zeta'$.  Hence assume
$\zeta\in\phi(\D)$ and $\zeta'\in\psi(\D)$.
Then
$
     \phi'\circ\phi^{-1}(\zeta)
     = \psi'\circ\psi^{-1}(\zeta')
$
and hence, by~(\ref{eq:phicappsi}),
$$
     \psi'\circ\psi^{-1}(\zeta')\in\psi'(\D\cap S^1).
$$
Hence $\zeta'\in\psi(\D\cap S^1)\subset\phi(\D)$,
so $\zeta$ and $\zeta'$ are both contained
in $\phi(\D)$, and it follows that $\zeta=\zeta'$.
Thus we have proved that $\chi:H\to H$ is
a homeomorphism. Since $w'=w\circ\chi$
it follows that $\chi$ is a diffeomorphism.
This proves~(i) in the case~(a).
The case~(c) follows by
reversing the orientation of $\Sigma$
and replacing $u$, $v$, $w$ by
$$
     u'=u\circ\rho,\qquad
     v'=v\circ\rho,\qquad
     w'(\zeta)=w(\i\bar\zeta),\qquad
     \rho(\zeta) = \frac{\i+\bar\zeta}{1+\i\bar\zeta}.
$$
Thus $\rho:\D\to\D$ is an orientation reversing
diffeomorphism with fixed points $\pm1$
that interchanges $\D\cap\R$ and $\D\cap S^1$.
The map $H\to H:\zeta\mapsto\i\bar\zeta$
is an orientation reversing diffeomorphism
with fixed points $0$ and $1+\i$
that interchanges $\p_0H$ and $\p_1H$.
This proves~(i).

We prove~(ii). Let $w:H\to\Sigma$ be a smooth
$(\alpha,\beta)$-heart of type~(ac) and denote
$$
     A_{xz} := w(\p_0H),\qquad
     B_{xz} := w(\p_1H).
$$
Let $\gamma\subset w^{-1}(\beta)$ be the
arc that starts at $1+\i$ and points into the
interior of $H$.  Let $\eta\in\p H$ denote the
second endpoint of $\gamma$.
Since $\beta$ has no self-intersections
we have
$
y := w(\eta) \in A_{xz}.
$
The arc $\gamma$ divides $H$ into two components,
each of which is diffeomorphic to~$\D$.
(See Figure~\ref{fig:b-heart}.)
\begin{figure}[htp]
\centering 
\includegraphics[scale=1.0]{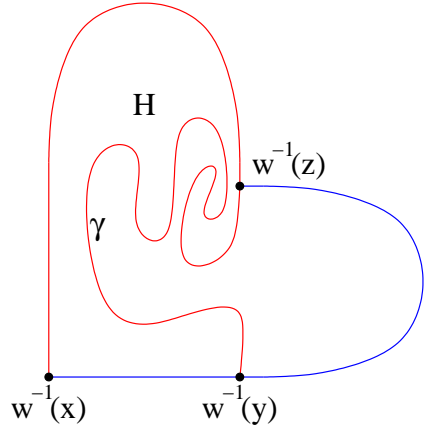}      
\caption{{Breaking a heart.}}
\label{fig:b-heart}
\end{figure}
The component which contains~$0$ gives rise to a smooth
$(\alpha,\beta)$-lune $u$ from $x$ to $y$
and the other component gives rise to a smooth
$(\alpha,\beta)$-lune $v$ from $y$ to $z$.
Let 
$$
\p\Lambda_u=:(x,y,A_{xy},B_{xy}),\qquad
\p\Lambda_v=:(y,z,A_{yz},B_{yz}). 
$$
Then
$$
B_{xy} = B_{xz}\cup B_{yz},\qquad
B_{yz} = w(\gamma).
$$
By Theorem~\ref{thm:arc}, $B_{xy}$ is an arc.
Hence $B_{yz}\subsetneq B_{xy}$,
and hence, by Proposition~\ref{prop:heart},
the broken $(\alpha,\beta)$-heart
$h=(u,y,v)$ from $x$ to $z$ satisfies~(a).
It is obviously compatible with $w$.
A similar argument,
using the arc $\gamma'\subset w^{-1}(\alpha)$
that starts at $1+\i$ and points into the
interior of $H$, proves the existence of a
broken $(\alpha,\beta)$-heart $h'\in\cH(x,z)$
that satisfies~(c) and is compatible with~$w$.
If $\th=(\tu,\ty,\tv)$ is any other broken $(\alpha,\beta)$-heart
of type~(a) that is compatible with~$w$,
then it follows from uniqueness in part~(i) that
$w^{-1}(\ty)=\eta$ is the endpoint of~$\gamma$,
hence $\ty=y$, and hence, by Proposition~\ref{prop:luni}, 
$\th$ is equivalent to~$h$.
This proves~(ii) and Proposition~\ref{prop:heartbreaker}.
\end{proof}

\begin{proof}[Proof of Theorem~\ref{thm:floer1}] 
\phantomsection\label{proof:floer1}
The square of the boundary operator is given by
$$
\p\p x = \sum_{z\in\alpha\cap\beta}n_H(x,z)z,
$$
where
$$
n_H(x,z):= \sum_{y\in\alpha\cap\beta}n(x,y)n(y,z)= \#\cH(x,z).
$$
By Proposition~\ref{prop:heartbreaker},
and the analogous result for smooth $(\alpha,\beta)$-hearts
of type~(bd), there is an involution
$
\tau:\cH(x,z)\to\cH(x,z)
$
without fixed points.
Hence $n_H(x,z)$ is even for all $x$ and $z$
and hence $\p\circ\p=0$.  This proves Theorem~\ref{thm:floer1}.
\end{proof}


\section{Invariance under Isotopy}\label{sec:ISOTOPY}

\begin{proposition}\label{prop:nolune}
Let $x,y,x',y'\in\alpha\cap\beta$ be distinct intersection
points such that
$$
    n(x,y)=n(x',y)=n(x,y')=1.
$$
Let $u:\D\to\Sigma$ be a smooth $(\alpha,\beta)$-lune from $x$ to~$y$
and assume that the boundary $\p\Lambda_u=:(x,y,A,B)$ 
of its $(\alpha,\beta)$-trace satisfies
$$
    A\cap\beta
    = \alpha\cap B
    = \{x,y\}.
$$
Then there is no smooth $(\alpha,\beta)$-lune
from $x'$ to $y'$, i.e.\ 
$$
    n(x',y')=0.
$$
Moreover, extending the arc from $x$ to $y$
(in either $\alpha$ or $\beta$) beyond $y$,
we encounter $x'$ before $y'$
(see Figures~\ref{fig:floer} and~\ref{fig:lune4})
and the two arcs $A'\subset\alpha$ and $B'\subset\beta$
from $x'$ to $y'$ that pass through $x$ and $y$ form
an $(\alpha,\beta)$-trace that satisfies the arc condition.
\end{proposition}

\begin{figure}[htp]
\centering 
\includegraphics[scale=0.7]{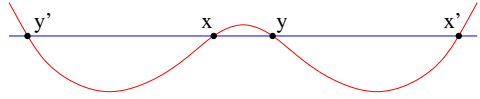}      
\caption{{No lune from $x'$ to $y'$.}}\label{fig:floer}
\end{figure}
\begin{figure}[htp]
\centering 
\includegraphics[scale=0.7]{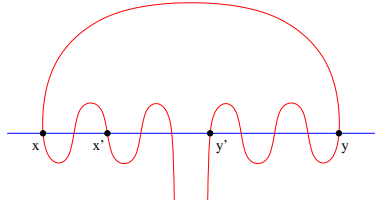}      
\caption{{Lunes from $x$ or $x'$ to $y$ or $y'$.}}
\label{fig:lune4}
\end{figure}

\begin{proof}
The proof has three steps.

\medskip\noindent{\bf Step~1.}
{\it There exist $(\alpha,\beta)$-traces
$\Lambda_{x'y}=(x',y,\w_{x'y})$,
$\Lambda_{yx}=(y,x,\w_{yx})$,
and $\Lambda_{xy'}=(x,y',\w_{xy'})$
with Viterbo--Maslov indices
$$
\mu(\Lambda_{x'y})=1,\qquad
\mu(\Lambda_{yx})=-1,\qquad
\mu(\Lambda_{xy'})=1
$$
such that}
$$
\w_{x'y}\ge 0,\qquad \w_{yx}\le 0,\qquad \w_{x,y'}\ge 0.
$$

\bigbreak

\medskip\noindent
By hypothesis, there exist smooth 
$(\alpha,\beta)$-lunes from $x'$ to $y$, 
from $x$ to $y$, and from $x$ to $y'$. 
By Theorem~\ref{thm:lune1} this implies the 
existence of combinatorial $(\alpha,\beta)$-lunes
$\Lambda_{x'y}=(x',y,\w_{x'y})$, 
$\Lambda_{xy}=(x,y,\w_{xy})$, and
$\Lambda_{xy'}=(x,y',\w_{xy'})$. 
To prove Step~1, reverse the direction 
of $\Lambda_{xy}$ to obtain the required
$(\alpha,\beta)$-trace $\Lambda_{yx}=(y,x,\w_{yx})$
with $\w_{yx}:=-\w_{xy}$.

\medskip\noindent{\bf Step~2.}
{\it Let $\Lambda_{x'y}$, $\Lambda_{yx}$, $\Lambda_{xy'}$ 
be as in Step~1. Then}
\begin{equation}\label{eq:m0}
\begin{split}
m_{x'}(\Lambda_{yx})
= m_{y'}(\Lambda_{yx})
&= 0, \\
m_x(\Lambda_{x'y})
= m_y(\Lambda_{xy'})
&= 0, \\
m_{y'}(\Lambda_{x'y})
= m_{x'}(\Lambda_{xy'})
&= 0.
\end{split}
\end{equation}

\medskip\noindent
By hypothesis, the combinatorial $(\alpha,\beta)$-lune
$\Lambda_{xy}=\Lambda_u$ has the boundary 
$\p\Lambda_{xy}=(x,y,A,B)$ with 
$A\cap\beta=\alpha\cap B=\{x,y\}$.
Hence $\w_{yx}=-\w_{xy}$ vanishes near every 
intersection point of $\alpha$ and $\beta$ 
other than $x$ and $y$.  This proves the first
equation in~\eqref{eq:m0}.
By~\eqref{eq:VM2}, the $(\alpha,\beta)$-trace
$\Lambda_{x'x}:=\Lambda_{x'y}\#\Lambda_{yx}$
has Viterbo--Maslov index index zero.
Hence, by the trace formula~\eqref{eq:VM1},
\begin{eqnarray*}
0 
&=& 
m_{x'}(\Lambda_{x'x})+m_x(\Lambda_{x'x}) \\
&=& 
m_{x'}(\Lambda_{x'y})+m_{x'}(\Lambda_{yx})
+ m_x(\Lambda_{x'y})+m_x(\Lambda_{yx}) \\
&=& 
m_x(\Lambda_{x'y}).
\end{eqnarray*}
Here the last equation follows from the fact that 
$m_{x'}(\Lambda_{yx})=0$, $m_{x'}(\Lambda_{x'y})=1$,
and $m_x(\Lambda_{yx})=-1$. The equation $m_y(\Lambda_{xy'})=0$
is proved by an analogous argument, using the fact that
$\Lambda_{yy'}:=\Lambda_{yx}\#\Lambda_{xy'}$ has
Viterbo--Maslov index zero.  
This proves the second equation in~\eqref{eq:m0}.
To prove the last equation in~\eqref{eq:m0}
we observe that the catenation
\begin{equation}\label{eq:x'y'}
\Lambda_{x'y'}:=\Lambda_{x'y}\#\Lambda_{yx}\#\Lambda_{xy'}
= \left(x',y',\w_{x'y}+\w_{yx}+\w_{xy'}\right)
\end{equation} 
has Viterbo--Maslov index one. 
Hence, by the trace formula~\eqref{eq:VM1}, 
\begin{eqnarray*}
2 
&=& 
m_{x'}(\Lambda_{x'y'})+m_{y'}(\Lambda_{x'y'}) \\
&=&
m_{x'}(\Lambda_{x'y}) + m_{x'}(\Lambda_{yx}) + m_{x'}(\Lambda_{xy'}) \\
&&
+\, m_{y'}(\Lambda_{x'y}) + m_{y'}(\Lambda_{yx}) + m_{y'}(\Lambda_{xy'}) \\
&=&
2 + m_{x'}(\Lambda_{xy'}) + m_{y'}(\Lambda_{x'y}).
\end{eqnarray*}
Here the last equation follows from the first equation 
in~\eqref{eq:m0} and the fact that 
$m_{x'}(\Lambda_{x'y})=m_{y'}(\Lambda_{xy'})=1$.
Since the numbers $m_{x'}(\Lambda_{xy'})$ 
and $m_{y'}(\Lambda_{x'y})$ are nonnegative, 
this proves the last equation in~\eqref{eq:m0}.
This proves Step~2. 

\medskip\noindent{\bf Step~3.}
{\it We prove the Proposition.}

\medskip\noindent
Let $\Lambda_{x'y}$, $\Lambda_{yx}$, $\Lambda_{xy'}$ 
be as in Step~1 and denote
\begin{equation*}
\begin{split}
\p\Lambda_{x'y} &=:(x',y,A_{x'y},B_{x'y}),\\
\p\Lambda_{yx} &=:(y,x,A_{yx},B_{yx}),\\
\p\Lambda_{xy'} &=:(x,y',A_{xy'},B_{xy'}).
\end{split}
\end{equation*}
By Step~2 we have
\begin{equation}\label{eq:notin}
x,y'\notin A_{x'y}\cup B_{x'y},\qquad
x',y'\notin A_{yx}\cup B_{yx},\qquad
x',y\notin A_{xy'}\cup B_{xy'}.
\end{equation}
In particular, the arc in $\alpha$ or $\beta$ from $y$ to $x'$ 
contains neither $x$ nor $y'$.  Hence it is the 
extension of the arc from $x$ to $y$ and we 
encounter $x'$ before $y'$ as claimed.
It follows also from~\eqref{eq:notin} that 
the catenation $\Lambda_{x'y'}$ in~\eqref{eq:x'y'} 
satisfies the arc condition and has boundary 
arcs 
$$
A_{x'y'}:=A_{x'y}\cup A_{yx}\cup A_{xy'},\qquad
B_{x'y'}:=B_{x'y}\cup B_{yx}\cup B_{xy'}.
$$
Thus $x\in A_{x'y'}$ and it follows from Step~2 that
$$
m_x(\Lambda_{x'y'})
= m_x(\Lambda_{x'y})+m_x(\Lambda_{yx})
+ m_x(\Lambda_{xy'})=0.
$$
This shows that the function $\w_{x'y'}$ is not 
everywhere nonnegative, and hence $\Lambda_{x'y'}$
is not a combinatorial $(\alpha,\beta)$-lune.
By Proposition~\ref{prop:luni} there is no other
$(\alpha,\beta)$-trace with endpoints $x',y'$ that
satisfies the arc condition.  Hence $n(x',y')=0$. 
This proves Proposition~\ref{prop:nolune}
\end{proof}

\begin{proof}[Proof of Theorem~\ref{thm:floer2}] 
\phantomsection\label{proof:floer2}
By composing with a suitable ambient isotopy
assume without loss of generality that $\alpha=\alpha'$.
Furthermore assume the isotopy $\{\beta_t\}_{0\le t\le1}$
with $\beta_0=\beta$ and $\beta_1=\beta'$
is generic in the following sense.
There exists a finite sequence of pairs
$(t_i,z_i)\in[0,1]\times\Sigma$ such that
$$
    0 < t_1 < t_2 < \cdots < t_m < 1,
$$
$\alpha\pitchfork_z\beta_t$ unless
$(t,z)=(t_i,z_i)$ for some $i$,
and for each $i$ there exists a coordinate
chart $U_i\to\R^2:z\mapsto(\xi,\eta)$ at $z_i$ such that
\begin{equation}\label{eq:remove}
    \alpha\cap U_i = \left\{\eta=0\right\},\qquad
    \beta_t\cap U_i = \left\{\eta=-\xi^2\pm(t-t_i)\right\}
\end{equation}
for $t$ near $t_i$. It is enough to consider two cases.

Case~1 is $m=0$.  In this case there exists
an ambient isotopy $\phi_t$ such that
$
    \phi_t(\alpha)=\alpha
$
and $\phi_t(\beta)=\beta_t$.  It follows that the
map $\CF(\alpha,\beta)\to\CF(\alpha,\beta')$
induced by $\phi_1:\alpha\cap\beta\to\alpha\cap\beta'$
is a chain isomorphism that identifies the boundary maps.

In Case~2 we have $m=1$,
the isotopy is supported near $U_1$,
and~\eqref{eq:remove} holds with the minus sign.
Thus there are two intersection points in $U_1$ for $t<t_1$, 
no intersection points in $U_1$ for $t>t_1$,
and all other intersection points of $\alpha$ 
and $\beta_t$ are independent of $t$.
Denote by $x,y\in\alpha\cap\beta$ the intersection 
points that cancel at time $t=t_1$ and choose the ordering
such that 
\begin{equation}\label{eq:one}
n(x,y)=1.
\end{equation}
Then $\alpha\cap\beta'=(\alpha\cap\beta)\setminus\{x,y\}$.
We prove in seven steps that
\begin{equation}\label{eq:2floer}
    n'(x',y') = n(x',y') + n(x',y)n(x,y')
\end{equation}
for $x',y'\in\alpha\cap\beta'$,
where $n(x',y')$ denotes the number
of $(\alpha,\beta)$-lunes from $x'$ to $y'$
and $n'(x',y')$ denotes the number
of $(\alpha,\beta')$-lunes from $x'$ to $y'$.

\medskip
\noindent{\bf Step~1.}
{\it
If there is no $(\alpha,\beta)$-trace from
$x'$ to $y'$ that satisfies the arc condition 
then~(\ref{eq:2floer}) holds.
}

\medskip
\noindent
In this case there is no $(\alpha,\beta_t)$-trace
from $x'$ to $y'$ that satisfies the arc condition for any~$t$.  
Hence it follows from Theorem~\ref{thm:arc} that 
$
{n(x',y')=n'(x',y')=0}
$ 
and it follows from Proposition~\ref{prop:nolune} that $n(x',y)n(x,y')=0$.  
Hence~(\ref{eq:2floer}) holds in this case.

\medskip\noindent{\bf Standing Assumptions, Part~1.}
{\it From now on we assume that there is an $(\alpha,\beta)$-trace 
from $x'$ to $y'$ that satisfies the arc condition.
Then there is an $(\alpha,\beta_t)$-trace from $x'$ to $y'$ 
that satisfies the arc condition for every $t$.
By Proposition~\ref{prop:luni}, this $(\alpha,\beta_t)$-trace
is uniquely determined by $x'$ and~$y'$.  We denote it by 
$\Lambda_{x'y'}(t)$ and its boundary by
$$
\p\Lambda_{x'y'}(t) =: (x',y',A_{x'y'},B_{x'y'}(t)).
$$
Choose a universal covering $\pi:\C\to\Sigma$
so that $\talpha=\R$ is a lift of $\alpha$, the map 
$\tz\mapsto\tz+1$ is a covering transformation, and 
$\pi$ maps the interval $[0,1)$ bijectively onto $\alpha$.
Let 
$$
\tLa_{x'y'}(t)=(\tx',\ty',\w_{\tLa_{x'y'}(t)})
$$ 
be a continuous family of lifts 
of the $(\alpha,\beta_t)$-traces to the universal 
cover $\pi:\C\to\Sigma$ with boundaries
$\p\tLa_{x'y'}(t) = (\tx',\ty',\tA_{x'y'},\tB_{x'y'}(t))$.}

\bigbreak

\medskip\noindent{\bf Step~2.}
{\it
The lift $\tLa_{x'y'}(t)$ satisfies condition~(II) (that the intersection
index of $\tA_{x'y'}$ and $\tB_{x'y'}(t)$ at $\tx'$ is $+1$ and at $\ty'$ is $-1$) 
and condition~(III) (that the winding numbers of $\tA_{x'y'}-\tB_{x'y'}(t)$
are zero or one near $\tx'$ and $\ty'$) either for all values of $t$ 
or for no value of $t$.
}

\medskip\noindent
The intersection indices of $\tA_{x'y'}$
and $\tB_{x'y'}(t)$ at $\tx'$ and $\ty'$,
and the winding numbers of $\tA_{x'y'}-\tB_{x'y'}(t)$
near  $\tx'$ and $\ty'$, are obviously independent of $t$.

\medskip
\noindent{\bf Step~3.}
{\it
If one of the arcs $A_{x'y'}(t)$ and $B_{x'y'}(t)$
does not pass through $U_1$ then~\eqref{eq:2floer}
holds.
}

\medskip
\noindent
In this case the winding numbers do not change sign as $t$
varies and hence, by Step~2, $\nu(\Lambda_{x'y'}(t))$
is independent of~$t$.  Hence 
$$
n(x',y')=n'(x',y').
$$
To prove equation~\eqref{eq:2floer} in this case, we must
show that one of the numbers $n(x',y)$ or $n(x,y')$ vanishes.
Suppose otherwise that 
$$
n(x',y)=n(x,y')=1.
$$
Since $n(x,y)=1$, by equation~\eqref{eq:one},
it follows from Proposition~\ref{prop:nolune} that the two arcs 
from $x'$ to $y'$ that pass through $U_1$ form an $(\alpha,\beta)$-trace
that satisfies the arc condition. Hence there are two
$(\alpha,\beta)$-traces from $x'$ to $y'$ that satisfy the 
arc condition, which is impossible by Proposition~\ref{prop:luni}.
This contradiction proves Step~3.

\begin{figure}[htp]
\centering 
\includegraphics[scale=0.9]{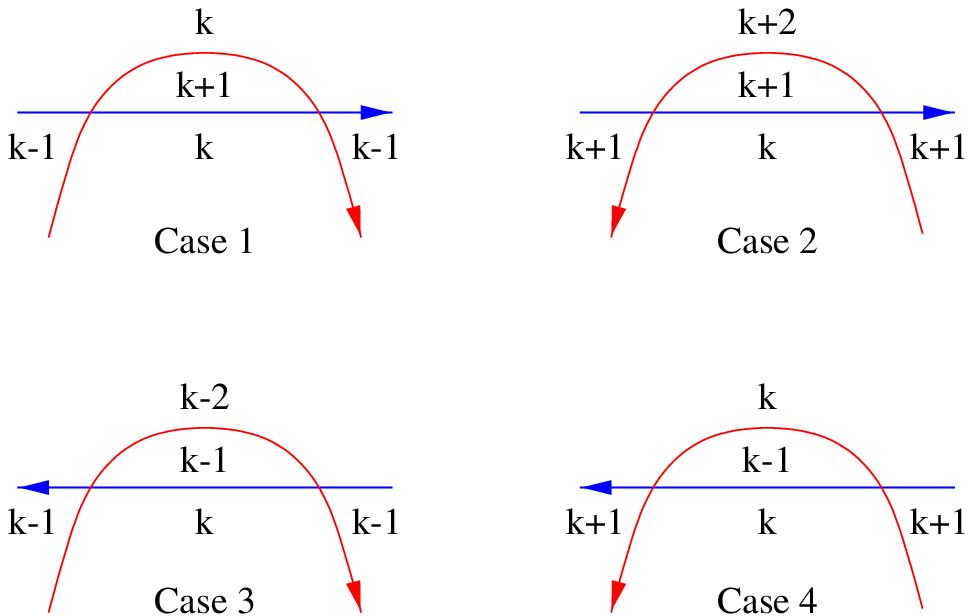}      
\caption{{The winding numbers of
$\tLa_{x'y'}$ in $\tU_1$ for $t=0$.}}
\label{fig:disappear}
\end{figure}

\medskip\noindent{\bf Standing Assumptions, Part~2.}
{\it From now on we assume that $A_{x'y'}$ and $B_{x'y'}(t)$
both pass through $U_1$.  Denote by $\tx$ and $\ty$ the 
unique lifts of $x$ and $y$, respectively, in $\tA_{x'y'}\cap\tB_{x'y'}(0)$.}

\medskip\noindent
Under this assumption the winding number of
$\tLa_{x'y'}(t)$ only changes in the area enclosed 
by the two arcs in the lift~$\tU_1$ of $U_1$ which contains
$\tx$ and $\ty$.  There are four cases, depending 
on the orientations of the two arcs from
$x'$ to $y'$. (See Figure~\ref{fig:disappear}.)
The next step deals with three of these cases.

\medskip
\noindent{\bf Step~4.}
{\it Assume that the orientation of $\alpha$ from $x'$ to $y'$
does not agree with one of the orientations from $x'$ to $y$
or from $x$ to $y'$, or else that this holds for $\beta$
(i.e.\ that one of the Cases~1,2,3 holds in 
Figure~\ref{fig:disappear}). Then~(\ref{eq:2floer}) holds.}

\medskip
\noindent
In Cases~1,2,3 the pattern of winding numbers shows
(for any value of $k$) that $\w_{\tLa_{x'y'}(t)}$
is either nonnegative for all values of $t$ or
is somewhere negative for all values of $t$.
Hence, by Step~2, $\nu(\Lambda_{x'y'}(t))$
is independent of $t$, and hence $n(x',y')=n'(x',y')$.
Moreover, by Proposition~\ref{prop:nolune}, we have
that in these cases $n(x,y')n(x',y)=0$.
Hence~(\ref{eq:2floer}) holds in the Cases~1,2,3.

\medskip\noindent{\bf Standing Assumptions, Part~3.}
{\it From now on we assume that Case~4 holds
in Figure~\ref{fig:disappear},
i.e.\ that the orientations of $\alpha$ and $\beta$
from $x'$ to $y'$ agree with the orientations from $x'$ to~$y$
and with the orientations from $x$ to $y'$.}

\medskip
\noindent{\bf Step~5.}
{\it
Assume Case~4 and $n(x',y)=n(x,y')=1$.
Then~(\ref{eq:2floer}) holds.
}

\medskip
\noindent
By Proposition~\ref{prop:nolune}, we have $n(x',y')=0$.
We must prove that $n'(x',y')=1$.
Let $\Lambda_{x'y}$ and $\Lambda_{xy'}$
be the $(\alpha,\beta)$-traces from $x'$ to $y$, 
respectively from $x$ to~$y'$,
that satisfy the arc condition and denote 
their lifts by $\tLa_{x'y}$ and $\tLa_{xy'}$.
By~\eqref{eq:x'y'}, 
\begin{equation}\label{eq:sum}
\w_{\tLa_{x'y'}(0)}(\tz)
= \w_{\tLa_{x'y}}(\tz) + \w_{\tLa_{xy'}}(\tz)
\qquad\mbox{for}\quad\tz\in\C\setminus\tU_1.
\end{equation}
Thus $\w_{\tLa_{x'y'}(0)}\ge 0$ in $\C\setminus\tU_1$.
Moreover, by Theorem~\ref{thm:lune1}, the lifts $\tLa_{x'y},\tLa_{xy'}$ 
have winding numbers zero in the regions labelled by $k$ and $k-1$
in Figure~\ref{fig:disappear}, Case~4.
Hence, by~\eqref{eq:sum}, we have $k=0$ and hence
$$
\w_{\tLa_{x'y'}(t)}\ge 0\qquad
\mbox{for}\qquad t>t_1.
$$
Thus we have proved that $\tLa_{x'y'}(t)$
satisfies~(I) for $t>t_1$.  Moreover, the Viterbo--Maslov 
index is given by
$$
     \mu(\tLa_{x'y'}(t))
     = \mu(\tLa_{x'y})
     + \mu(\tLa_{xy'})
     - \mu(\tLa_{xy})
     = 1.
$$
Hence, by Theorem~\ref{thm:lune1}, $\tLa_{x'y'}(t)$
is a combinatorial lune for $t>t_1$ and we have 
$n'(x',y')=1$.

\bigbreak

\medskip\noindent{\bf Step~6.}
{\it
Assume Case~4 and $n'(x',y')=1$.
Then~(\ref{eq:2floer}) holds.
}

\medskip
\noindent
The winding numbers of $\tLa_{x'y'}(1)$ are
nonnegative and hence we must have $k\ge 0$
in Figure~\ref{fig:disappear}, Case~4.
If $k>0$ then the winding numbers of
$\tLa_{x'y'}(0)$ are also nonnegative
and hence, by Step~2, $n(x',y')=1$.
Hence, by~\eqref{eq:one} and Proposition~\ref{prop:nolune}, 
one of the numbers $n(x',y)$ and $n(x,y')$ must vanish,
and hence~\eqref{eq:2floer} holds when $k>0$.

\begin{figure}[htp]
\centering 
\includegraphics[scale=1.0]{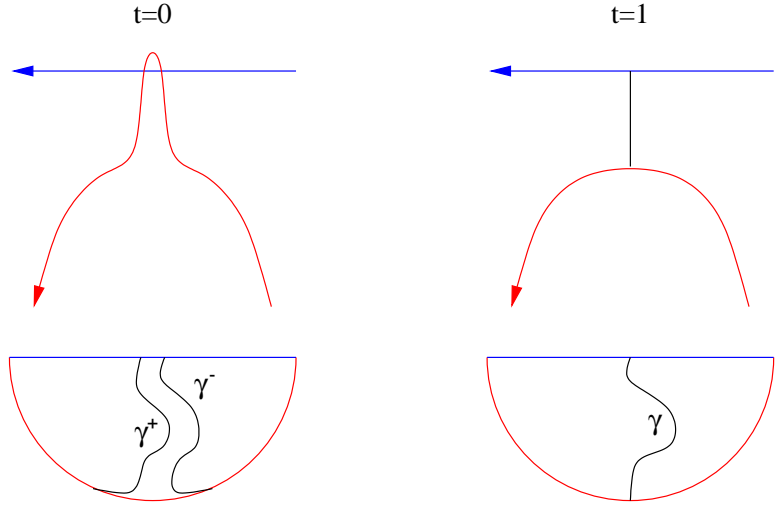}      
\caption{{A lune splits.}}
\label{fig:splitting}
\end{figure}

Now assume $k=0$. Then $n(x',y')=0$ and
we must prove that 
$$
n(x',y)=n(x,y')=1.
$$
To see this, we choose a smooth
$(\talpha,\tbeta')$-lune
$$
\tu':\D\to\C
$$ 
from $\tx'$ to $\ty'$.
Since $k=0$ this lune has precisely one preimage
in the region in $\tU_1$ where there winding number is $k+1=1\ne 0$ 
for $t=1$.  Choose an embedded arc 
$$
\tga:[0,1]\to\C
$$
in $\tU_1$ for $t=1$ connecting
the two branches of $\talpha$
and $\tbeta'$ such that $\tga$
intersects $\talpha$ and $\tbeta'$ only
at the endpoints (see Figure~\ref{fig:splitting}).
Then
$$
      \gamma:={{\tu'\,}}^{-1}(\tga)
$$
divides $\D$ into two components.
Deform the curve $\tbeta'$ along $\tga$
to a curve $\tbeta''$ that intersects $\talpha$ 
in two points $\tx'',\ty''\in\tU_1$.
The preimage of $\tbeta''$ under $\tu'$ contains 
two arcs $\gamma^-$ and $\gamma^+$, parallel to $\gamma$,
in the two components of $\D$.
This results in two half-discs contained in $\D$, 
and the  restriction of $\tu$ to these two half discs 
gives rise to two $(\talpha,\tbeta'')$-lunes, one from
$\tx'$ to $\ty''$ and one from $\tx''$ to $\ty'$
(see Figure~\ref{fig:splitting}).
Moreover $\tbeta''$ descends to an embedded
loop in $\Sigma$ that is isotopic to $\beta$ through
loops that are transverse to $\alpha$.
Hence $n(x',y)=n(x,y')=1$, as claimed.

\medskip
\noindent{\bf Step~7.}
{\it
Assume Case~4 and $n'(x',y')=n(x',y)n(x,y')=0$.
Then (\ref{eq:2floer}) holds.
}

\medskip
\noindent
We must prove that $n(x',y')=0$.
By Step~2, conditions~(II) and~(III) on
$\tLa_{x'y'}(t)$ are independent of $t$
and so we have that $n(x',y')=0$
whenever these conditions are not satisfied.
Hence assume that $\tLa_{x'y'}(t)$ satisfies~(II)
and~(III) for every $t$.
If $k>0$ in Figure~\ref{fig:disappear}
then condition~(I) also holds for every~$t$
and hence $n(x',y')=n'(x',y')=1$, a contradiction.
If $k\le0$ in Figure~\ref{fig:disappear}
then $\tLa_{x'y'}(0)$ does not satisfy~(I)
and hence $n(x',y')=0$, as claimed.

\medskip
\noindent
Thus we have established~(\ref{eq:2floer}).
Hence, by Lemma~\ref{le:floer},
$\HF(\alpha,\beta)$ is isomorphic
to $\HF(\alpha,\beta')$. 
This proves Theorem~\ref{thm:floer2}.
\end{proof}


\section{Lunes and Holomorphic Strips}\label{sec:LS}

We assume throughout that $\Sigma$ and 
$\alpha,\beta\subset\Sigma$
satisfy hypothesis~\hyperlink{property_H}{(H)}.  
We also fix a complex structure $J$ on $\Sigma$.
A {\bf ho\-lo\-mor\-phic $(\alpha,\beta)$-strip}
is a holomorphic map $v:\S\to\Sigma$
of finite energy such that
\begin{equation}\label{eq:BC}
v(\R)\subset\alpha,\qquad v(\R+\i)\subset\beta.
\end{equation}
It follows~\cite[Theorem~A]{ASYMPTOTIC}
\phantomsection\label{ASYMPTOTIC1} 
that the limits
\begin{equation}\label{eq:limit}
x=\lim_{s\to-\infty} v(s+\i t),\qquad
y=\lim_{s\to+\infty} v(s+\i t)
\end{equation}
exist; the convergence is exponential and uniform in $t$;
moreover $\p_sv$ and all its derivatives converge 
exponentially to zero as $s$ tends to $\pm\infty$. 
Call two holomorphic strips {\bf equivalent} 
\index{equivalent!holomorphic strips}
if they differ by a time shift. Every holomorphic strip $v$
has a {\bf Viterbo--Maslov index $\mu(v)$}, 
\index{Viterbo--Maslov index!of a holomorphic strip}
defined as follows.
Trivialize the complex line bundle $v^*T\Sigma\to\S$ such that
the trivialization converges to a frame of $T_x\Sigma$
as $s$ tends to $-\infty$ and to a frame of $T_y\Sigma$
as $s$ tends to $+\infty$ (with convergence uniform in $t$).
Then $s\mapsto T_{v(s,0)}\alpha$
and $s\mapsto T_{v(s,1)}\beta$ are Lagrangian
paths and their relative Maslov index is $\mu(v)$
(see~\cite{VITERBO} and~\cite{RS2}).
\phantomsection\label{V6}\label{RS2f}

\smallbreak

At this point it is convenient to introduce the notation
$$
\cM^\Floer(x,y;J) := 
\frac{
\left\{
v:\S\to\Sigma\,\bigg|\,\begin{array}{l}
v\mbox{ is a holomorphic $(\alpha,\beta)$-strip}\\
\mbox{from $x$ to $y$ with $\mu(v)=1$}
\end{array}\right\}
}
{\mbox{time shift}}
$$
for the moduli space of index one holomorphic strips 
from $x$ to $y$ up to time shift.  This moduli space depends 
on the choice of a complex structure $J$ on $\Sigma$.
We also introduce the notation
$$
\cM^\comb(x,y) := 
\frac{
\left\{
u:\D\to\Sigma\,\bigg|\,\begin{array}{l}
u\mbox{ is a smooth $(\alpha,\beta)$-lune}\\
\mbox{from $x$ to $y$}
\end{array}\right\}
}
{\mbox{isotopy}}
$$
for the moduli space of (equivalence classes of)
smooth $(\alpha,\beta)$-lunes.  This space is independent
of the choice of $J$.  We show that there is a bijection between
these moduli spaces for every pair $x,y\in\alpha\cap\beta$.

Given a smooth $(\alpha,\beta)$-lune
$u:\D\to\Sigma$, the Riemann mapping theorem
gives a unique homeomorphism $\phi_u:\D\to\D$
such that the restriction of $\phi_u$ to $\D\setminus\{\pm1\}$ 
is a diffeomorphism and
\begin{equation}\label{eq:phiu}
\phi_u(-1)=-1,\qquad 
\phi_u(0)=0,\qquad
\phi_u(1)=1,\qquad
\phi_u^*u^*J = \i.
\end{equation}
Let $g:\S\to\D\setminus\{\pm1\}$ be the
holomorphic diffeomorphism given by
\begin{equation}\label{eq:g}
g(s+\i t) := \frac{e^{(s+\i t)\pi/2}-1}{e^{(s+\i t)\pi/2}+1}.
\end{equation}
Then, for every $(\alpha,\beta)$-lune $u$,
the composition $v:=u\circ\phi_u\circ g$
is a holomorphic $(\alpha,\beta)$-strip.

\begin{theorem}\label{thm:lunestrip}
Assume~\hyperlink{property_H}{(H)}, 
let $x,y\in\alpha\cap\beta$, and 
choose a complex structure $J$ on $\Sigma$. 
Then the map $u\mapsto u\circ\phi_u\circ g$ 
induces a bijection 
\begin{equation}\label{eq:lunestrip}
\cM^\comb(x,y)\to\cM^\Floer(x,y;J):
[u]\mapsto[u\circ\phi_u\circ g]
\end{equation}
between the corresponding moduli spaces.
\end{theorem}

\begin{proof} 
See page~\pageref{proof:lunestrip} below.
\end{proof}

\bigbreak

The proof of Theorem~\ref{thm:lunestrip}  
relies on the asymptotic analysis of holomorphic 
strips in~\cite{ASYMPTOTIC}\phantomsection\label{ASYMPTOTIC2}
(see Appendx~\ref{app:asymptotic} for a summary)
and on an explicit formula for the Viterbo--Maslov index. 
For each intersection point $x\in\alpha\cap\beta$ 
we denote by $\theta_x\in(0,\pi)$ the angle 
from $T_x\alpha$ to $T_x\beta$ with respect 
to our complex structure $J$.  Thus
\begin{equation}\label{eq:thetax}
T_x\beta = \left(\cos(\theta_x)
+ \sin(\theta_x)J\right)T_x\alpha,\qquad
0<\theta_x<\pi.
\end{equation}
Fix a nonconstant holomorphic $(\alpha,\beta)$-strip
$v:\S\to\Sigma$ from $x$ to $y$.  
Choose a holomorphic coordinate chart 
$\psi_y:U_y\to\C$ on an open neighborhood 
$U_y\subset\Sigma$ of $y$ such that 
$\psi_y(y)=0$. 
By~\cite[Theorem~C]{ASYMPTOTIC}\phantomsection\label{ASYMPTOTIC3} 
(see also Corollary~\ref{cor:asymptotic})
there is a complex number $c_y$ 
and a integer $\nu_y(v)\ge1$ such that 
\begin{equation}\label{eq:myv}
\psi_y(v(s+\i t)) = c_ye^{-(\nu_y(v)\pi-\theta_y)(s+\i t)} 
+ O(e^{-(\nu_y(v)\pi-\theta_y+\delta)s})
\end{equation}
for some $\delta>0$ and all $s>0$ sufficiently close to $+\infty$.
The complex number $c_y$ belongs to the 
tangent space $T_0(\psi_y(\alpha\cap U_y))$ 
and the integer $\nu_y(v)$ is independent 
of the choice of the coordinate chart.

Now let us interchange $\alpha$ and $\beta$
as well as $x$ and $y$, and replace $v$ by the 
$(\beta,\alpha)$-holomorphic strip 
$$
s+\i t\mapsto v(-s+\i(1-t))
$$
from $y$ to $x$. Choose a holomorphic coordinate chart 
$\psi_x:U_x\to\C$ on an open neighborhood 
$U_x\subset\Sigma$ of $x$ such that $\psi_x(x)=0$. 
Using~\cite[Theorem~C]{ASYMPTOTIC}\phantomsection\label{ASYMPTOTIC4} 
again we find that there is a complex number $c_x$ and
an integer $\nu_x(v)\ge 0$ such that 
\begin{equation}\label{eq:mxv}
\psi_x(v(s+\i t)) = c_xe^{(\nu_x(v)\pi+\theta_x)(s+\i t)} 
+ O(e^{(\nu_x(v)\pi+\theta_x+\delta)s})
\end{equation}
for some $\delta>0$ and all $s<0$ sufficiently close to $-\infty$.
As before, the complex number $c_x$ belongs to the 
tangent space $T_0(\psi_x(\alpha\cap U_x))$ and
the integer $\nu_x(v)$ is independent 
of the choice of the coordinate chart.

Denote the set of critical points of $v$ by 
$$
C_v := \left\{z\in\S\,|\,dv(z)=0\right\}.
$$
It follows from Corollary~\ref{cor:asymptotic}~(i)
that this is a finite set. For $z\in C_v$ denote by 
$\nu_z(v)\in\N$ the order to which $dv$ vanishes at $z$.  
Thus the first nonzero term in the Taylor expansion 
of $v$ at $z$ (in a local holomorphic coordinate chart 
on~$\Sigma$ centered at $v(z)$) has order $\nu_z(v)+1$.

\begin{theorem}\label{thm:MASLOV}
Assume~\hyperlink{property_H}{(H)} and choose a 
complex structure $J$ on $\Sigma$. 
Let $x,y\in\alpha\cap\beta$ and 
$
v:\S\to\Sigma
$ 
be a nonconstant holomorphic $(\alpha,\beta)$-strip 
from $x$ to $y$.  Then the linearized operator $D_v$  
associated to this strip in Floer theory is surjective.  
Moreover, the Viterbo--Maslov index of $v$ is equal 
to the Fredholm index of $D_v$ and is given by the
{\bf index formula}\index{index formula}
\index{Viterbo--Maslov index!index formula}
\begin{equation}\label{eq:MASLOV}
\mu(v) = \nu_x(v)+\nu_y(v)+\sum_{z\in C_v\cap\p\S}\nu_z(v)
+ 2\sum_{z\in C_v\cap\INT(\S)}\nu_z(v).
\end{equation}
The right hand side in equation~\eqref{eq:MASLOV} 
is positive because all summands 
are nonnegative and ${\nu_y(v)\ge1}$.  
\end{theorem}

\begin{proof} 
See page~\pageref{proof:MASLOV} below.  
\end{proof}

The surjectivity statement in Theorem~\ref{thm:MASLOV}
has been observed by many authors. 
A proof for holomorphic polygons is contained in Seidel's 
book~\cite{SEIDEL1}.\phantomsection\label{SEIDEL1a}

We will need  a more general index formula 
(equation~\eqref{eq:index} below) which we explain next.  
Choose a Riemannian metric on $\Sigma$ that is compatible 
with the complex structure $J$.  This metric induces 
a Hermitian structure on the pullback tangent bundle 
$
v^*T\Sigma\to\S
$ 
and the Hilbert spaces $W^{1,2}(\S,v^*T\Sigma)$ 
and $L^2(\S,v^*T\Sigma)$ are understood with respect 
to this induced structure.  These Hilbert spaces 
are independent of the choice of the metric on $\Sigma$,
only their inner products depend on this choice.  
The linearized operator
$$
D_v:W_\BC^{1,2}(\S,v^*T\Sigma)\to L^2(\S,v^*T\Sigma)
$$
with 
$$
W_\BC^{1,2}(\S,v^*T\Sigma)
:=\left\{\hat v\in W^{1,2}(\S,v^*T\Sigma)\,\bigg|\,
\begin{array}{l}
\hat v(s,0)\in T_{v(s,0)}\alpha\;\forall s\in\R\\
\hat v(s,1)\in T_{v(s,1)}\beta\;\forall s\in\R
\end{array}\right\}
$$
is given by 
$$
D_v\hat v=\Nabla{s}\hat v+J\Nabla{t}\hat v
$$
for $\hat v\in W^{1,2}_\BC(\S,v^*T\Sigma)$, 
where $\nabla$ denotes the Levi-Civita connection.  
Here we use the fact that $\nabla J=0$ because $J$ 
is integrable.  We remark that, first, this is a 
Fredholm operator for every smooth map 
$v:\S\to\Sigma$ satisfying~\eqref{eq:BC} 
and~\eqref{eq:limit} (where the convergence 
is exponential and uniformly in $t$, and 
$\p_sv$, $\Nabla{s}\p_sv$, $\Nabla{s}\p_tv$ converge 
exponentially to zero as $s$ tends to $\pm\infty$). 
Second, the definition of the Viterbo--Maslov index 
$\mu(v)$ extends to this setting and it is equal 
to the Fredholm index of $D_v$ 
(see~\cite{RS3})\phantomsection\label{RS3a}
Third, the operator~$D_v$ is independent of the choice 
of the Riemannian metric whenever~$v$ is an 
$(\alpha,\beta)$-holomorphic strip. 

\bigbreak

Next we choose a unitary trivialization 
$$
\Phi(s,t):\C\to T_{v(s,t)}\Sigma
$$
of the pullback tangent bundle such that 
$$
\Phi(s,0)\R = T_{v(s,0)}\alpha,\qquad
\Phi(s,1)\R = T_{v(s,1)}\beta,
$$
and $\Phi(s,t)=\Psi_t(v(s,t))$ 
for $\abs{s}$ sufficiently large.
Here $\Psi_t$, $0\le t\le 1$, is a smooth family of unitary
trivializations of the tangent bundle over a neighborhood 
$U_x\subset \Sigma$ of $x$, respectively $U_y\subset\Sigma$ of $y$, 
such that $\Psi_0(z)\R=T_z\alpha$ for $z\in (U_x\cup U_y)\cap\alpha$
and $\Psi_1(z)\R=T_z\beta$ for $z\in (U_x\cup U_y)\cap\beta$.
Then
\begin{equation*}
\begin{split}
\cW &:=\Phi^{-1}W^{1,2}_\BC(\S,v^*T\Sigma) 
= \left\{\xi\in W^{1,2}(\S,\C)\,|\,
\xi(s,0),\xi(s,1)\in\R\;\forall s\in\R\right\},\\
\cH &:= \Phi^{-1}L^2(\S,v^*T\Sigma) = L^2(\S,\C).
\end{split}
\end{equation*}
The operator 
$
D_S:=\Phi^{-1}\circ D_v\circ \Phi:\cW\to\cH
$ 
has the form
\begin{equation}\label{eq:DS}
D_S\xi = \p_s\xi + \i\p_t\xi + S\xi
\end{equation}
where the function $S:\S\to\End_\R(\C)$ is given by
$$
S(s,t) := \Phi(s,t)^{-1}\Bigl(\Nabla{s}\Phi(s,t) 
+ J(v(s,t))\Nabla{t}\Phi(s,t)\Bigr).
$$
The matrix $\Phi^{-1}\Nabla{s}\Phi$ is skew-symmetric 
and the matrix $\Phi^{-1}J(v)\Nabla{t}\Phi$ is symmetric.  
Moreover, it follows from our hypotheses on $v$ 
and the trivialization that $S$ converges exponentially 
and $\Phi^{-1}\Nabla{s}\Phi$ as well as $\p_sS$ converge 
exponentially to zero as as $s$ tends to $\pm\infty$. 
The limits of $S$ are the symmetric matrix functions 
\begin{equation}\label{eq:Sxy}
\begin{split}
S_x(t) &:= \lim_{s\to-\infty}S(s,t) = \Psi_t(x)^{-1}J(x)\p_t\Psi_t(x),\\
S_y(t) &:= \lim_{s\to+\infty}S(s,t) = \Psi_t(y)^{-1}J(y)\p_t\Psi_t(y).
\end{split}
\end{equation}
Thus there exist positive constants $c$ and $\eps$ such that
\begin{equation}\label{eq:S}
\begin{split}
\Abs{S(s,t)-S_x(t)} + \Abs{\p_sS(s,t)} &\le ce^{\eps s},\\
\Abs{S(s,t)-S_y(t)} + \Abs{\p_sS(s,t)} &\le ce^{-\eps s}
\end{split}
\end{equation}
for every $s\in\R$.  This shows that the operator~\eqref{eq:DS} 
satisfies the hypotheses of~\cite[Lemma~3.6]{ASYMPTOTIC}. 
\phantomsection\label{ASYMPTOTIC5}
This lemma asserts the following.  Let $\xi\in\cW$
be a nonzero function in the kernel of $D_S$:
$$
\xi\in\cW,\qquad D_S\xi = \p_s\xi+\i\p_t\xi+S\xi = 0,\qquad \xi\ne 0.
$$
Then there exist nonzero functions $\xi_x,\xi_y:[0,1]\to\C$ and
positive real numbers $\lambda_x,\lambda_y,C,\delta$
such that
\begin{equation}\label{eq:asymp1}
\begin{split}
\i\dot\xi_x(t)+S_x(t)\xi_x(t) 
&= -\lambda_x\xi_x(t),\qquad \xi_x(0),\xi_x(1)\in\R\\
\i\dot\xi_y(t)+S_y(t)\xi_y(t) 
&= \lambda_y\xi_y(t),\qquad \xi_y(0),\xi_y(1)\in\R
\end{split}
\end{equation}
and
\begin{equation}\label{eq:asymp2}
\begin{split}
\Abs{\xi(s,t)-e^{\lambda_xs}\xi_x(t)} 
&\le Ce^{(\lambda_x+\delta)s},\qquad s\le 0\\
\Abs{\xi(s,t)-e^{-\lambda_ys}\xi_y(t)} 
&\le Ce^{-(\lambda_y+\delta)s},\qquad s\ge 0.
\end{split}
\end{equation}
We prove that there exist integers 
$\iota(x,\xi)\ge 0$ and $\iota(y,\xi)\ge 1$
such that
\begin{equation}\label{eq:asymp3}
\lambda_x = \iota(x,\xi)\pi + \theta_x,\qquad
\lambda_y = \iota(y,\xi)\pi - \theta_y.
\end{equation}
Here $\theta_x$ is chosen as above such that 
$$
T_x\beta=\exp(\theta_xJ(x))T_x\alpha,\qquad 0<\theta_x<\pi,
$$
and the same for $\theta_y$. To prove~\eqref{eq:asymp3},
we observe that the function 
$$
v_x(t):=\Psi_t(x)\xi_x(t)
$$ 
satisfies 
$$
J(x)\dot v_x(t)=-\lambda_xv_x(t),\qquad 
v_x(0)\in T_x\alpha,\qquad v_x(1)\in T_x\beta.
$$
Hence 
$
v_x(t) = \exp(t\lambda_xJ(x))v_x(0)
$ 
and this proves the first equation in~\eqref{eq:asymp3}.
Likewise, the function 
$
v_y(t):=\Psi_t(y)\xi_y(t)
$ 
satisfies $J(y)\dot v_y(t)=\lambda_yv_y(t)$
and $v_y(0)\in T_y\alpha$ and $v_y(1)\in T_y\beta$.  
Hence $v_y(t) = \exp(-t\lambda_yJ(y))v_y(0)$,
and this proves the second equation in~\eqref{eq:asymp3}. 

\begin{lemma}\label{le:index}
Suppose $S$ satisfies the asymptotic condition~\eqref{eq:S}
and let $\xi\in\cW$ be a smooth function with isolated zeros
that satisfies~\eqref{eq:asymp1}, \eqref{eq:asymp2}, 
and~\eqref{eq:asymp3}.  Then the Fredholm index of $D_S$ 
is given by the {\bf linear index formula}\index{linear index formula}
\index{Viterbo--Maslov index!linear index formula}
\begin{equation}\label{eq:index}
\mathrm{index}(D_S) 
= \iota(x,\xi)+\iota(y,\xi)+\sum_{z\in\p\S\atop\xi(z)=0}\iota(z,\xi)
+ 2\sum_{z\in\INT(\S)\atop\xi(z)=0}\iota(z,\xi).
\end{equation}
In the second sum $\iota(z,\xi)$ denotes the index of $z$
as a zero of $\xi$. In the first sum $\iota(z,\xi)$ denotes the degree
of the loop $[0,\pi]\to\RP^1:\theta\mapsto\xi(z+\eps e^{\i\theta})\R$
when $z\in\R$ and of the loop 
$[0,\pi]\to\RP^1:\theta\mapsto\xi(z-\eps e^{\i\theta})\R$
when $z\in\R+\i$; in both cases $\eps>0$ is chosen so small that
the closed $\eps$-neighborhood of $z$ contains 
no other zeros of $\xi$.
\end{lemma}

\begin{proof}
Since $\xi_x$ and $\xi_y$ have no zeros, by~\eqref{eq:asymp1}, 
it follows from equation~\eqref{eq:asymp2} that the zeros of $\xi$
are confined to a compact subset of $\S$.  Moreover the zeros 
of~$\xi$ are isolated and so the right hand side 
of~\eqref{eq:index} is a finite sum. Now let 
$\xi_0:\S\to\C$ be the unique solution of the equation
\begin{equation}\label{eq:xi0}
\i\p_t\xi_0(s,t)+S(s,t)\xi_0(s,t)=0,\qquad \xi_0(s,0)=1.
\end{equation}
Then 
\begin{equation}\label{eq:asymp4}
\begin{split}
\xi_{0,x}(t)&:=\lim_{s\to-\infty}\xi_0(s,t) = \Psi_t(x)^{-1}\Psi_0(x)1,\\
\xi_{0,y}(t)&:=\lim_{s\to+\infty}\xi_0(s,t) = \Psi_t(y)^{-1}\Psi_0(y)1.
\end{split}
\end{equation}
Thus the Lagrangian path 
$$
\R\to\RP^1:s\mapsto\Lambda_S(s):=\R\xi(s,1) 
$$
is asymptotic to the subspace 
$\Psi_1(x)^{-1}T_x\alpha$ as $s$ tends to $-\infty$ 
and to the subspace 
$\Psi_1(y)^{-1}T_y\alpha$ as $s$ tends to $+\infty$.
These subspaces are both transverse to $\R$.  
By the spectral-flow-equals-Maslov-index theorem 
in~\cite{RS3}\phantomsection\label{RS3b} 
the Fredholm index of $D_S$ is equal to the relative Maslov index 
of the pair $(\Lambda_S,\R)$:
\begin{equation}\label{eq:maspec}
\mathrm{index}(D_S) = \mu(\Lambda_S,\R).
\end{equation}
It follows from~\eqref{eq:asymp1} and~\eqref{eq:asymp4} that
\begin{equation}\label{eq:asymp5}
\frac{\xi_{0,x}(t)}{\xi_x(t)} = \frac{e^{-\i\lambda_xt}}{\xi_x(0)},\qquad
\frac{\xi_{0,y}(t)}{\xi_y(t)} = \frac{e^{\i\lambda_yt}}{\xi_y(0)}.
\end{equation}
Now let 
$
U= \bigcup_{\xi(z)=0}U_z\subset\S
$
be a union of open discs or half discs $U_z$ of radius 
less than one half, centered at the zeros $z$ of $\xi$, 
whose closures are disjoint.  Consider the smooth map 
$\Lambda:\S\setminus U\to\RP^1$ defined by
$$
\Lambda(s,t):=\xi_0(s,t)\overline{\xi(s,t)}\R.
$$
By~\eqref{eq:asymp5} this map converges, uniformly in $t$,
as $s$ tends to $\pm\infty$ with limits 
\begin{equation}\label{eq:asymp6}
\Lambda_x(t) := \lim_{s\to-\infty}\Lambda(s,t) = e^{-\i\lambda_xt}\R,\qquad
\Lambda_y(t) := \lim_{s\to+\infty}\Lambda(s,t) = e^{\i\lambda_yt}\R.
\end{equation}
Moreover, we have
\begin{equation*}
\begin{split}
\Lambda(s,1) = \xi_0(s,1)\R = \Lambda_S(s),\qquad &(s,1)\notin U,\\
\Lambda(s,0) = \xi_0(s,0)\R = \R,\quad\;\;\,\qquad &(s,0)\notin U.
\end{split}
\end{equation*}
If $z\in\mathrm{int}(\S)$ with $\xi(z)=0$ 
then the map $\Lambda_z:=\Lambda|_{\p U_z}$
is homotopic to the map 
$\p U_z\to\RP^1:s+\i t\mapsto\overline{\xi(s,t)}\R$.
Hence it follows from the definition of the index 
$\iota(z,\xi)$ that its degree is
\begin{equation}\label{eq:indint}
\deg(\Lambda_z:\p U_z\to\RP^1) = - 2\iota(z,\xi),\qquad 
z\in\mathrm{int}(\S),\qquad \xi(z)=0.
\end{equation}
If $z\in\p\S$ with $\xi(z)=0$, 
define the map $\Lambda_z:\p U_z\to\R P^1$ by 
$$
\Lambda_z(s,t):=\left\{\begin{array}{ll}
\xi_0(s,t)\overline{\xi(s,t)}\R,&\mbox{if }(s,t)\in\p U_z\setminus\p\S,\\
\xi_0(s,t)\R,&\mbox{if }(s,t)\in\p U_z\cap\p\S.
\end{array}\right.
$$
This map is homotopic to the map $(s,t)\mapsto\overline{\xi(s,t)}\R$
for $(s,t)\in\p U_z\setminus\p\S$ and $(s,t)\mapsto\R$
for $(s,t)\in\p U_z\cap\p\S$.  Hence it follows from 
the definition of the index $\iota(z,\xi)$ that its degree is
\begin{equation}\label{eq:indb}
\deg(\Lambda_z:\p U_z\to\RP^1) = - \iota(z,\xi),\qquad 
z\in\p\S,\qquad \xi(z)=0.
\end{equation}
Abbreviate $\S_T:=[-T,T]+\i[0,1]$ for $T>0$ sufficiently large. 
Since the map $\Lambda:\p(\S_T\setminus U)\to\RP^1$
extends to $\S_T\setminus U$ its degree is zero and it 
is equal to the relative Maslov index of the pair of 
Lagrangian loops $(\Lambda|_{\p(\S_T\cap U)},\R)$. Hence
\begin{eqnarray*}
0
&=&
\lim_{T\to\infty}\mu(\Lambda|_{\p(\S_T\setminus U)},\R) \\
&=&
\mu(\Lambda_y,\R) - \mu(\Lambda_x,\R) - \mu(\Lambda_S,\R)
- \sum_{z\in \S\atop \xi(z)=0} \mu(\Lambda_z,\R) \\
&=&
\iota(y,\xi) + \iota(x,\xi) - \mu(\Lambda_S,\R)
+ \sum_{z\in\p\S\atop\xi(z)=0}\iota(z,\xi)
+ 2\sum_{z\in\mathrm{int}(\S)\atop\xi(z)=0}\iota(z,\xi).
\end{eqnarray*}
Here the second equality follows from the additivity of the 
relative Maslov index for paths~\cite{RS2}.
\phantomsection\label{RS2g}  
It also uses the fact that, for ${z\in\R+\i}$ with ${\xi(z)=0}$,
the relative Maslov index of the pair 
$
(\Lambda_z|_{\p U_z\cap(\R+\i)},\R)=(\Lambda_S|_{\p U_z\cap(\R+\i)},\R)
$
appears with a plus sign when using the orientation of $\R+\i$
and thus compensates for the intervals in the relative Maslov 
index $-\mu(\Lambda_S,\R)$ that are not contained in the boundary 
of $\S\setminus U$.  Moreover, for $z\in\R$ with $\xi(z)=0$,
the relative Maslov index of the pair $(\Lambda_z|_{\p U_z\cap\R},\R)$
is zero.  The last equation follows from the 
formulas~\eqref{eq:asymp3} and~\eqref{eq:asymp6} for the 
first two terms and from~\eqref{eq:indb} and~\eqref{eq:indint}
for the last two terms.  With this understood, the linear index 
formula~\eqref{eq:index} follows from equation~\eqref{eq:maspec}.  
This proves Lemma~\ref{le:index}. 
\end{proof}

\begin{lemma}\label{le:onto}
The operator $D_S$ is injective whenever $\mathrm{index}(D_S)\le 0$
and is surjective whenever $\mathrm{index}(D_S)\ge 0$.
\end{lemma}

\begin{proof}
If $\xi\in\cW$ is a nonzero element in the kernel of $D_S$ 
then $\xi$ satisfies the hypotheses of Lemma~\ref{le:index}.
Moreover, every zero of $\xi$ has positive index by the argument 
in the proof of Theorem~C.1.10 
in~\cite[pages~561/562]{MS}.\phantomsection\label{MS}
Hence the index of $D_S$ is positive by the linear index formula
in Lemma~\ref{le:index}.  This shows that $D_S$ is injective 
whenever $\mathrm{index}(D_S)\le 0$.
If $D_S$ has nonnegative index then the formal adjoint operator 
$\eta\mapsto -\p_s\eta+\i\p_t\eta+S^T\eta$ has nonpositive index 
and is therefore injective by what we just proved.  
Since its kernel is the $L^2$-orthogonal complement 
of the image of $D_S$ it follows that $D_S$ is surjective.
This proves Lemma~\ref{le:onto}.
\end{proof}

\begin{proof}[Proof of Theorem~\ref{thm:MASLOV}] 
\phantomsection\label{proof:MASLOV}
The index formula~\eqref{eq:MASLOV} follows from 
the linear index formula~\eqref{eq:index} in Lemma~\ref{le:index} 
with $\xi:=\Phi^{-1}\p_sv$.  The index formula shows that $D_v$ 
has positive index for every nonconstant $(\alpha,\beta)$-holomorphic 
strip $v:\S\to\Sigma$. Hence $D_v$ is onto by Lemma~\ref{le:onto}.
This proves Theorem~\ref{thm:MASLOV}.
\end{proof}

\begin{proof}[Proof of Theorem~\ref{thm:lunestrip}]  
\phantomsection\label{proof:lunestrip}
The proof has four steps.

\medskip\noindent{\bf Step~1.}
{\it The map~\eqref{eq:lunestrip} is well defined.}

\medskip\noindent
Let $u,u':\D\to\Sigma$ be equivalent smooth $(\alpha,\beta)$-lunes
from $x$ to $y$. Then there is an orientation preserving 
diffeomorphism $\phi:\D\to\D$ such that $\phi(\pm1)=\pm1$
and $u':=u\circ\phi$.  Consider the holomorphic strips
$
v:=u\circ\phi_u\circ g:\S\to\Sigma
$
and
\begin{eqnarray*}
v'
&:=&
u'\circ\phi_{u'}\circ g \\
&=&
u\circ\phi\circ\phi_{u\circ\phi}\circ g \\
&=&
v\circ g^{-1}\circ\phi_u^{-1}\circ\phi\circ\phi_{u\circ\phi}\circ g.
\end{eqnarray*}
By the definition of $\phi_u$ 
we have $(u\circ\phi_u)^*J=\i$ and 
$(u\circ\phi\circ\phi_{u\circ\phi})^*J=\i$.
(See equation~\eqref{eq:phiu}.)
Hence the composition 
$\phi_u^{-1}\circ\phi\circ\phi_{u\circ\phi}:
\D\setminus\{\pm1\}\to\D\setminus\{\pm1\}$
is holomorphic and so is the composition
$
g^{-1}\circ\phi_u^{-1}\circ\phi\circ\phi_{u\circ\phi}\circ g:\S\to\S.
$
Hence this composition is given by a time shift and this proves Step~1.

\medskip\noindent{\bf Step~2.}
{\it The map~\eqref{eq:lunestrip} is injective.}

\medskip\noindent
Let $u,u':\D\to\Sigma$ be smooth $(\alpha,\beta)$-lunes from $x$
to $y$ and define $v:=u\circ\phi_u\circ g$ and $v':=u'\circ\phi_{u'}\circ g$.
(See equations~\eqref{eq:phiu} and~\eqref{eq:g}.)
Assume that $v'=v\circ\tau$ for a translation $\tau:\S\to\S$.
Then $u'=u\circ\phi$, where $\phi:\D\to\D$ is given by 
$\phi|_{\D\setminus\{\pm1\}}=\phi_u\circ g\circ\tau\circ g^{-1}\circ\phi_{u'}^{-1}$
and $\phi(\pm1)=\pm1$.  Since $u$ and $u'$ are immersions, it follows
that $\phi$ is a diffeomorphism of $\D$. This proves Step~2.

\medskip\noindent{\bf Step~3.}
{\it Every holomorphic $(\alpha,\beta)$-strip $v:\S\to\Sigma$
from $x$ to $y$ with Viterbo--Maslov index one is an immersion 
and satisfies $\nu_x(v)=0$ and $\nu_y(v)=1$.}

\medskip\noindent
This follows immediately from the index formula~\eqref{eq:MASLOV}
in Theorem~\ref{thm:MASLOV}.

\smallbreak

\medskip\noindent{\bf Step~4.}
{\it The map~\eqref{eq:lunestrip} is surjective.}

\medskip\noindent
Let $v:\S\to\Sigma$ be a holomorphic $(\alpha,\beta)$-strip 
from $x$ to $y$ with Viterbo--Maslov index one.  
By Step~3, $v$ is an immersion and satisfies 
$\nu_x(v)=0$ and ${\nu_y(v)=1}$.  Hence it follows 
from~\eqref{eq:myv} and~\eqref{eq:mxv} that
\begin{equation}\label{eq:mxyv}
\begin{split}
\psi_y(v(s+\i t)) 
&= c_ye^{-(\pi-\theta_y)(s+\i t)} 
+ O(e^{-(\pi-\theta_y+\delta)s}),\qquad s> T,\\
\psi_x(v(s+\i t)) 
&= c_xe^{\theta_x(s+\i t)} 
+ O(e^{(\theta_x+\delta)s}),\qquad\qquad\;\;\; s< -T,
\end{split}
\end{equation}
for $T$ sufficiently large.  
This implies that the composition 
$$
u':=v\circ g^{-1}:\D\setminus\{\pm1\}\to\Sigma
$$
($g$ as in equation~\eqref{eq:g})
is an immersion and extends continuously to $\D$ by $u'(-1):=x$
and $u'(1):=y$.   Moreover, locally near $z=-1$, the image of $u'$
covers only one of the four quadrants into which $\Sigma$
is divided by $\alpha$ and $\beta$ and the same holds near $z=1$.

We must prove that there exists a homeomorphism $\phi:\D\to\D$
such that

\smallskip\noindent{\bf (a)}
$\phi(\pm1)=\pm1$ and $\phi(0)=0$,

\smallskip\noindent{\bf (b)}
$\phi$ restricts to an orientation preserving diffeomorphism
of $\D\setminus\{\pm1\}$, 

\smallskip\noindent{\bf (c)}
the map $u := u'\circ\phi^{-1}:\D\to\Sigma$
is a smooth lune.

\medskip\noindent
Once $\phi$ has been found it follows from~(c) that 
$u\circ\phi=v\circ g^{-1}:\D\setminus\{\pm1\}\to\Sigma$
is holomorphic and hence $\phi^*u^*J=\i$.
Hence it follows from~(a) and~(b) that $\phi=\phi_u$ 
(see equation~\eqref{eq:phiu}) and this implies that the 
equivalence class $[v]\in\cM^\Floer(x,y;J)$ belongs to the 
image of our map~\eqref{eq:lunestrip} as claimed.

To construct $\phi$, choose a smooth function 
$\lambda:\R\to(1/2,\infty)$ such that
\begin{equation}\label{eq:lambda}
\lambda(s) = \left\{\begin{array}{ll}
\pi/2\theta_x,&\mbox{for }s\le -2,\\
1,&\mbox{for }s\ge -1,
\end{array}\right.,\qquad
\lambda(s) + \lambda'(s)s >0.
\end{equation}
(For example define 
$\lambda(s):=\pi/\theta_x-1+\left(\pi/\theta_x-2\right)/s$
for $-2\le s\le -1$ to obtain a piecewise smooth function
and approximate by a smooth function.)
Then the map $s\mapsto \lambda(s)s$ is a diffeomorphism
of $\R$.  Consider the sets 
$$
\K_\lambda := \left\{e^{s+\i t}\,\bigg|\,
0\le t\le\frac{\pi}{2\lambda(s)}\right\},\qquad
\K := \left\{e^{s+\i t}\,\bigg|\,
0\le t\le\frac{\pi}{2}\right\}.
$$
Denote their closures by $\overline{\K}_\lambda:=\K_\lambda\cup\{0\}$
and $\overline{\K}:=\K\cup\{0\}$ and define the homeomorphism 
$\rho_\lambda:\overline{\K}_\lambda\to\overline{\K}$ 
by $\rho_\lambda(0):=0$ and
$$
\rho_\lambda(e^{s+\i t}) := e^{(s+\i t)\lambda(s)},\qquad e^{s+\i t}\in\K_\lambda.
$$
It restricts to a diffeomorphism from $\K_\lambda$
to $\K$ and it satisfies $\rho_\lambda(\zeta) = \zeta^{\pi/2\theta_x}$ 
for $\abs{\zeta}\le e^{-2}$ and $\rho_\lambda(\zeta)=\zeta$ 
for $\abs{\zeta}\ge e^{-1}$.  Consider the map 
\begin{equation}\label{eq:wla}
w_\lambda:\overline{\K}_\lambda\to\Sigma,\qquad
w_\lambda(\zeta) 
:= u'\left(\frac{\rho_\lambda(\zeta)-1}{\rho_\lambda(\zeta)+1}\right).
\end{equation}
We claim that $w_\lambda:\overline{\K}_\lambda\to\Sigma$
is a $C^1$ immersion.  If $\zeta := e^{s+\i t} \in \K_\lambda$ 
with $\abs{\zeta} = e^s \le e^{-2}$ then by~\eqref{eq:g}
$$
\frac{\rho_\lambda(e^{s+\i t})-1}{\rho_\lambda(e^{s+\i t})+1}
= \frac{e^{(s+\i t)\pi/2\theta_x}-1}{e^{(s+\i t)\pi/2\theta_x}+1}
= g\left(\theta_x^{-1}(s+\i t)\right).
$$
Insert this as an argument in $u'$ and use the 
formula $u'\circ g=v$ to obtain
\begin{equation}\label{eq:wlav}
w_\lambda(e^{s+\i t}) 
= u'\left(\frac{\rho_\lambda(e^{s+\i t})-1}{\rho_\lambda(e^{s+\i t})+1}\right)
= v\left(\theta_x^{-1}(s+\i t)\right),\qquad
s\le -2.
\end{equation}
Thus 
\begin{equation*}
\begin{split}
\psi_x\left(w_\lambda(\zeta)\right)
&=
\psi_x\left(v\left(\theta_x^{-1}(s+\i t)\right)\right) \\
&= 
c_xe^{s+\i t} + O\bigl(e^{1+\delta/\theta_x)s}\bigr) \\
&= 
c_x\zeta + O\bigl(\abs{\zeta}^{1+\delta/\theta_x}\bigr)
\end{split}
\end{equation*}
for $\zeta=e^{s+\i t}\in\K_\lambda$ sufficiently small.
(Here the second equation follows from~\eqref{eq:mxyv}.)
Hence the map $w_\lambda:\overline{\K}_\lambda\to\Sigma$
in~\eqref{eq:wla} is complex differentiable at the origin and 
$d(\psi_x\circ w_\lambda)(0)=c_x$.  Next we prove that the derivative 
of $w_\lambda$ is continuous.  To see this, recall that
$\psi_x\circ w_\lambda$ and $\psi_x\circ v$ are 
holomorphic wherever defined and denote their complex derivatives
by $d(\psi_x\circ w_\lambda)$ and $d(\psi_x\circ v)$.  Differentiating 
equation~\eqref{eq:wlav} gives
$$
e^{s+\i t}d(\psi_x\circ w_\lambda)(e^{s+\i t})
= \theta_x^{-1}d(\psi_x\circ v)\left(\theta_x^{-1}(s+\i t)\right).
$$
By Corollary~\ref{cor:asymptotic}~(i),
$e^{-\theta_x(s+\i t)}d(\psi_x\circ v)(s+\i t)$ converges 
uniformly to $\theta_xc_x$ as $s$ tends to $-\infty$. 
Hence $d(\psi_x\circ w_\lambda)(\zeta)
=\theta_x^{-1}e^{-(s+\i t)}d(\psi_x\circ v)(\theta_x^{-1}(s+\i t))$
converges to ${c_x=d(\psi_x\circ w_\lambda)(0)}$ 
as $\zeta=e^{s+\i t}\in\K_\lambda$ tends to zero.  
Hence $w_\lambda$ is continuously differentiable 
near the origin and is a $C^1$ immersion as claimed.

\bigbreak

Now let $\sigma_\lambda:\overline{\K}_\lambda\to\D\setminus\{1\}$ 
be any diffeomorphism that satisfies 
$$
\sigma_\lambda(\zeta) = \frac{\zeta-1}{\zeta+1}\qquad\mbox{for}\quad \abs{\zeta}\ge 1.
$$
Define $\phi'':\D\to\D$ and $u'':\D\to\Sigma$ by $\phi''(1):=1$, $u''(1):=y$, and
$$
\phi''(z):= \frac{\rho_\lambda(\sigma_\lambda^{-1}(z))-1}
{\rho_\lambda(\sigma_\lambda^{-1}(z))+1},\qquad
u''(z) := u'\circ\phi''(z) = w_\lambda\circ\sigma_\lambda^{-1}(z)
$$
for $z\in\D\setminus\{1\}$. 
Then $\phi''=\id$ on $\left\{z\in\D\,|\,\RE\,z\ge0\right\}$,
the map ${\phi'':\D\to\D}$ satisfies~(a) and~(b),
and $u''|_{\D\setminus\{1\}}$ is an orientation preserving $C^1$ immersion.  
A similar construction near $y$ yields an orientation preserving $C^1$ immersion 
$u'''=u'\circ\phi''':\D\to\Sigma$ where $\phi''':\D\to\D$ satisfies~(a) and~(b).  
Now approximate $u'''$ in the $C^1$-topology by a smooth lune 
$u=u'\circ\phi:\D\to\Sigma$ to obtain the required map $\phi$.
This proves Step~4 and Theorem~\ref{thm:lunestrip}.
\end{proof}

\begin{proof}[Proof of Theorem~\ref{thm:floer3}] 
\phantomsection\label{proof:floer3}
By Theorem~\ref{thm:MASLOV}, the linearized operator $D_v$
in Floer theory is surjective for every $(\alpha,\beta)$-holomorphic 
strip $v$.  Hence there is a  boundary operator on the $\Z_2$ 
vector space $\CF(\alpha,\beta)$ as defined 
by Floer~\cite{FLOER1,FLOER2}\phantomsection\label{FLOERb}
in terms of the mod two count of $(\alpha,\beta)$-ho\-lo\-mor\-phic 
strips. By Theorem~\ref{thm:lunestrip} this boundary operator
agrees with the combinatorial one defined in terms 
of the mod two count of $(\alpha,\beta)$-lunes.  
Hence the combinatorial Floer homology of the pair 
$(\alpha,\beta)$ agrees with the analytic Floer homology defined by 
Floer.\index{Floer homology}\index{Floer homology!analytic}
This proves Theorem~\ref{thm:floer3}.
\end{proof}

\begin{remark}[{\bf Hearts and Diamonds}]\label{rmk:HL}\rm 
\index{hearts!and diamonds}\index{diamonds}
We have seen that the combinatorial boundary operator 
$\p$ on $\CF(\alpha,\beta)$ agrees with Floer's boundary 
operator by Theorem~\ref{thm:lunestrip}.  Thus we have  
two proofs that $\p^2=0$: the combinatorial proof using broken 
hearts and Floer's proof using his gluing construction. 
He showed (in much greater generality) 
that two $(\alpha,\beta)$-holomorphic strips
of index one (one from $x$ to $y$ and one from $y$ to $z$) 
can be glued together to give rise to a $1$-parameter family 
of $(\alpha,\beta)$-holomorphic strips (modulo time shift) 
of index two from~$x$ to~$z$.   
This one parameter family can be continued until it
ends at another pair of $(\alpha,\beta)$-holomorphic strips 
of index one (one from $x$ to some intersection point $y'$ 
and one from $y'$ to $z$).  These one parameter families 
are in one-to-one correspondence to 
$(\alpha,\beta)$-hearts from~$x$ to~$z$. 
This can be seen geometrically as follows. 
Each glued $(\alpha,\beta)$-holomorphic strip from~$x$ to~$z$
has a critical point on the $\beta$-boundary near $y$ for a broken
heart of type~(a). The $1$-manifold is parametrized by the 
position of the critical value. There is precisely 
one $(\alpha,\beta)$-holomorphic strip in this moduli space 
without critical point and an angle between $\pi$ and $2\pi$ at $z$.  
The critical value then moves onto the $\alpha$-boundary 
and tends towards $y'$ at the other end of the moduli space 
giving a broken heart of type~(c) (See Figure~\ref{fig:hearts}).  

Here is an explicit formula for the gluing construction in the two dimensional setting.  
Let $h=(u,y,v)$ be a broken 
$(\alpha,\beta)$-heart of type~(a) or~(b) from $x$ to $z$. 
(Types~(c) and~(d) are analogous with $\alpha$ and $\beta$
interchanged.) Denote the left and right upper quadrants by
$Q_L:=(-\infty,0)+\i(0,\infty)$ and $Q_R:=(0,\infty)+\i(0,\infty)$.
Define diffeomorphisms $\psi_L:Q_L\to\D\setminus\p\D$ and 
$\psi_R:Q_R\to\D\setminus\p\D$ by
$$
\psi_L(\zeta) := \frac{1+\zeta}{1-\zeta},\qquad
\psi_R(\zeta) := \frac{\zeta-1}{\zeta+1}.
$$
The extensions of these maps to M\"obius transformations
of the Riemann sphere are inverses of each other. 
Define the map $w:Q_L\cup Q_R\to\Sigma$ by 
$$
w(\zeta):= \left\{\begin{array}{ll}
u(\psi_L(\zeta)),&\mbox{for }\zeta\in Q_L,\\
v(\psi_R(\zeta)),&\mbox{for }\zeta\in Q_R.
\end{array}\right.
$$
The maps $u\circ\psi_L$ and $v\circ\psi_R$ send suitable intervals
on the imaginary axis starting at the origin to the same arc on $\beta$.
Modify $u$ and $v$ so that $w$ extends to a smooth map 
on the slit upper half plane $\mathbb{H}\setminus\i[1,\infty)$,
still denoted by $w$.   Here $\mathbb{H}\subset\C$
is the closed upper half plane.  
Define $\phi_\eps:\D\to\mathbb{H}$ by 
$$
\phi_\eps(z) := \frac{2\eps z}{1-z^2},\qquad z\in\D,\qquad 0<\eps<1.
$$
This map sends the open set $\mathrm{int}(\D)\cap Q_L$ 
diffeomorphically onto $Q_L$
and it sends $\mathrm{int}(\D)\cap Q_R$ 
diffeomorphically onto $Q_R$.
It also sends the interval $\i[0,1]$ 
diffeomorphically onto $\i[0,\eps]$. 
The composition $w\circ\phi_\eps:\mathrm{int}(\D)\to\Sigma$
extends to a smooth map on $\D$ denoted by $w_\eps:\D\to\Sigma$.
An explicit formula for $w_\eps$ is 
$$
w_\eps(z) = \left\{\begin{array}{ll}
u\left(\frac{1-z^2+2\eps z}{1-z^2-2\eps z}\right),&
\mbox{if }z\in\D\mbox{ and }\RE\,z\le 0,\\
v\left(\frac{-1+z^2+2\eps z}{1+z^2+2\eps z}\right),
&\mbox{if }z\in\D\mbox{ and }\RE\,z\ge 0.
\end{array}\right.
$$
The derivative of this map at every point
$z\ne\i$ is an orientation preserving isomorphism.
Its only critical value is the point 
$$
c_\eps := u\left(\frac{1+\i\eps}{1-\i\eps}\right)
= v\left(\frac{\i\eps-1}{\i\eps+1}\right)
\in\beta.
$$
Note that $c_\eps$ tends to $y=u(1)=v(-1)$ as $\eps$ tends to zero.
The composition of $w_\eps$ with a suitable $\eps$-dependent
diffeomorphism $\S\to\D\setminus\{\pm 1\}$ gives the 
required one-parameter family of glued holomorphic strips.
\end{remark}


\section{Further Developments}\label{sec:FD}

There are many directions in which the theory developed 
in the present memoir can be extended. Some of these 
directions and related work in the literature
are discussed below.

\subsection*{Floer Homology}

If one drops the hypothesis that the loops $\alpha$ and $\beta$
are not contractible and not isotopic to each other there are three 
possibilities.  In some cases the Floer homology groups are still well
defined and invariant under (Hamiltonian) isotopy,
in other cases invariance under isotopy breaks down,
and there are examples with $\p\circ\p\ne 0$, 
so Floer homology is not even defined.  All these phenomena
have their counterparts in combinatorial  Floer homology.

A case in point is that of two transverse embedded circles 
$
\alpha,\beta\subset\C
$ 
in the complex plane. In this case the boundary 
operator 
$$
\p:\CF(\alpha,\beta)\to\CF(\alpha,\beta)
$$
of Section~\ref{sec:FLOER} still satisfies $\p\circ\p=0$.
However, 
$
\ker\p=\im\,\p
$ 
and so the (combinatorial) Floer homology groups vanish.
This must be true because (combinatorial) Floer homology is 
still invariant under isotopy and the loops can be disjoined 
by a translation.  

A second case is that of two transverse embedded loops
in the sphere $\Sigma=S^2$.  Here the Floer homology groups
are nonzero when the loops intersect and vanish otherwise.
An interesting special case is that of two equators.
(Following Khanevsky we call an embedded circle
$\alpha\subset S^2$ an {\bf equator} 
when the two halves \index{equator} of 
$S^2\setminus\alpha$ have the same area.)  
In this case the combinatorial Floer homology groups 
do not vanish, but are only invariant under Hamiltonian isotopy.  
This is an example of the monotone case for Lagrangian
Floer theory (see Oh~\cite{OH}),\phantomsection\label{OHb}
and the theory  developed by Biran--Cornea 
applies~\cite{BC}.\phantomsection\label{BC}
For an interesting study of diameters (analogues 
of equators for discs) see 
Khanevsky~\cite{K}.\phantomsection\label{K}

\bigbreak

A similar case is that of two noncontractible transverse 
embedded loops 
$$
\alpha,\beta\subset\Sigma
$$ 
that are isotopic to each other.
Fix an area form on $\Sigma$. If $\beta$ is Hamiltonian
isotopic to $\alpha$ then the combinatorial Floer homology groups 
do not vanish and are invariant under Hamiltonian isotopy, 
as in the case of two equators on $S^2$.  
If $\beta$ is non-Hamiltonian isotopic to $\alpha$ 
(for example a distinct parallel copy), then the Floer
homology groups are no longer invariant under
Hamiltonian isotopy, as in the case of two embedded circles 
in $S^2$ that are not equators.  However, if we take account 
of the areas of the lunes by introducing combinatorial 
Floer homology with coefficients in an appropriate Novikov ring, 
the Floer homology groups will be invariant under Hamiltonian 
isotopy. When the Floer homology groups vanish it is
interesting to give a combinatorial description of the 
relevant torsion invariants 
(see Hutchings--Lee~\cite{HL1,HL2}).\phantomsection\label{HL}
This involves an interaction between lunes and annuli.

A different situation occurs when $\alpha$ is not contractible and
$\beta$ is contractible.  In this case
$$
\p\circ\p=\id
$$ 
and one can prove this directly in the combinatorial setting.
For example, if~$\alpha$ and $\beta$ intersect in precisely
two points $x$ and $y$ then there is precisely one lune from
$x$ to $y$ and precisely one lune from $y$ to $x$. 
They are embedded and their union is the disc encircled 
by $\beta$  rather than a heart as in Section~\ref{sec:HEART}.  
In the analytical setting this disc bubbles off in the moduli space 
of index two holomorphic strips from $x$ to itself.
This is a simple example of the obstruction theory developed
in great generality by 
Fukaya--Oh--Ohta--Ono~\cite{FOOO}.\phantomsection\label{FOOO1}

\subsection*{Moduli Spaces}

Another direction is to give a combinatorial description
of all holomorphic strips, not just those of index one.
The expected result is that they are uniquely determined,
up to translation, by their $(\alpha,\beta)$-trace 
$$
\Lambda=(x,y,\w)
$$ 
with $\w\ge0$,  the positions of the critical values,
suitable monodromy data, and the angles at infinity. 
(See Remark~\ref{rmk:HL} for a discussion of the 
Viterbo--Maslov index two case.)  This can be viewed 
as a natural generalization of Riemann--Hurwitz theory.  
For inspiration see the work of Okounkov and Pandharipande 
on the Gromov--Witten theory of 
surfaces~\cite{OP1,OP2,OP3,OP4}.\phantomsection\label{OP}

\subsection*{The Donaldson Triangle Product}

Another step in the program, already discussed 
in~\cite{DESILVA},\phantomsection\label{DESILVA2} 
is the combinatorial description of the product 
structures\index{Floer homology!triangle product}
$$
\HF(\alpha,\beta)\otimes\HF(\beta,\gamma)\to\HF(\alpha,\gamma)
$$
for triples of noncontractible, pairwise nonisotopic, 
and pairwise transverse embedded loops in a closed 
oriented $2$-manifold $\Sigma$.  The combinatorial 
setup involves the study of immersed triangles in $\Sigma$.  
When the triangle count is infinite, for example
on the $2$-torus, the definition of the product 
requires Floer homology with coefficients in Novikov rings.
The proof that the resulting map
$$
\CF(\alpha,\beta)\otimes\CF(\beta,\gamma)\to\CF(\alpha,\gamma)
$$
on the chain level is a chain homomorphism is based on similar 
arguments as in Section~\ref{sec:HEART}.  The proof that the product 
on homology is invariant under isotopy is based on similar arguments as 
in Section~\ref{sec:ISOTOPY}.  A new ingredient is 
the phenomenon that $\gamma$ can pass over an intersection point
of $\alpha$ and $\beta$ in an isotopy.  In this case the number 
of intersection points does not change but it is necessary to 
understand how the product map changes on the chain level.
The proof of associativity requires the study of immersed
rectangles and uses similar arguments as in Section~\ref{sec:HEART}.

In the case of the $2$-torus the study of triangles gives rise to 
Theta-functions as noted by Kontsevich~\cite{KONTSEVICH}.
\phantomsection\label{KONTSEVICH}
This is an interesting, and comparatively easy, 
special case of homological mirror symmetry. 

\subsection*{The Fukaya Category}

A natural extension of the previous discussion is 
to give a combinatorial description of the Fukaya 
category~\cite{FOOO}.\phantomsection\label{FOOO2}
A directed version of this category was described 
by Seidel~\cite{SEIDEL1}.\phantomsection\label{SEIDEL1b}  
In dimension two the directed Fukaya category is associated 
to a finite ordered collection
$$
\alpha_1,\alpha_2,\dots,\alpha_n\subset\Sigma
$$
of noncontractible, pairwise nonisotopic, and pairwise transverse
embedded loops in $\Sigma$. Interesting examples of such tuples 
arise from vanishing cycles of Lefschetz fibrations over the disc
with regular fiber $\Sigma$ (see~\cite{SEIDEL1}). 
\phantomsection\label{SEIDEL1c}  

The Fukaya category, on the combinatorial level, 
involves the study of immersed polygons. 
Some of the results in the present memoir (such as the combinatorial
techniques in Sections~\ref{sec:HEART} and~\ref{sec:ISOTOPY},
the surjectivity of the Fredholm operator, and the formula 
for the Viterbo--Maslov index in Section~\ref{sec:LS})
extend naturally to this setting.  On the other hand the algebraic 
structures are considerably more intricate for $A^\infty$ categories.  
The combinatorial approach has been used to compute the derived
Fukaya category of a surface by 
Abouzaid~\cite{A},\phantomsection\label{A}
and to establish homological mirror symmetry for punctured
spheres by Abouzaid--Auroux--Efimov--Katzarkov--Orlov~\cite{AAEKO}
\phantomsection\label{AAEKO} and for a genus two surface by 
Seidel~\cite{SEIDEL2}.\phantomsection\label{SEIDEL2}

\newpage
\appendix
\part*{Appendices}
\addcontentsline{toc}{part}{Appendices}


\section{The Space of Paths}\label{app:path} 

We assume throughout that $\Sigma$ is a connected oriented smooth
$2$-manifold without boundary and 
$
\alpha,\beta\subset\Sigma
$
are two embedded loops.  Let
$$
    \Om_{\alpha,\beta}
    := \left\{x\in\Cinf([0,1],\Sigma)\,|\,
       x(0)\in\alpha,\,x(1)\in\beta\right\}
$$
denote the space of paths connecting $\alpha$ to $\beta$.

\begin{proposition}\label{prop:dbae}
Assume that $\alpha$ and $\beta$ are not contractible 
and that $\alpha$ is not isotopic to $\beta$.  
Then each component of $\Om_{\alpha,\beta}$ 
is simply connected and hence
$
H^1(\Om_{\alpha,\beta};\R)=0.
$
\end{proposition}

The proof was explained to us by David 
Epstein~\cite{DBAE}.\phantomsection\label{DBAE2}
It is based on the following three lemmas. 
We identify $S^1\cong\R/\Z$.

\begin{lemma}\label{le:dbae1}
Let $\gamma:S^1\to\Sigma$ be a noncontractible loop 
and denote by
$$
\pi:\tSi\to\Sigma
$$
the covering generated by $\gamma$.
Then $\tSi$ is diffeomorphic to the cylinder.
\end{lemma}

\begin{proof}
By hypothesis, $\Sigma$ is oriented and has a nontrivial
fundamental group.  By the uniformization theorem,
choose a metric of constant curvature.
Then the universal cover of $\Sigma$ is isometric
to either $\R^2$ with the flat metric or
to the upper half space $\HH^2$ with the
hyperbolic metric.
The $2$-manifold $\tSi$
is a quotient of the universal cover of
$\Sigma$ by the subgroup of the group
of covering transformations generated by
a single element (a translation in the
case of $\R^2$ and a hyperbolic element
of $\PSL(2,\R)$ in the case of $\HH^2$).
Since $\gamma$ is not contractible, this element
is not the identity.  Hence $\tSi$
is diffeomorphic to the cylinder.
\end{proof}

\begin{lemma}\label{le:dbae2}
Let $\gamma:S^1\to\Sigma$ be a noncontractible loop 
and, for $k\in\Z$, define $\gamma^k:S^1\to\Sigma$
by 
$$
\gamma^k(s):=\gamma(ks).
$$
Then $\gamma^k$ is contractible if and only if $k=0$.
\end{lemma}

\begin{proof}
Let $\pi:\tSi\to\Sigma$ be as in Lemma~\ref{le:dbae1}.
Then, for $k\ne0$, the loop $\gamma^k:S^1\to\Sigma$ lifts to
a noncontractible loop in $\tSi$.
\end{proof}

\begin{lemma}\label{le:dbae3}
Let $\gamma_0,\gamma_1:S^1\to\Sigma$ be noncontractible 
embedded loops and suppose that $k_0,k_1$ 
are nonzero integers such that $\gamma_0^{k_0}$ 
is homotopic to $\gamma_1^{k_1}$.  
Then either $\gamma_1$ is homotopic to $\gamma_0$ and $k_1=k_0$
or $\gamma_1$ is homotopic to ${\gamma_0}^{-1}$ and $k_1=-k_0$.
\end{lemma}

\begin{proof}
Let $\pi:\tSi\to\Sigma$ be the covering generated by $\gamma_0$. 
Then ${\gamma_0}^{k_0}$ lifts to a closed curve in $\tSi$ 
and is homotopic to ${\gamma_1}^{k_1}$.
Hence ${\gamma_1}^{k_1}$ lifts
to a closed immersed curve in $\tSi$.
Hence there exists a nonzero integer $j_1$
such that ${\gamma_1}^{j_1}$ lifts
to an embedding $S^1\to\tSi$.
Any embedded curve in the cylinder is either
contractible or is homotopic to a generator.
If the lift of ${\gamma_1}^{j_1}$ were contractible
it would follow that ${\gamma_0}^{k_0}$ is contractible,
hence, by Lemma~\ref{le:dbae2}, $k_0=0$ 
in contradiction to our hypothesis.
Hence the lift of ${\gamma_1}^{j_1}$
to $\tSi$ is not contractible.
With an appropriate sign of $j_1$ it follows that
the lift of ${\gamma_1}^{j_1}$ is homotopic
to the lift of $\gamma_0$.
Interchanging the roles of $\gamma_0$ and $\gamma_1$,
we find that there exist nonzero integers
$j_0,j_1$ such that
$$
\gamma_0\sim{\gamma_1}^{j_1},\qquad
\gamma_1\sim{\gamma_0}^{j_0}
$$
in $\tSi$.  Hence $\gamma_0$ is homotopic to 
${\gamma_0}^{j_0j_1}$ in the free loop space of $\tSi$.
Since the homotopy lifts to the cylinder $\tSi$
and the fundamental group of $\tSi$
is abelian, it follows that 
$$
j_0j_1=1.
$$
If $j_0=j_1=1$ then $\gamma_1$ is homotopic to $\gamma_0$,
hence $\gamma_0^{k_1}$ is homotopic to ${\gamma_0}^{k_0}$,
hence ${\gamma_0}^{k_0-k_1}$ is contractible, 
and hence $k_0-k_1=0$, by Lemma~\ref{le:dbae2}.
If $j_0=j_1=-1$ then $\gamma_1$ is homotopic to ${\gamma_0}^{-1}$,
hence $\gamma_0^{-k_1}$ is homotopic to ${\gamma_0}^{k_0}$,
hence ${\gamma_0}^{k_0+k_1}$ is contractible, and hence $k_0+k_1=0$,
by Lemma~\ref{le:dbae2}.  This proves Lemma~\ref{le:dbae3}.
\end{proof}

\begin{proof}[Proof of Proposition~\ref{prop:dbae}] 
\phantomsection\label{proof:dbae}
Orient $\alpha$ and $\beta$ and
and choose orientation preserving
diffeomorphisms 
$$
\gamma_0:S^1\to\alpha,\qquad
\gamma_1:S^1\to\beta.
$$
A closed loop in $\Om_{\alpha,\beta}$ gives rise
to a map $u:S^1\times[0,1]\to\Sigma$ such that
$$
u(S^1\times\{0\})\subset\alpha,\qquad
u(S^1\times\{1\})\subset\beta.
$$
Let $k_0$ denote the degree of $u(\cdot,0):S^1\to\alpha$
and $k_1$ denote the degree of $u(\cdot,1):S^1\to\beta$.
Since the homotopy class of a map $S^1\to\alpha$
or a map $S^1\to\beta$ is determined by the degree
we may assume, without loss of generality, that
$$
u(s,0) = \gamma_0(k_0s),\qquad
u(s,1) = \gamma_1(k_1s).
$$
If one of the integers $k_0,k_1$ vanishes, so does the other,
by Lemma~\ref{le:dbae2}.  If they are both nonzero then 
$\gamma_1$ is homotopic to either $\gamma_0$ 
or $\gamma_0^{-1}$, by Lemma~\ref{le:dbae3}.
Hence $\gamma_1$ is isotopic to either $\gamma_0$ 
or $\gamma_0^{-1}$, by~\cite[Theorem~4.1]{EPSTEIN}.
\phantomsection\label{EPSTEIN2}
Hence $\alpha$ is isotopic to $\beta$, in contradiction
to our hypothesis.  This shows that
$$
k_0=k_1=0.
$$
With this established it follows that the map
$
u:S^1\times[0,1]\to\Sigma
$
factors through a map $v:S^2\to\Sigma$
that maps the south pole to $\alpha$ and the north pole
to $\beta$.  Since $\pi_2(\Sigma)=0$ it follows
that $v$ is homotopic, via maps with fixed north and
south pole, to one of its meridians.
This proves Proposition~\ref{prop:dbae}.
\end{proof}


\section{Diffeomorphisms of the Half Disc}
\label{app:diffeos_of_D}

\begin{proposition}\label{prop:disc}
The group of orientation preserving diffeomorphisms
$\phi:\D\to\D$ that satisfy $\phi(1)=1$ and $\phi(-1)=-1$
is connected.
\end{proposition}

\begin{proof} 
Choose $\phi$ as in the proposition.
We prove in five steps that $\phi$ is isotopic to the identity.

\medskip
\noindent{\bf Step~1.}
{\it
We may assume that $d\phi(-1)=d\phi(1)=\one$.
}

\medskip
\noindent
The differential of $\phi$ at $-1$ has the form
$$
     d\phi(-1)
     = \left(\begin{array}{cc}
        a & 0 \\ 0 & b
       \end{array}\right).
$$
Let $X:\D\to\R^2$ be a vector field on $\D$
that is tangent to the boundary, is supported
in an $\eps$-neighborhood of $-1$,
and satisfies
$$
     dX(-1)
     = \left(\begin{array}{cc}
       \log a & 0 \\ 0 & \log b
       \end{array}\right).
$$
Denote by $\psi_t:\D\to\D$ the flow of $X$.
Then $d\psi_1(-1)=d\phi(-1)$.
Replace $\phi$ by $\phi\circ{\psi_1}^{-1}$.

\smallbreak

\medskip
\noindent{\bf Step~2.}
{\it
We may assume that $\phi$ is equal to the identity map
near $\pm1$.
}

\medskip
\noindent
Choose local coordinates near $-1$ that identify
a neighborhood of $-1$ with a neighborhood
of zero in the right upper quadrant $Q$.
This gives rise to a local diffeomorphism
$\psi:Q\to Q$ such that $\psi(0)=0$.
Choose a smooth cutoff function
$\rho:[0,\infty)\to[0,1]$
such that
$$
     \rho(r)
     = \left\{\begin{array}{rl}
       1,&\mbox{for }r\le1/2,\\
       0,&\mbox{for }r\ge 1,
       \end{array}\right.
$$
For $0\le t\le 1$ define $\psi_t:Q\to Q$ by
$$
     \psi_t(z)
     := \psi(z) + t\rho(|z|^2/\eps^2))(z-\psi(z)).
$$
Since $d\psi(0)=\one$ this map is a diffeomorphism
for every $t\in[0,1]$ provided that $\eps>0$
is sufficiently small.  Moreover,
$\psi_t(z)=\psi(z)$ for $|z|\ge\eps$,
$\psi_0=\psi$, and $\psi_1(z)=z$ for $|z|\le\eps/2$.

\medskip
\noindent{\bf Step~3.}
{\it
We may assume that $\phi$ is equal to the
identity map near $\pm 1$ and on $\p\D$.
}

\medskip
\noindent
Define $\tau:[0,\pi]\to[0,\pi]$ by
$$
     \phi(e^{i\theta})=e^{i\tau(\theta)}.
$$
Let $X_t:\D\to\R^2$ be a vector field that is
equal to zero near $\pm1$ and satisfies
$$
     X_t(z+t(\phi(z)-z)) = \phi(z)-z
$$
for $z\in\D\cap\R$ and
$$
     X_t(z)
     = i(\tau(\theta)-\theta)z,\qquad
     z = e^{i(\theta+t(\tau(\theta)-\theta))}.
$$
Let $\psi_t:\D\to\D$ be the isotopy generated by
$X_t$ via
$
     \p_t\psi_t = X_t\circ\psi_t
$
and
$
     \psi_0=\id.
$
Then $\psi_1$ agrees with $\psi$ on $\p\D$
and is equal to the identity near $\pm 1$.
Replace $\phi$ by $\phi\circ{\psi_1}^{-1}$.

\medskip
\noindent{\bf Step~4.}
{\it
We may assume that $\phi$ is equal to the identity map
near $\p\D$.
}

\medskip
\noindent
Write
$$
     \phi(x+iy) = u(x,y)+iv(x,y).
$$
Then
$$
     u(x,0)=x,\qquad \p_yv(x,0) = a(x).
$$
Choose a cutoff function $\rho$ equal to one near
zero and equal to zero near one.  Define
$$
     \phi_t(x,y) := u_t(x,y) + v_t(x,y)
$$
where
$$
     u_t(x,y)
     := u(x,y) + t\rho(y/\eps)(x-u(x,y))
$$
and
$$
     v_t(x,y)
     := v(x,y) + t\rho(y/\eps)(a(x)y-v(x,y)).
$$
If $\eps>0$ is sufficiently small then
$\phi_t$ is a diffeomorphism for every $t\in[0,1]$.
Moreover, $\phi_0=\phi$ and $\phi_1$ satisfies
$$
     \phi_1(x+iy) = x + ia(x)y
$$
for $y\ge 0$ sufficiently small.
Now choose a smooth family of vector fields
$X_t:\D\to\D$ that vanish on the boundary
and near $\pm1$ and satisfy
$$
     X_t(x+i(y+t(a(x)y-y))) = i(a(x)y-y)
$$
near the real axis. Then the isotopy $\psi_t$
generated by $X_t$ satisfies
$\psi_t(x+iy)=x+iy + it(a(x)y-y)$
for $y$ sufficiently small.
Hence $\psi_1$ agrees with $\phi_1$ near
the real axis. Hence $\phi\circ{\psi_1}^{-1}$
has the required form near $\D\cap\R$.
A similar isotopy near $\D\cap S^1$
proves Step~4.

\medskip
\noindent{\bf Step~5.}
{\it
We prove the proposition.
}

\medskip
\noindent
Choose a continuous map $f:\D\to S^2=\C\cup\{\infty\}$
such that $f(\p\D)=\{0\}$ and $f$
restricts to a diffeomorphism
from $\D\setminus\p\D$ to $S^2\setminus\{0\}$.
Define $\psi:S^2\to S^2$ by
$$
     f\circ\psi = \phi\circ f.
$$
Then $\psi$ is a diffeomorphism, equal to the identity near the origin.
By a well known Theorem of 
Smale~\cite{SMALE2}\phantomsection\label{SMALE2}
(see also~\cite{EE} and~\cite{GRS-HMS})
\phantomsection\label{EE}\label{GRS-HMS}
$\psi$ is isotopic to the identity.
Compose with a path in $\SO(3)$
which starts and ends at the identity
to obtain an isotopy $\psi_t:S^2\to S^2$
such that $\psi_t(0)=0$. Let
$$
    \Psi_t:=d\psi_t(0),\qquad
    U_t:=\Psi_t({\Psi_t}^T\Psi_t)^{-1/2}.
$$
Then $U_t\in\SO(2)$ and $U_0=U_1=\one$.
Replacing $\psi_t$ by ${U_t}^{-1}\psi_t$ we may assume
that $U_t=\one$ and hence $\Psi_t$
is positive definite for every $t$.
Hence there exists a smooth path
$[0,1]\to\R^{2\times 2}:t\mapsto A_t$ such that
$$
    e^{A_t}=\Psi_t,\qquad
    A_0 = A_1 = 0.
$$
Choose a smooth family of compactly supported
vector fields $X_t:\C\to\C$ such that
$$
    dX_t(0) = A_t,\qquad
    X_0 = X_1 = 0.
$$
For every fixed $t$ let $\chi_t:S^2\to S^2$ be the time-$1$
map of the flow of $X_t$.  Then
$$
    \chi_t(0)=0,\qquad d\chi_t(0)=\Psi_t,\qquad
    \chi_0=\chi_1=\id.
$$
Hence the diffeomorphisms
$$
    \psi'_t := \psi_t\circ{\chi_t}^{-1}
$$
form an isotopy from $\psi_0'=\id$ to $\psi_1'=\psi$
such that $\psi_t'(0)=0$ and $d\psi_t'(0)=\one$
for every $t$. Now let $\rho:\R\to[0,1]$
be a smooth cutoff function that is equal to
one near zero and equal to zero near one. Define
$$
    \psi''_t(z):= \psi'_t(z) + \rho(|z|/\eps)(z-\psi'_t(z)).
$$
For $\eps>0$ sufficiently small this is an isotopy
from $\psi''_0=\id$ to $\psi''_1=\psi$
such that $\psi_t=\id$ near zero for every $t$.
The required isotopy $\phi_t:\D\to\D$ is now
given by
$
    f\circ\psi_t = \phi_t\circ f.
$
This proves Proposition~\ref{prop:disc}.
\end{proof}


\section{Homological Algebra}\label{app:homological}

Let $P$ be a finite set and $\nu:P\times P\to\Z$ be a function
that satisfies
\begin{equation}\label{eq:nu-CF}
     \sum_{q\in P}\nu(r,q)\nu(q,p) = 0
\end{equation}
for all $p,r\in P$.  Any such function
determines a chain complex $\p:C\to C$,
where $C=C(P)$ and $\p=\p^\nu$ are
defined by
$$
     C := \bigoplus_{p\in P}\Z p,\qquad
     \p q := \sum_{p\in P}\nu(q,p)p
$$
for $q\in P$. Throughout we fix two
elements $\bar p,\bar q\in P$ such that
$
     \nu(\bar q,\bar p) = 1.
$
Consider the set
\begin{equation}\label{eq:P'}
     P' := P \setminus \{\bar p,\bar q\}
\end{equation}
and the function $\nu':P'\times P'\to\Z$ defined by
\begin{equation}\label{eq:nu'}
     \nu'(q,p) := \nu(q,p) - \nu(q,\bar p)\nu(\bar q,p)
\end{equation}
for $p,q\in P$ and denote
$C':= C(P')$ and $\p':=\p^{\nu'}$.
The following lemma is due to 
Floer~\cite{FLOER2}.\phantomsection\label{FLOERc}

\begin{lemma}[{\bf Floer}]\label{le:floer} 
\index{Floer chain complex}
The endomorphism $\p':C'\to C'$ is a chain complex
and its homology $H(C',\p')$ is isomorphic to $H(C,\p)$.
\end{lemma}

\begin{proof} The proof consists of four steps.

\medskip
\noindent{\bf Step~1.} $\p'\circ\p'=0$.

\medskip
\noindent
Let $r\in P'$.  Then $\p'\circ\p'r=\sum_{p\in P'}\mu'(r,p)p$ where
$\mu'(r,p)\in\Z$ is given by
\begin{eqnarray*}
     \mu'(r,p)
&= &
     \sum_{q\in P'}\nu'(r,q)\nu'(q,p) \\
&= &
     \sum_{q\in P}\left(\nu(r,q)-\nu(r,\bar p)\nu(\bar q,q)\right)
             \left(\nu(q,p)-\nu(q,\bar p)\nu(\bar q,p)\right) \\
&= &
     0
\end{eqnarray*}
for $p\in P'$.  Here the first equation follows
from the fact that $\nu(\bar q,\bar p)=1$ and
the last equation follows from the fact
that $\p\circ\p=0$.

\medskip
\noindent{\bf Step~2.}
{\it
The operator $\Phi:C'\to C$ defined by
\begin{equation}\label{eq:phi}
    \Phi q := q - \nu(q,\bar p)\bar q
\end{equation}
for $q\in P'$ is a chain map, i.e.
$$
    \Phi\circ\p' = \p\circ\Phi.
$$
}

\medskip
\noindent
For $q\in P'$ we have
\begin{eqnarray*}
    \Phi\p'q
&= &
    \sum_{p\in P'}\nu'(q,p)\Phi p \\
&= &
    \sum_{p\in P'}
    \left(
    \nu(q,p) - \nu(q,\bar p)\nu(\bar q,p)
    \right)
    \left(
    p-\nu(p,\bar p)\bar q
    \right)  \\
&= &
    \sum_{p\in P}
    \left(
    \nu(q,p) - \nu(q,\bar p)\nu(\bar q,p)
    \right)
    \left(
    p-\nu(p,\bar p)\bar q
    \right)  \\
&= &
    \sum_{p\in P}
    \left(
    \nu(q,p)-\nu(q,\bar p)\nu(\bar q,p)
    \right)p \\
&= &
    \p q - \nu(q,\bar p)\p\bar q \\
&= &
    \p\Phi q.
\end{eqnarray*}

\medskip
\noindent{\bf Step~3.}
{\it
The operator $\Psi:C\to C'$ defined by
$\Psi q=q$ for $q\in P'$ and
\begin{equation}\label{eq:psi}
    \Psi \bar q := 0,\qquad
    \Psi \bar p := - \sum_{p\in P'}\nu(\bar q,p)p
\end{equation}
is a chain map, i.e.
$$
    \p'\circ\Psi = \Psi\circ\p.
$$
}

\medskip
\noindent
For $q\in P'$ we have
\begin{eqnarray*}
    \Psi\p q
&= &
    \sum_{p\in P}\nu(q,p)\Psi p \\
&= &
    \sum_{p\in P'}\nu(q,p)p
    + \nu(q,\bar p)\Psi\bar p \\
&= &
    \sum_{p\in P'}
    \left(
    \nu(q,p) - \nu(q,\bar p)\nu(\bar q,p)
    \right)p  \\
&= &
    \p'q.
\end{eqnarray*}
Moreover,
$$
    \Psi\p\bar q
    = \sum_{p\in P}\nu(\bar q,p)\Psi p
    = \sum_{p\in P'}\nu(\bar q,p)p
      + \Psi \bar p
    = 0
    = \p'\Psi\bar q,
$$
and
\begin{eqnarray*}
    \p'\Psi\bar p
&= &
    -\sum_{q\in P'}\nu(\bar q,q)\p'q \\
&= &
    -\sum_{q\in P'}\sum_{p\in P'}
    \nu(\bar q,q)
    \left(
    \nu(q,p)-\nu(q,\bar p)\nu(\bar q,p)
    \right)p \\
&= &
    \sum_{p\in P'}
    \left(
    \nu(\bar p,p)-\nu(\bar p,\bar p)\nu(\bar q,p)
    \right)p  \\
&= &
    \sum_{p\in P'}\nu(\bar p,p)p
    + \nu(\bar p,\bar p)\Psi\bar p  \\
&= &
    \sum_{p\in P}\nu(\bar p,p)\Psi p  \\
&= &
    \Psi\p\bar p.
\end{eqnarray*}

\medskip
\noindent{\bf Step~4.}
{\it
The operator $\Psi\circ\Phi:C'\to C'$ is equal
to the identity and
$$
     \id - \Phi\circ\Psi = \p\circ T+T\circ\p,
$$
where $T:C\to C$ is defined by $T\bar p=\bar q$
and $Tq=0$ for $q\in P\setminus\{\bar p\}$.}

\bigbreak

\medskip
\noindent
For $q\in P'$ we have
$$
     \Phi\Psi q = \Phi q = q-\nu(q,\bar p)\bar q
$$
and hence
$$
     q - \Phi\Psi q
     = \nu(q,\bar p)\bar q
     = \nu(q,\bar p)T\bar p
     = T\p q
     = T\p q+\p Tq.
$$
Moreover,
$$
    \bar q - \Phi\Psi\bar q
    = \bar q
    = \nu(\bar q,\bar p)T\bar p
    = T\p\bar q
    = T\p\bar q+\p T\bar q
$$
and
\begin{eqnarray*}
    \bar p - \Phi\Psi\bar p
&= &
    \bar p + \sum_{p\in P'}\nu(\bar q,p)\Phi p \\
&= &
    \bar p + \sum_{p\in P'}\nu(\bar q,p)p
    - \sum_{p\in P'}\nu(\bar q,p)\nu(p,\bar p)\bar q \\
&= &
    \bar p + \sum_{p\in P'}\nu(\bar q,p)p
    + \nu(\bar q,\bar q)\bar q
    + \nu(\bar p,\bar p)\bar q \\
&= &
    \p\bar q
    + \nu(\bar p,\bar p)\bar q \\
&= &
    \p T\bar p + T\p\bar p.
\end{eqnarray*}
This proves Lemma~\ref{le:floer}.
\end{proof}

Now let $(P,\preceq)$ be a finite poset.
An ordered pair $(p,q)\in P\times P$ is called
{\bf adjacent} if $p\preceq q$, $p\ne q$, 
and \index{adjacent elements of a poset}
$$
     p\preceq r\preceq q
     \qquad\implies\qquad
     r\in\{p,q\}.
$$
Fix an adjacent pair $(\bar p,\bar q)\in P\times P$
and consider the relation $\preceq'$
on 
$$
P':=P\setminus\{\bar p,\bar q\}
$$
defined by
\begin{equation}\label{eq:po}
     p\preceq' q
     \qquad\iff\qquad
     \left\{\begin{array}{ll}
     \mbox{either}& p\preceq q, \\
     \mbox{or}& \bar p\preceq q\mbox{ and }p\preceq \bar q.
     \end{array}\right.
\end{equation}

\begin{lemma}\label{le:po}
$(P',\preceq')$ is a poset.
\end{lemma}

\begin{proof}
We prove that the relation $\preceq'$ is transitive.
Let $p,q,r\in P'$ such that 
$$
p\preceq' q,\qquad q\preceq' r.
$$
There are four cases.  

\smallskip\noindent{\bf Case~1:}
{\it $p\preceq q$ and $q\preceq r$.}
Then $p\preceq r$ and hence $p\preceq' r$.

\smallskip\noindent{\bf Case~2:}
{\it $p\not\preceq q$ and $q\preceq r$.}
Then
$$
\bar p\preceq q\preceq r,\qquad
p\preceq\bar q,
$$
and hence $p\preceq'r$.

\smallskip\noindent{\bf Case~3:}
{\it $p\preceq q$ and $q\not\preceq r$.}
The argument is as in the Case~2.

\smallskip\noindent{\bf Case~4:}
{\it $p\not\preceq q$ and $q\not\preceq r$.}
Then
$$
p\preceq\bar q,\qquad
\bar p\preceq r,
$$
and hence $p\preceq' r$.

Next we prove that the relation $\preceq'$ is
anti-symmetric. Hence assume that $p,q\in P'$
such that $p\preceq'q$ and $q\preceq'p$.
We claim that $p\preceq q$ and $q\preceq p$.
Assume otherwise that $p\not\preceq q$.
Then $\bar p\preceq q$ and $p\preceq\bar q$.
Since $q\preceq' p$, it follows that
$\bar p\preceq p\preceq\bar q$ and
$\bar p\preceq q\preceq\bar q$, and hence
$\{p,q\}\subset\{\bar p,\bar q\}$, a contradiction.  
Thus we have shown that $p\preceq q$.  
Similarly, $q\preceq p$ and hence $p=q$. 
This proves Lemma~\ref{le:po}
\end{proof}

A function $\mu:P\to\Z$ is called an {\bf index function} 
for $(P,\preceq)$ if \index{index function on a poset}
\begin{equation}\label{eq:Pmu}
     p\preceq q
     \qquad\implies\qquad
     \mu(p) < \mu(q).
\end{equation}
Let $\mu$ be an index function for $P$.
A function 
$$
\nu:P\times P\to\Z
$$ 
is called a {\bf connection matrix} \index{connection matrix} 
for $(P,\preceq,\mu)$ if it satisfies~(\ref{eq:nu-CF}) and
\begin{equation}\label{eq:numu}
     \nu(q,p)\ne 0\qquad\implies\qquad
     \mu(q)-\mu(p)=1,\;\; p\preceq q
\end{equation}
for $p,q\in P$.

\begin{lemma}\label{le:connection}
If $\mu:P\to\Z$ is an index function
for $(P,\preceq)$ then $\mu':=\mu|_{P'}$
is an index function for $(P',\preceq')$.
Moreover, if $\nu$ is a connection matrix for
$(P,\preceq,\mu)$ and $\nu(\bar q,\bar p)=1$
then $\nu'$ is a connection matrix for
$(P',\preceq',\mu')$.
\end{lemma}

\begin{proof}
We prove that $\mu'$ is an
index function for $(P',\preceq')$.
Let $p',q'\in P'$ such that $p\preceq' q$.
If $p\preceq q$ then $\mu(p)<\mu(q)$,
since $\mu$ is an index function for $(P,\preceq)$.
If $p\not\preceq q$ then $p\preceq\bar q$
and $\bar p\preceq q$, and hence
$$
     \mu(p) < \mu(\bar q) = \mu(\bar p) + 1 \le \mu(q).
$$
Hence $\mu'$ satisfies~(\ref{eq:Pmu}), as claimed.
Next we prove that $\nu'$ is a connection matrix
for $(P',\preceq',\mu')$. By Lemma~\ref{le:floer},
$\nu'$ satisfies~(\ref{eq:nu-CF}). We prove that
it satisfies~(\ref{eq:numu}).  Let
$p,q\in P'$ such that $\nu'(q,p)\ne 0$.
If $\nu(q,p)\ne 0$ then, since $\nu$ is a connection
matrix for $(P,\preceq,\mu)$, we have
$\mu(q)-\mu(p)=1$ and $p\preceq'q$.
If $\nu(q,p)=0$ then in follows from the definition
of $\nu'$ that $\nu(q,\bar p)\ne 0$ and $\nu(\bar q,p)\ne 0$.
Hence
$$
     \mu(q) - \mu(\bar p) = 1,\qquad
     \mu(\bar q) - \mu(p) = 1,\qquad
     \mu(\bar q) - \mu(\bar p) = 1,
$$
and hence
$$
     \bar p\preceq q,\qquad p\preceq\bar q.
$$
It follows again that $\mu(q)-\mu(p)=1$ and $p\preceq' q$.
Hence $\nu'$ satisfies~(\ref{eq:numu}), as claimed.
This proves Lemma~\ref{le:connection}.
\end{proof}


\section{Asymptotic behavior of holomorphic strips}\label{app:asymptotic}

This appendix deals with the asymptotic behaviour 
of pseudoholomorphic strips in symplectic manifolds
that satisfy Lagrangian boundary conditions. 
More precisely, let $(M,\om)$ be a symplectic
manifold and 
$$
L_0,L_1\subset M
$$
be closed (not necessarily compact)
Lagrangian submanifolds that intersect transversally. 
Fix a $t$-de\-pen\-dent family of $\om$-tame almost complex 
structures $J_t$ on $M$.  
We consider smooth maps $v:\S\to M$ 
that satisfy the boundary value problem
\begin{equation}\label{eq:FLOER}
\p_sv+J_t(v)\p_tv = 0,\qquad
v(\R)\subset L_0,\qquad
v(\R+\i)\subset L_1.
\end{equation}

\begin{theorem}\label{thm:asymptotic}
Let $v:\S\to M$ be a solution of~\eqref{eq:FLOER}.
Assume $v$ has finite energy
$$
E(v) := \int_\S v^*\om < \infty
$$
and that the image of $v$ has compact closure. 
Then the following hold.

\smallskip\noindent{\bf (i)}
There exist intersection points $x,y\in L_0\cap L_1$ such that
\begin{equation}\label{eq:LIMITS}
x = \lim_{s\to-\infty}v(s,t),\qquad
y = \lim_{s\to+\infty}v(s,t),
\end{equation}
where the convergence is uniform in $t$.  Moreover, $\p_sv$ decays 
exponentially in the $\Cinf$ topology, i.e.\ there are positive 
constants $\delta$ and $c_1,c_2,c_3,\dots$ such that 
$\Norm{\p_su}_{C^k([s-1,s+1]\times[0,1])}\le c_ke^{-\delta\abs{s}}$
for all $s$ and $k$.

\smallskip\noindent{\bf (ii)}
Assume $v$ is nonconstant. Then there exist positive real numbers
$\lambda_x,\lambda_y$ and smooth paths $\eta_x:[0,1]\to T_xM$
and $\eta_y:[0,1]\to T_yM$ satisfying 
\begin{equation}\label{eq:etaxy}
\begin{split}
&J_t(x)\p_t\eta_x(t)+\lambda_x\eta_x(t)=0,\qquad
\eta_x(0)\in T_xL_0,\qquad \eta_x(1)\in T_xL_1,\\
&J_t(y)\p_t\eta_y(t)-\lambda_y\eta_y(t)=0,\qquad
\eta_y(0)\in T_yL_0,\qquad \eta_y(1)\in T_yL_1,
\end{split}
\end{equation}
and 
\begin{equation}\label{eq:etaxy1}
\eta_x(t) = \lim_{s\to-\infty}e^{-\lambda_xs}\p_sv(s,t),\qquad
\eta_y(t) = \lim_{s\to+\infty}e^{\lambda_ys}\p_sv(s,t),
\end{equation}
where the convergence is uniform in $t$.

\smallskip\noindent{\bf (iii)}
Assume $v$ is nonconstant and let $\lambda_x,\lambda_y>0$
and $\eta_x(t)\in T_xM$ and $\eta_y(t)\in T_yM$ be as in~(ii).
Then there exists a constant $\delta>0$ such that
\begin{equation}\label{eq:etaxx}
\begin{split}
&v(s,t) = \exp_x\left(
\frac{1}{\lambda_x}e^{\lambda_xs}\eta_x(t)
+ R_x(s,t)\right),\\
&\lim_{s\to-\infty}e^{-(\lambda_x+\delta)s}\sup_t\abs{R_x(s,t)} = 0,
\end{split}
\end{equation}
and
\begin{equation}\label{eq:etayy}
\begin{split}
&v(s,t) = \exp_y\left(
- \frac{1}{\lambda_y}e^{-\lambda_ys}\eta_y(t)
+ R_y(s,t)\right),\\
&\lim_{s\to+\infty}e^{(\lambda_y+\delta)s}\sup_t\abs{R_y(s,t)} = 0.
\end{split}
\end{equation}
\end{theorem}

As a warmup the reader is encouraged to verify the signs in~(ii) and~(iii)
in the linear setting where $M$ is a symplectic vector space, 
$L_0,L_1$ are transverse Lagrangian subspaces, $J_t=J$ 
is a constant $\om$-tame complex structure, and 
$v(s,t)=\lambda_x^{-1}e^{\lambda_xs}\eta_x(t)$, respectively
$v(s,t)=-\lambda_y^{-1}e^{-\lambda_ys}\eta_y(t)$.

\bigbreak

\begin{proof}[Proof of Theorem~\ref{thm:asymptotic}]
Assertion~(i) is standard (see e.g.~\cite[Theorem~A]{ASYMPTOTIC}).
Assertions~(ii) and~(iii) are proved in~\cite[Theorem~B]{ASYMPTOTIC}
for $\om$-compatible\phantomsection\label{ASYMPTOTIC6}
almost complex structures $J_t$.  The $\om$-tame case is treated 
in~\cite{SIMCEVIC}.\phantomsection\label{SIMCEVIC}
\end{proof}

The next corollary summarizes the consequences of 
Theorem~\ref{thm:asymptotic} in the special case of dimension two.
Assume $\Sigma$ and $\alpha,\beta\subset\Sigma$
satisfy~\hyperlink{property_H}{(H)} and fix a complex structure $J$
on $\Sigma$. For $x\in\alpha\cap\beta$ denote by $\theta_x\in(0,\pi)$ 
the angle from $T_x\alpha$ to $T_x\beta$ with respect to our complex 
structure $J$ so that 
$$
T_x\beta = \left(\cos(\theta_x)+\sin(\theta_x)J\right)T_x\alpha,\qquad
0<\theta_x<\pi.
$$
(See equation~\eqref{eq:thetax}.)
Fix two intersection points $x,y\in\alpha\cap\beta$
and two holomorphic coordinate charts $\psi_x:U_x\to\C$
and $\psi_y:U_y\to\C$ defined on neighborhoods
$U_x\subset \Sigma$ of $x$ and $U_y\subset\Sigma$ of $y$ 
such that $\psi_x(x)=0$ and $\psi_y(y)=0$.  

\begin{corollary}\label{cor:asymptotic}
Let $v:\S\to\Sigma$ be a nonconstant holomorphic 
$(\alpha,\beta)$-strip from $x$ to $y$.  Thus $v$ is has 
finite energy and satisfies the boundary conditions 
$v(\R)\subset\alpha$ and $v(\R+\i)\subset\beta$ 
and the endpoint conditions
$$
\lim_{s\to-\infty}v(s+\i t) = x,\qquad 
\lim_{s\to+\infty}v(s+\i t) = y.
$$
(See equation~\eqref{eq:limit}.)
Then there exist nonzero complex numbers
$$
c_x\in T_0\psi_x(U_x\cap\alpha),\qquad
c_y\in T_0\psi_y(U_y\cap\alpha)
$$
and integers $\nu_x\ge 0$ and $\nu_y\ge1$ such that
the following holds.

\smallskip\noindent{\bf (i)}
Define 
\begin{equation}\label{eq:lambdaxy}
\lambda_x:=\nu_x\pi+\theta_x,\qquad 
\lambda_y:=\nu_y\pi-\theta_y.
\end{equation}
Then 
\begin{equation}\label{eq:claxy}
\begin{split}
\lambda_xc_x &= \lim_{s\to-\infty}e^{-\lambda_x(s+\i t)}d(\psi_x\circ v)(s+\i t),\\
-\lambda_yc_y &= \lim_{s\to+\infty}e^{\lambda_y(s+\i t)}d(\psi_y\circ v)(s+\i t),
\end{split}
\end{equation}
where the convergence is uniform in $t$.

\smallskip\noindent{\bf (ii)}
Let $\lambda_x,\lambda_y>0$ be given by~\eqref{eq:lambdaxy}.
There exists a constant $\delta>0$ such that
\begin{equation}\label{eq:nuxv}
\psi_x(v(s+\i t)) = c_xe^{\lambda_x(s+\i t)} 
+ O(e^{(\lambda_x+\delta)s})
\end{equation}
for $s<0$ close to $-\infty$ and
\begin{equation}\label{eq:nuyv}
\psi_y(v(s+\i t)) = c_ye^{-\lambda_y(s+\i t)} 
+ O(e^{-(\lambda_y+\delta)s})
\end{equation}
for $s>0$ close to $+\infty$.
\end{corollary}

\begin{proof}
Define $\eta_x:[0,1]\to T_x\Sigma$ and $\eta_y:[0,1]\to T_y\Sigma$ by
$$
d\psi_x(x)\eta_x(t) := \lambda_xc_xe^{\i\lambda_xt},\qquad
d\psi_y(y)\eta_y(t) := -\lambda_yc_ye^{-\i\lambda_yt}.
$$
These functions satisfy the conditions in~\eqref{eq:etaxy}
if and only if $c_x\in T_0\psi_x(U_x\cap\alpha)$, 
$c_y\in T_0\psi_y(U_y\cap\alpha)$, and
$\lambda_x,\lambda_y$ are given by~\eqref{eq:lambdaxy}
with integers $\nu_x\ge 0$ and $\nu_y\ge1$.  
Moreover, the limit condition~\eqref{eq:claxy} is equivalent
to~\eqref{eq:etaxy1}.  Hence assertion~(i) in 
Corollary~\ref{cor:asymptotic} follows from~(ii) 
in Theorem~\ref{thm:asymptotic}.  With this understood,
assertion~(ii) in Corollary~\ref{cor:asymptotic} follows 
immediately from~(iii) in Theorem~\ref{thm:asymptotic}. 
This proves the corollary.  (For assertion~(ii) see 
also~\cite[Theorem~C]{ASYMPTOTIC}.)\phantomsection\label{ASYMPTOTIC7}
\end{proof}


\newpage
\addcontentsline{toc}{part}{Index}
\input{comb_floer_memoir.ind}

\end{document}